\documentclass[a4paper]{amsart}
\usepackage{bm}
\usepackage{amssymb}
\usepackage[polutonikogreek, english]{babel}
\usepackage{graphicx}
\usepackage{pstricks}
\usepackage{enumerate}
\theoremstyle{plain}
\newtheorem{theorem}{Theorem}[section]
\newtheorem{lemma}[theorem]{Lemma}
\newtheorem{corollary}[theorem]{Corollary}
\newtheorem{definition}{Definition}
\newtheorem{axiom}{Axiom}
\dedicatory{Est enim per contrapositionem conversio ut si dicas \\omnis homo animal est omne non animal non homo est.\\ \flushright \emph{Boethius}, \textit{de Syll. Cat.}}
\title{The  Principle of Open Induction on $[0,1]$ and the Approximate-Fan Theorem}
\author{Wim Veldman}
\address{Institute for Mathematics, Astrophysics and Particle Physics, Faculty of Science, Radboud University Nijmegen,
Postbus 9010, 6500 GL Nijmegen, the Netherlands}
\email{W.Veldman@science.ru.nl}
\date{}
\begin{document}

\begin{abstract}
 In the earlier papers \cite{veldman2009b}, \cite{veldman2011b} and \cite{veldman2011d} we collected statements that are, in a weak formal context,   equivalent to  Brouwer's Fan Theorem. This time, we do the same for the Principle of Open Induction on $[0,1]$ and the Approximate-Fan Theorem.
These principles follow from Brouwer's Thesis on Bars and imply the Fan Theorem. 
 \end{abstract}

\maketitle
\section{Introduction}
\subsection{} L.E.J.~Brouwer  wanted to do mathematics differently.\footnote{See \cite{veldman2021}.}
He restored the logical constants to their natural constructive meaning and  introduced some new principles or axioms. 
   His famous \textit{Continuity Theorem} claims that every function from the unit interval $[0,1]$ to the set $\mathcal{R}$ of the real numbers is uniformly continuous. This theorem states two things:
  a function  from  $[0,1]$ to  $\mathcal{R}$   is 
  continuous everywhere, and  a function from $[0,1]$ to $\mathcal{R}$ that is continuous everywhere is 
  continuous uniformly on $[0,1]$.
  The principle underlying the first statement now is called {\it   Brouwer's Continuity
  Principle} and the principle underlying the second statement is 
  {\it Brouwer's Fan Theorem}.
  
   The Fan Theorem asserts that every thin bar in Cantor  space $ 2^\omega$ is finite. The theorem extends to every subset of Baire space $\omega^\omega$, that, like $2^\omega$, is a \textit{fan}.
Brouwer derived the Fan Theorem  from  his {\em Thesis on Bars in $\omega^\omega$} although  a Thesis on Bars in $2^\omega$ would have been sufficient for the purpose at hand.    Brouwer's Thesis on Bars in $\omega^\omega$ is   a much stronger statement than the Fan Theorem.

The  \textit{Principle of Open Induction on $2^\omega$} asserts that every open subset of $\omega^\omega$ that is progressive in $2^\omega$ under the lexicographical ordering contains $2^\omega$.
   T.~Coquand, (see \cite{coquand} and \cite{veldman2004b}, Section 11), saw that this principle follows from Brouwer's Thesis on Bars in $\omega^\omega$.  The principle is a contraposition of the (classical) fact that a non-empty closed subset of $2^\omega$ must have a smallest element under the lexicographical ordering. 
  
  The \textit{Approximate-Fan Theorem} asserts that every thin bar in an {\it approximate} fan is {\it almost-finite}. Every fan is an approximate fan, but not conversely, and every finite subset of the set $\omega$ of the natural numbers is almost-finite, but not conversely. 
  
  We shall show that, in a weak formal context, the Principle of Open Induction on $2^\omega$ is a consequence of the Approximate-Fan Theorem and that the Approximate-Fan Theorem follows from the Thesis on Bars in $\omega^\omega$. The Principle of Open Induction on $2^\omega$ implies the Fan Theorem but not conversely.  Both the Principle of Open Induction on $2^\omega$ and the Approximate-Fan Theorem  have many equivalents.

  The important work done in \textit{classical reverse mathematics} and beautifully described in \cite{Simpson} has been a source of inspiration.  Our results belong  to {\em intuitionistic reverse mathematics} and are an intuitionistic counterpart to the Sections III.1, III.2, III.7 and part of Section V.1 from \cite{Simpson}. 
 
 I want to thank the referee of the first version of this paper for his  (or her) careful reading, his encouraging and useful comments and his good advice. I also want to thank J.R.~Moschovakis for her unstinting support and interest.
\subsection{The contents of the paper}
Apart from this introductory Section 1, the paper contains  Sections 2-13.

In Section 2, we introduce  $\mathsf{BIM}$, a formal system  for \textit{Basic Intuitionistic Mathematics} also used
in Veldman \cite{veldman2009b}, \cite{veldman2011b} and \cite{veldman2011d}. 

In Section 3, we introduce $\mathbf{FT}$,  the \textit{Fan Theorem},  and  $\mathbf{HB}$, the \textit{Heine-Borel Theorem}. We prove that $\mathsf{BIM}\vdash \mathbf{FT}\leftrightarrow \mathbf{HB}.$

In Section 4, we introduce   $\mathbf{BI}$, Brouwer's \textit{Principle of Induction on Bars in  $\omega^\omega$} and    $\mathbf{OI}([0,1])$, the \textit{Principle of Open Induction on $[0,1]$}

  We prove that $\mathsf{BIM}\vdash \mathbf{BI}\rightarrow \mathbf{OI}([0,1])$ and that $\mathsf{BIM}\vdash \mathbf{OI}([0,1])\rightarrow \mathbf{HB}$.

In Section 5, we introduce $\mathbf{Ded}$, \textit{Dedekind's Theorem} which says that {\it every infinite bounded and nondecreasing sequence of real numbers is convergent}. We show that this statement is constructively false and introduce $\overleftarrow{\mathbf{Ded}}$, a contraposition of $\mathbf{Ded}$. We  introduce $\overleftarrow{\mathbf{Ded}_{\omega^\omega}}$, a statement similar tot $\overleftarrow{\mathbf{Ded}}$, intended for  $\omega^\omega$ rather than for $\mathcal{R}$.  
We also introduce $\mathbf{OI}(2^\omega)$, the  \textit{Principle of Open Induction on Cantor space $2^\omega$}. 
We prove that $\mathsf{BIM}\vdash \mathbf{OI}([0,1])\leftrightarrow \overleftarrow{\mathbf{Ded}} \leftrightarrow \overleftarrow{\mathbf{Ded}_{\omega^\omega}} \leftrightarrow\mathbf{OI}(2^\omega)$.

In Section 6 we introduce $\mathbf{EnDec?!}$, a positive statement with the negative consequence
  that \textit{no enumerable subset of $\omega$ positively fails to be a decidable subset of $\omega$}. We prove that $\mathsf{BIM}\vdash \mathbf{OI}([0,1])\leftrightarrow  \mathbf{EnDec?!}$.

In Section 7, we introduce the statement $\mathbf{WF}_{KB}$: \textit{ for every decidable subset $B$ of $\omega$ that is a bar in $\omega^\omega$, the set of all (codes of) finite sequences below the bar $B$ is well-founded under the Kleene-Brouwer-ordering}. We prove: $\mathsf{BIM}\vdash \mathbf{OI}([0,1]) \leftrightarrow \mathbf{WF}_{KB}$.  As a corollary we obtain the result that, in $\mathsf{BIM}$, $\mathbf{OI}([0,1])$ proves the principle of transfinite induction on the ordinal number $\varepsilon_0$, and we conclude, using a result of Troelstra's, that $\mathsf{BIM} + \mathbf{FT} \nvdash\mathbf{OI}([0,1])$.

 In Section 8, we introduce \textit{approximate fans} and the \textit{Approximate-Fan Theorem} $\mathbf{AppFT}$. We prove that $\mathsf{BIM}\vdash \mathbf{BI}\rightarrow \mathbf{AppFT}$. 
 
 In Section 9, we introduce $\mathbf{BW}$, the \textit{Bolzano-Weierstrass Theorem} which says that {\it every infinite and bounded sequence of real numbers has a convergent subsequence}. As this statement is constructively false, we introduce $\overleftarrow{\mathbf{BW}}$, a contraposition of $\mathbf{BW}$, and also $\overleftarrow{\mathbf{BW}_{\omega^\omega}}$  and $(\overleftarrow{\mathbf{BW}_{\omega^\omega}})^+$,  two  similar statements for $\omega^\omega$ rather than for $\mathcal{R}$. We prove that $\mathsf{BIM}\vdash \mathbf{AppFT}  \leftrightarrow \overleftarrow{\mathbf{BW}} \leftrightarrow  \overleftarrow{\mathbf{BW}_{\omega^\omega}}\leftrightarrow (\overleftarrow{ \mathbf{BW}_{\omega^\omega}})^+ $.\\We observe that  $\mathsf{BIM}\vdash \overleftarrow{\mathbf{BW}} \rightarrow \overleftarrow{\mathbf{Ded}}$.

In Section 10, we  introduce $\mathbf{Asc}$,  {\it Ascoli's Lemma}, the constructively wrong statement  saying that {\it every infinite sequence of uniformly continuous functions from $[0,1]$ to $[0,1]$, all obeying the same modulus of uniform continuity,  has a convergent subsequence}. We also introduce $\overleftarrow{\mathbf{Asc}}$, a contraposition of $\mathbf{Asc}$.   We prove: $\mathsf{BIM}\vdash \mathbf{AppFT}\leftrightarrow \overleftarrow{\mathbf{Asc}}$.

In Section 11, we introduce $\mathbf{IRT}=\forall k>0[\mathbf{IRT}(k)]$,  the Intuitionistic Ramsey Theorem.   We prove that $\mathsf{BIM}\vdash \mathbf{AppFT} \leftrightarrow \mathbf{IRT}\leftrightarrow \mathbf{IRT}(3)$. We also prove that $\mathsf{BIM}\vdash \bigl(\mathbf{OI}([0,1])+\mathbf{IRT}(2)\bigr)\rightarrow\mathbf{AppFT}$.
We introduce  $\mathbf{PH}$, the {\it Paris-Harrington Theorem}. We prove that  $\mathsf{BIM}\vdash \mathbf{IRT}\rightarrow \mathbf{PH}$,  and thus find a second argument proving that $\mathsf{BIM} + \mathbf{FT} \nvdash \mathbf{AppFT}$.

 In Section 12, we introduce $\mathsf{MP}_1$,   \textit{Markov's Principle}: $\forall \alpha[\neg\neg\exists n[\alpha(n) =0]\rightarrow \exists n[\alpha(n) =0]]$. $\mathsf{MP}_1$ is intuitionistically not acceptable. We prove that  $\mathsf{BIM} +\mathsf{MP}_1\vdash \mathbf{OI}([0,1]) \rightarrow \mathbf{AppFT}$.  We introduce $\mathbf{\Sigma}^0_1$-$\mathbf{BI}$,  the restriction of $\mathbf{BI}$ to enumerable subsets of $\omega$. We introduce $\mathbf{\Sigma}^0_1$-$\mathbf{ND}$,  the principle that every enumerable subset $A$ of $\omega$ is \textit{nearly-decidable}, i.e. $\neg\neg\exists \beta \forall m[m\in A \leftrightarrow \beta(m) \neq 0]$. We   prove that  $\mathsf{BIM} +\mathsf{MP}_1\vdash \mathbf{OI}([0,1]) \leftrightarrow \mathbf{\Sigma}^0_1$-$\mathbf{BI}\leftrightarrow \mathbf{\Sigma}^0_1$-$\mathbf{ND}$.     Section 12 builds upon earlier work by R. Solovay and J.R. Moschovakis. 
 
 In Section 13, we have collected some notations and conventions and some basic facts. The reader may consult this Section whenever he feels the need to do so.
  \section{Basic Intuitionistic Mathematics $\mathsf{BIM}$}
\subsection{The axioms of $\mathsf{BIM}$}\label{SS:bimaxioms} The formal system $\mathsf{BIM}$ (for \textit{Basic Intuitionistic Mathematics}) that we now introduce has also been used in Veldman \cite{veldman2011b} and \cite{veldman2011d}.

 There are two kinds of variables, \textit{numerical} variables $m,n,p,\ldots$,
 whose intended range is the set  $\omega$ of the natural numbers, and \textit{function}
 variables $\alpha, \beta, \gamma, \ldots$, whose intended range is the set
 $\omega^\omega$ of
 all infinite sequences of natural numbers, that is, the set of all functions from $\omega$ to
 $\omega$. There is a numerical constant $0$. There are unary function
 constants $\mathit{Id}$, a name for the identity function, $\underline{0}$, a name for the zero function, and $S$, a name for
 the successor function, and $K$, $L$, names for the projection
 functions from $\omega\times\omega$ to $\omega$. There is one binary function symbol $J$, a name for the  (surjective) pairing
 function from $\omega$ tot $\omega\times\omega$.  From these symbols \textit{numerical terms} are formed in the usual 
 way. The basic terms are the numerical variables and the numerical constant and
 more generally, a term is obtained from earlier constructed terms by the use of
 a function symbol. The function constants $\underline 0$, $S$, $K$ and $L$ and the function variables are, in the beginning stage of the development of $\mathsf{BIM}$, the only 
 \textit{function 
 terms}. As the theory
 develops, names for operations on infinite sequences are introduced and
 more complicated function terms  appear. 
 
 There are two equality symbols, $=_0$ and $=_1$. The first symbol may be placed
 between numerical terms only and the second one between function terms only.
 When confusion seems improbable we simply write $=$ and not $=_0$ or $=_1$. A
 \textit{basic formula} is an equality between numerical terms or an equality
 between function terms. A \textit{basic formula in the strict sense} is an
 equality between numerical terms. We obtain the formulas of the theory from the
 basic formulas by using the connectives, the numerical quantifiers and the
 function quantifiers. 
 
 Theorems are obtained from the axioms by the rules of intuitionistic
 logic.
  
  The first axiom is the \textit{Axiom of Extensionality}.
 \begin{axiom}\label{ax:ext}
 
 $\forall \alpha \forall \beta [ \alpha =_1 \beta  \leftrightarrow \forall n [
 \alpha(n) =_0 \beta(n) ] ] $.
 \end{axiom}
  Axiom \ref{ax:ext}  guarantees that every formula will be provably
 equivalent to a formula built up by means of connectives and quantifiers from
 basic formulas in the strict sense.

The second axiom is the \textit{axiom on  the  function constants $Id$,  $\underline 0$, $S$,   $J$, $K$ and $L$}.  
\begin{axiom}\label{ax:const}
 $\forall n[Id(n) = n]\;\wedge\;\forall n [ \underline{0}(n) = 0] \;\wedge\;\forall n [ S(n) \neq 0 ] \;\wedge\; \forall m \forall n [ S(m) = S(n)
 \rightarrow m = n ] \;\wedge\;\forall m \forall n [ K\bigl(J(m,n)\bigr) = m \;\wedge\; L\bigl(J(m,n)\bigr) = n \;\wedge\; n = J\bigl(K(n), L(n)\bigr)]$.
 \end{axiom}

 Thanks to the presence of the pairing function we may treat binary, ternary and
 other non-unary operations on $\omega$ as unary functions. ``$\alpha(m,n,p)$"
 for instance will be an abbreviation of ``$\alpha\bigl(J(J(m,n),p)\bigr)$".
 
 We introduce the following notation: for each $n$, $n':= K(n)$ and $n'' := L(n)$, and for all $m, n$, $(m,n) := J(m, n)$. The last part of Axiom 2 now reads as follows: $\forall m \forall n[(m,n)'= m \;\wedge\; (m,n)'' =n\;\wedge \; n = (n', n'')]$. Given any $\alpha$, we let $\alpha'$ and $\alpha''$ be the elements of $\omega^\omega$ defined by: $\forall n[\alpha'(n) =\bigl(\alpha(n)\bigr)'\;\wedge\; \alpha''(n) =\bigl(\alpha(n)\bigr)'']$.

 \medskip
 The third axiom\footnote{This axiom underwent an improvement with respect to its version in \cite{veldman2011b}.} asks for the closure of the set $\omega^\omega$ under the operations \textit{composition, primitive recursion and unbounded search}.
 \begin{axiom}\label{ax:recop}

 $\forall \alpha \forall \beta \exists \gamma \forall n [ \gamma(n) =
 \alpha\bigl(\beta(n)\bigr) ]\;\wedge\;\forall \alpha \forall \beta \forall \gamma \exists \delta\forall n[\delta(n) = \gamma\bigl(\alpha(n), \beta(n)\bigr)]\;\wedge$
 
 $\forall p \forall \beta \exists \gamma[\gamma(0) = p \;\wedge\; \forall n[\gamma\bigl(S(n)\bigr)=\beta\bigl(n, \gamma(n)\bigr)]\;\wedge$
 
   $\forall \alpha \forall \beta \exists \gamma \forall m \forall n [ \gamma(m,0)
 = \alpha(m) \wedge \gamma \bigl(m,S(n)\bigr) = \beta\bigl(m,n,\gamma(m,n)\bigr) \;\wedge$
 
 $\forall \alpha [ \forall m \exists n [ \alpha(m,n) = 0 ] \rightarrow
 \exists \gamma \forall m [ \alpha\bigl(m, \gamma(m)\bigr) = 0\;\wedge\; \forall n<\gamma(m)[\alpha(m,n) \neq 0] ] ]$.
 
 \end{axiom}

We introduce $\circ$, \textit{composition}, as a binary operation on functions, with defining axiom: 
$\forall \alpha \forall \beta[\alpha\circ \beta(n) = \alpha\bigl(\beta(n)\bigr)]$.

 The fourth axiom is the \textit{Unrestricted Axiom Scheme of Induction:} 
 \begin{axiom}\label{ax:ind}
 
 For every formula $\phi = \phi(n)$ the universal closure of 
 the following formula is an axiom:
 \[(\phi(0) \wedge \forall n [ \phi(n) \rightarrow \phi(S(n))]) \rightarrow
 \forall n [\phi(n)]\]
 \end{axiom}

 The system consisting of the axioms mentioned up to now will be called $\mathsf{BIM}$.

\subsection{Extending $\mathsf{BIM}$}We mention some axioms and assumptions that may be studied in the context of $\mathsf{BIM}$.  \subsubsection{}\label{SSS:cp}

\textit{Brouwer's Unrestricted Continuity Principle}, $\mathbf{BCP}$:

\textit{For every subset $\mathcal{A}$ of $\omega^\omega\times \omega$, if $\forall \alpha \exists n[\alpha \mathcal{A} n]$, then $\forall \alpha \exists m \exists n\forall \beta[\overline \alpha m = \overline \beta m \rightarrow \beta \mathcal{A}n]$.}

\smallskip
(``$\alpha\mathcal{A}n$'' abbreviates ``$(\alpha,n) \in \mathcal{A}$''.)

Note that $\mathbf{BCP}$ is an \textit{axiom scheme} and not a single axiom. The revolutionary impact of $\mathbf{BCP}$ on the development of intuitionistic mathematics is treated in \cite{veldman2001}.

 \subsubsection{}\label{SSS:ct}\textit{Church's Thesis}, $\mathbf{CT}$:
\[\exists \tau \exists \psi \forall \alpha \exists e\forall n \exists z[\tau(e, n, z) \neq 0 \;\wedge\; \forall i<z[\tau(e, n, i) =0]\;\wedge\; \psi(z) = \alpha(n)].\]

$\mathbf{CT}$ contradicts $\mathbf{BCP}$, see, for instance, \cite[Proposition 6.7]{Troelstra-van Dalen}.

\subsubsection{}\label{SSS:AC00} The \textit{ Unrestricted First Axiom of Countable Choice}, $\mathbf{AC}_{0,0}$:
\begin{center}
\textit{For every subset $A$ of $\omega$:} $\forall n \exists m[nAm] \rightarrow \exists \gamma \forall n[nA\gamma(n)]$.
\end{center}

(``$nAm$'' abbreviates ``$(n,m) \in A$''.)

Also $\mathbf{AC}_{0,0}$ is an \textit{axiom scheme}.  

The formula
 $\forall \alpha [ \forall m \exists n [ \alpha(m,n) = 0 ] \rightarrow
 \exists \gamma \forall m [ \alpha\bigl(m, \gamma(m)\bigr) = 0 ] ]$ is called the \textit{Minimal Axiom
 of Countable Choice} $\mathbf{Min}$-$\mathbf{AC}_{0,0}$. $\mathbf{Min}$-$\mathbf{AC}_{0,0}$ follows from  Axiom 3. Another special case\footnote{For each $\alpha$, $E_\alpha:=\{m\mid \exists p[\alpha(p)=m+1]$\}, see Subsection \ref{SS:decen}.} is the
$\mathbf{\Pi}^0_1$-\textit{First Axiom of Countable Choice}, $\mathbf{\Pi}^0_1$-$\mathbf{AC}_{0,0}$: 
$$\forall \alpha[\forall n\exists m[m \notin E_{\alpha^n}] \rightarrow \exists \gamma \forall n[ \gamma(n) \notin E_{\alpha^n}]],$$ i.e. $$\forall \alpha[\forall n\exists m\forall p [\alpha^n(p) \neq m+1] \rightarrow \exists \gamma \forall n\forall p[ \alpha^n(p) \neq \gamma(n) +1]].$$
 A more cautious, almost \textit{innocent}, version  is the \textit{Weak} $\mathbf{\Pi}_1^0$-\textit{First Axiom of Countable Choice}, \textbf{Weak}-$\mathbf{\Pi}^0_1$-$\mathbf{AC}_{0,0}$:
$$ \forall \alpha [ \forall n \exists m \forall p \ge m[ \alpha(n,p) = 1]
 \rightarrow \exists \gamma \forall n \forall p \ge \gamma(n) [ \alpha(n,p) =
 1]]. $$
 This axiom may be used for proving that every fan is an explicit fan:\footnote{For the notions `fan' and explicit fan', see Subsection \ref{SS:fans}.}\begin{lemma}$\mathsf{BIM}+ \mathbf{Weak}$-$\mathbf{\Pi}^0_1$-$\mathbf{AC}_{0,0}\vdash \forall \beta[Fan(\beta) \rightarrow Fan^+(\beta)]$. \end{lemma}

 \begin{proof} Let $\beta$ be given such that $Fan(\beta)$. Note: $\forall s \exists n \forall p \ge n[\beta (s\ast\langle p \rangle) = 0]$. Find $\gamma$ such that $\forall s  \forall p \ge \gamma(s)[\beta (s\ast\langle p \rangle) = 0]$ and note that $gamma$ shows that $\beta$ is an explict fan-law.  \end{proof}
\subsubsection{}\label{SSS:lpo}  The {\it Limited Principle of Omniscience}, $\mathbf{LPO}$: \[\forall \alpha[\exists n[\alpha(n) \neq 0] \;\vee\; \forall n[\alpha(n) = 0]].\]

From a constructive point view, $\mathbf{LPO}$ makes no sense. The following is a well-known counterexample in Brouwer's style. Let $d:\omega \rightarrow \omega$ be the decimal expansion of the real number $\pi$, i.e. $\pi = 3 +\sum_{n=0}^\infty d(n)\cdot 10^{-n-1}$.  Define $\alpha$ such that, for all $n$, $\alpha(n)\neq 0$ if and only if $\forall k<99[d(n+k)=9]$. We then have no proof of $\exists n[\alpha(n)\neq 0]$ and also no proof of $\forall n[\alpha(n)=0]$. This example shows that $\mathbf{LPO}$ is constructively \emph{unjustified}.

 In  $\mathsf{BIM} +\mathbf{BCP}$ one may prove  $\neg\mathbf{LPO}$, i.e. $\mathbf{LPO}$ is even \emph{contradictory} see, for instance, \cite{veldman2011d}. If we find that some given statement implies $\mathbf{LPO}$, we may conclude that this statement is unjustified, and, in the presence of $\mathbf{BCP}$, even contradictory. 

\section{The Fan Theorem and the Heine-Borel Theorem}
\subsection{Brouwer's argument for the Fan Theorem}\label{SS:fanth} 
Brouwer claims, in \cite{brouwer27} and \cite{brouwer54}, implicitly, as he jumps at once to a larger claim, see Subsection \ref{SS:bi}, that, for every subset $B$ of $2^{<\omega}$,  if $Bar_{2^\omega}(B)$, i.e. $\forall \alpha \in 2^\omega\exists n[\overline \alpha n \in B]$, then there exists a \textit{canonical} proof of this fact. The canonical  proof  is an arrangement  of statements  ``$Bar_{2^\omega\cap s}(B)$'', i.e. $\forall \alpha \in 2^\omega[s\sqsubset \alpha \rightarrow \exists n[\overline \alpha n \in B]]$,  where $s\in2^{<\omega}$. The conclusion of the proof is: $Bar_{2^\omega\cap \langle \;\rangle}(B)$, and  the proof uses only  steps of one of the following kinds: 
\begin{enumerate}[\upshape(i)]
 \item {\emph `starting points:'} 
 
 $s \in B$, \textit{and, therefore,}  $Bar_{2^\omega\cap s}(B)$,
\item {\emph `forward steps:'} 

$Bar_{2^\omega\cap s\ast \langle 0 \rangle}(B)$ and $Bar_{2^\omega\cap s\ast\langle 1\rangle}(B)$, \textit{and, therefore,} $Bar_{2^\omega\cap s}(B)$,
\item {\emph `backward steps:'} 

$Bar_{2^\omega\cap s}(B)$, \textit{and, therefore,} $Bar_{2^\omega\cap s\ast\langle 0\rangle}(B)$, and:

 $Bar_{2^\omega\cap s}(B)$, \textit{and, therefore,} $Bar_{2^\omega\cap s\ast\langle 1 \rangle}(B)$.
\end{enumerate}
One may prove that one can do without the backward steps.

\begin{definition}Let $B$ be a subset of $2^{<\omega}$.  

$B$ is  \emph{inductive in $2^{<\omega}$} if and only if  $\forall s \in 2^{<\omega}[\forall i<2[ s\ast\langle i \rangle \in B] \rightarrow s \in B]]$. 

 $B$ is \emph{monotone in $2^{<\omega}$} if and only if $\forall s \in 2^{<\omega}\forall i<2[ s \in B
 \rightarrow s\ast\langle i \rangle \in B]] $.\end{definition} 

\begin{definition} $\mathbf{BI}_{2^\omega}$, the \emph{Principle of Bar Induction in Cantor space} is the following statement:  

\textit{For all  $B\subseteq\omega$, if $Bar_{2^\omega}(B)$ and $B$ is monotone and inductive in $2^{<\omega}$, then $ \langle \;\rangle \in B$}.

$\mathbf{BI}_{2^\omega}$ may be added to $\mathsf{BIM}$ as an axiom scheme.
 We introduce two special cases of this axiom scheme:

\noindent $\mathbf{\Delta}^0_1$-$\mathbf{BI}_{2^\omega}$: 
 $\forall \alpha[Bar_{2^\omega}(D_\alpha) \;\wedge\; \forall s  \in 2^{<\omega}[s \in D_\alpha \leftrightarrow \forall i < 2[s\ast\langle i \rangle \in D_\alpha]])\rightarrow \langle \;\rangle \in D_\alpha]$,
  and
  
 \noindent $\mathbf{\Sigma}^0_1$-$\mathbf{BI}_{2^\omega}$: 
 $\forall \alpha[Bar_{2^\omega}(E_\alpha) \;\wedge\; \forall s  \in 2^{<\omega}[s \in E_\alpha \leftrightarrow \forall i < 2[s\ast\langle i \rangle \in E_\alpha]])\rightarrow \langle \;\rangle \in E_\alpha]$.
 
 \smallskip
We also introduce the \emph{Fan Theorem (for $2^\omega$)}:

$\mathbf{FT}=\mathbf{FT}_{2^\omega}$:
$\forall \alpha[Bar_{2^\omega}(D_\alpha) \rightarrow \exists n[Bar_{2^\omega}(D_{\overline \alpha n}]]$.
\end{definition}
Let $B$ be a subset of  $2^{<\omega}$ such that $Bar_{2^\omega}(B)$ and  $B$ is both inductive and monotone in $2^{<\omega}$. Take the canonical proof of ``$Bar_{2^\omega}(B)$'' and replace in this proof every statement ``$Bar_{2^\omega \cap s}(B)$'' by the statement: ``$s \in B$''.  We then obtain a proof of: ``$\langle \;\rangle \in B$''.

This argument justifies $\mathbf{BI}_{2^\omega}$.

\begin{theorem}\label{T:fantheorem}

\begin{enumerate}[\upshape (i)]
\item $\mathsf{BIM}\vdash \mathbf{\Delta}^0_1$-$\mathbf{BI}_{2^\omega}$.
\item $\mathsf{BIM}\vdash \mathbf{\Sigma}^0_1$-$\mathbf{BI}_{2^\omega}\leftrightarrow \mathbf{FT}_{2^\omega}$.\end{enumerate}
\end{theorem}

\begin{proof}
(i)  Let $\alpha$ be given such that $Bar_{2^\omega}(D_\alpha)$ and $\forall s \in 2^{<\omega}[s \in D_\alpha\leftrightarrow \forall i<2[s\ast\langle i \rangle \in D_\alpha]]$. 
Assume $\langle \; \rangle \notin D_\alpha$. Define $\gamma$  such that $\forall n[\gamma(n)=\mu i<2[\overline \gamma n \ast\langle i \rangle \notin D_\alpha]$. Note: $\gamma \in 2^\omega$ and  $\forall n[\overline \gamma n \notin D_\alpha]$. Contradiction. 
 Conclude that $\langle \; \rangle \in D_\alpha$.

(ii). First, assume $\mathbf{\Sigma}^0_1$-$\mathbf{BI}_{2^\omega}$. 
Let $\alpha$ be given such that  $Bar_{2^\omega}(D_\alpha)$.
Define $\beta$ such that, for each $n$, \textit{if} $ n'\in 2^{<\omega}$ and $length(n')\le n''$ and $\forall t\in 2^{<\omega}\cap \omega^{n''}[n'\sqsubseteq t \rightarrow \exists s \sqsubseteq  t[s\in D_\alpha]]$, {\it then $\beta(n) = n'+1$}, and, \textit{if not}, then  $\beta(n) = 0$. 
Note that $E_\beta=\{u\in 2^{<\omega}\mid\exists m\ge length(u)\forall t\in
2^{<\omega}\cap \omega^m[u \sqsubseteq t\rightarrow \exists s\sqsubseteq t[s\in D_\alpha]]\}$. 
Note that $D_\alpha\subseteq E_\beta$ and  $E_\beta$ is both monotone and inductive in $2^{<\omega}$. 
 Conclude that $\langle \; \rangle \in E_\beta$, and find $n$ such that $\beta(n)= \langle\;\rangle +1 = 1$. 
 Note that $n'= 0=\langle\;\rangle $ and $length(n')=0$ and $\forall t\in 2^{<\omega}\cap\omega^{n''}\exists s\sqsubseteq t[s\in D_\alpha]$, and, therefore, $\exists m[Bar_{2^\omega}(D_{\overline \alpha m})]$. 
 Conclude that $\mathbf{FT}$ holds.

Now assume $\mathbf{FT}$. 
Let $\alpha$ be given such that $\mathit{Bar}_{2^\omega}(E_\alpha)$  and $\forall s \in 2^{<\omega}[s\in E_\alpha\leftrightarrow \forall i<2[s\ast\langle i \rangle \in E_\alpha]]$. 
Define $\beta$ such that, for each $s$,  $\beta(s)\neq 0$ if and only if $s \in 2^{<\omega}$ and $\exists p<s\exists m< s[\alpha(p) = \overline s m + 1])]$. Note that $Bar_{2^\omega}(D_\beta)$.  Using $\mathbf{FT}$, find $p$ such that $Bar_{2^\omega}(D_{\overline \beta p})$. Find $m:=\max\{\mathit{length}(s)\mid s\in D_{\overline \beta p}\}$.   Note that $\forall s\in 2^{<\omega}\cap\omega^m\exists q\le m[\overline s q \in E_\alpha]$. We now prove that  $\forall k \le m\forall s \in 2^{<\omega}\cap\omega^k \exists q \le k[\overline s q \in E_\alpha]]$ and do so by backwards induction, starting from the case $k = m$. Suppose $k+1 \le m$ and $\forall s \in 2^{<\omega}\cap\omega^{k+1}\exists q \le k+1[\overline s q \in E_\alpha]]$. Assume that  $s\in2^{<\omega}\cap\omega^k$.  As $\forall i<2\exists t\sqsubseteq s\ast\langle i \rangle[s\ast\langle i \rangle\in E_\alpha]$, one may  distinguish two cases. \textit{Either} $\exists q\le k$[$\overline s q \in E_\alpha]$, \textit{or}  $\forall i<2[s\ast\langle i \rangle \in E_\alpha]$, and, therefore, $s \in E_\alpha$. In both cases  $\exists q \le k[\overline s q \in E_\alpha]$.
We thus see that $\langle \; \rangle \in E_\alpha$.
Conclude that $\mathbf{\Sigma}^0_1$-$\mathbf{BI}_{2^\omega}$ holds.
\end{proof} 

There are many equivalents of  $\mathbf{FT}$, see \cite{veldman2009b}, \cite{veldman2011b} and \cite{veldman2011d}.  
 \subsection{Extending  the Fan Theorem to arbitrary fans}\label{SS:extfant}

  $2^\omega$ is a fan, but it is not the only one. The notions \textit{spread, fan} and \textit{explicit fan} are introduced in Subsection \ref{SS:fans}.  
In Subsection \ref{SSS:AC00}, we saw that, in
  $\mathsf{BIM}+ \mathbf{Weak}$-$\mathbf{\Pi}^0_1$-$\mathbf{AC}_{0,0}$, one may prove $\forall \beta[Fan(\beta) \rightarrow Fan^+(\beta)]$, i.e. every fan is an explicit fan.

  \begin{definition} The following two statements are versions of the  \emph{Extended Fan Theorem},
  
   \noindent
$\mathbf{FT_{ext}}: \;\forall \beta \forall \alpha[\bigl(\mathit{Fan^+}(\beta) \;\wedge\; Bar_{\mathcal{F}_\beta}(D_\alpha)\bigr) \rightarrow \exists n[Bar_{\mathcal{F}_\beta}(D_{\overline \alpha n})]],$
and

  \noindent
$\mathbf{FT_{ext}^+}: \;\forall \beta \forall \alpha[\bigl(\mathit{Fan}(\beta) \;\wedge\; Bar_{\mathcal{F}_\beta}(D_\alpha)\bigr) \rightarrow \exists n[Bar_{\mathcal{F}_\beta}(D_{\overline \alpha n})]].$ \end{definition}

$\mathbf{FT_{ext}}$ and $\mathbf{FT_{ext}}^+$ extend $\mathbf{FT}$ from Cantor space $\mathcal{C}$ tot arbitrary (explicit) fans.
 \begin{theorem}\label{T:generalfant1} $\mathbf{BIM}\vdash \mathbf{FT}\rightarrow \mathbf{FT_{ext}}$.\end{theorem}
\begin{proof} The proof may be found in \cite{troelstra} and \cite{veldman2011b}, but we find it useful to give it here.     

Assume $\mathbf{FT}$. Let $\beta, \alpha$  be given such that $Fan^+(\beta)$ and $Bar_{\mathcal{F}_\beta}(D_\alpha)$. One may assume that $\beta(\langle\;\rangle)=0$. Find $\gamma$ such that, for each $s$, if $\beta(s) = 0$, then $\gamma(s)$ is the greatest  $n$ such that $\beta(s\ast\langle n \rangle) = 0$.

For each $s$, we define $\delta(s)$ in  $2^{<\omega}$, as follows, by induction on $\mathit{length}(s)$.
		$\delta(\langle \; \rangle)=\langle \; \rangle$ and, for all $s$, for all $n$, $\delta(s\ast\langle n \rangle) = \delta(s) \ast \underline{\overline 0}n\ast\langle 1 \rangle$.  
		
		We then define $\eta$ such that, for each $t$, $\eta(t) \neq 0$ if and only if there exist $s,i$  such that   $t=\delta(s)\ast \overline{\underline 0 }i$ and either: $\beta(s)\neq 0$ or: $\beta(s)=0$ and $i>\gamma(s)$  or: $\alpha(s) \neq 0$.   
		
		We now prove that $Bar_{2^\omega}(D_\eta)$.  Assume $\varepsilon \in 2^\omega$.  Define $\varepsilon^\ast$  as follows, by induction. For each $n$, {\it if} there exists $j\le \gamma(\overline{\varepsilon^\ast}n)$ such that   $\beta(\overline{\varepsilon^\ast}n\ast\langle j \rangle) = 0$ and $ \delta(\overline{\varepsilon^\ast}n\ast\langle j \rangle)\sqsubset \varepsilon$, then $\varepsilon^\ast(n)$ is the least such $j$, and  \textit{if no such $j$ exists}, then  $\varepsilon^\ast(n)=\gamma(\overline{\varepsilon^\ast} n)$. Note: $\varepsilon^\ast \in \mathcal{F}_\beta$ and find $n$ such that $\overline{\varepsilon^\ast}n \in D_\alpha$.  \textit{Either} $\delta(\overline{\varepsilon^\ast}n)\sqsubset \varepsilon$, \textit{or}  there exist $s, i$ such $\beta(s) = 0$ and $\delta(s)\ast\overline{\underline 0}i\sqsubset \varepsilon$ and $i> \gamma(s)$ or $\alpha(s)\neq 0$.  In all three cases, $\exists m[\overline \varepsilon m \in D_\eta]$. Conclude that $D_\eta$ is a bar in $2^\omega$.
		
		 Using $\mathbf{FT}$, find $m$ such that $Bar_{2^\omega}(D_{\overline \eta m})$. Find $p$ such that, for each $s$, if $s\ge p$, then $\delta(s) \ge m$. Conclude: $Bar_{\mathcal{F}_\beta}(D_{\overline \alpha p})$.
		
		Conclude $\mathbf{FT_{ext}}$. \end{proof}
		
\subsection{Extending $\mathbf{FT}$ and extending $\mathbf{WKL}$.}\label{SS:ftwkl} The following statement,  \textit{Weak K\"onig's Lemma} $\mathbf{WKL}$, is studied in classical reverse mathematics and constructively false:
\[\forall \alpha[ \forall m 
\exists s \in 2^{<\omega}[\mathit{length}(s) = m \;\wedge\; \forall n \le m[\alpha(\overline s n) = 0]] \rightarrow \exists \gamma \in 2^\omega \forall n[\alpha(\overline \gamma n) = 0] ].\]
  $\mathbf{FT}$  is a contraposition of  $\mathbf{WKL}$.  

Extending  $\mathbf{WKL}$ from subtrees of $2^{<\omega}$ to arbitrary finitely branching trees, one obtains \textit{K\"onig's Lemma} $\mathbf{KL}$:
\[\forall \beta \forall \alpha[\bigl(Fan(\beta)\;\wedge\; 
\forall m \exists s[\mathit{length}(s) = m \;\wedge\; \beta(s) = 0 \;\wedge\; \forall n \le m[\alpha(\overline s n) = 0]]\bigr) \rightarrow\]\[ \exists \gamma \in \mathcal{F}_{\beta} \forall n[\alpha(\overline \gamma n) = 0] ].\] 
$\mathbf{FT^+_{ext}}$  is a contraposition of  $\mathbf{KL}$. 

Restricting $\mathbf{KL}$ to \textit{explicitly} finitely branching trees, one obtains \textit{Bounded K\"onig's Lemma} $\mathbf{BKL}$:
\[\forall \beta \forall \alpha [\bigl(\mathit{Fan^+}(\beta) \;\wedge\; 
\forall m \exists s[\mathit{length}(s) = m \;\wedge\; \beta(s) = 0 \;\wedge\; \forall n \le m[\alpha(\overline s n) = 0]]\bigr) \rightarrow\]\[ \exists \gamma \in \mathcal{F}_{\beta } \forall n[\alpha(\overline \gamma n) = 0] ]].\] 
$\mathbf{FT_{ext}}$  is a contraposition of  $\mathbf{BKL}$.
 
In the classical formal context of $\mathsf{RCA_0}$, $\mathbf{KL}$ is equivalent to $\mathsf{ACA_0}$ and definitely stronger than $\mathbf{WKL}$.  $\mathbf{WKL}$ and $\mathbf{BKL}$, on the other hand, are equivalent, see \cite[Lemma IV.1.4, page 130]{Simpson} and our Theorem \ref{T:generalfant1}.

We must conclude that  \textbf{Weak}-$\mathbf{\Pi}^0_1$-$\mathbf{AC}_{0,0}$, if introduced in the classical context of $\mathsf{RCA}_0$, would bring us from $\mathbf{WKL}$ to $\mathbf{KL}$ and thus, in this context,  would have serious consequences. 

In the intuitionistic context of $\mathsf{BIM}$, however, quantifiers are read constructively, so the assumption from which \textbf{Weak}-$\mathbf{\Pi}^0_1$-$\mathbf{AC}_{0,0}$ draws its conclusion is, if it would be possible to say so, stronger than in the classical context. The intuitionistic step from $\mathbf{FT_{ext}}$ to $\mathbf{FT^+_{ext}}$ seems to be a more innocent step than the classical step from $\mathbf{BKL}$ to $\mathbf{KL}$.

\subsection{Strong bars and thin bars}
 
 \begin{definition}\label{D:strongbar}

  For every $\beta$, for every $B\subseteq\omega$,  \emph{$B$ is a strong bar in $\mathcal{F}_\beta$}, notation: $Strongbar_{\mathcal{F}_\beta}(B)$,  if and only if $\forall \zeta \in [\omega]^\omega[\forall n[\beta\bigl(\zeta(n) \bigr)=0 ] \rightarrow \exists m [\overline{\zeta(n)}m \in B]]$. \end{definition}

Note that $B$ is a strong bar in $\mathcal{F}_\beta$ if and only if  every increasing sequence of finite sequences admitted by $\beta$ contains an element that \emph{`meets'} $B$.

 We now prove that, for every  explicit fan $\mathcal{F}$, for every decidable subset $B$ of $\omega$, $B$  is a strong bar in $\mathcal{F}$ if and only if $B$ has a finite subset that is a bar in $\mathcal{F}$.  
  
\begin{lemma}\label{L:fankl1} $\mathsf{BIM} \vdash \forall \beta[Fan^+(\beta) \rightarrow\forall \alpha[Strongbar_{\mathcal{F}_\beta}(D_\alpha) \leftrightarrow \exists n[Bar_{\mathcal{F}_\beta}(D_{\overline \alpha n}]]]$.
\end{lemma}
\begin{proof} Let $\beta$ be given such that $Fan^+(\beta)$. Find $\delta$ such that, for each $n$, for each $s$ in $\omega^n$ , if $\beta(s) = 0$, then $s< \delta(n)$. 

(i) Let $\alpha$ be given such that $Strongbar_{\mathcal{F}_\beta}(D_\alpha)$.  Define $\zeta$ such that, for every $n$, $\beta\bigl(\zeta(n)\bigr)=0 $ and, if  there exists $ s<\delta(n)$ such that 
$s\in \omega^n$ and  $\neg \exists m \le n[\overline s m \in D_\alpha]$, then  $\zeta(n)$ is the least such $s$.  Find $\eta$ in $[\omega]^\omega$ such that $\zeta\circ \eta \in [\omega]^\omega$. Find $n$ such that $\exists m [\overline{\zeta\circ\eta(n)}m \in D_\alpha]$  and conclude that $\forall s \in \omega^{\eta(n)}[\beta(s)=0 \rightarrow \exists m\le \eta(n)[\overline s m \in D_\alpha]]$ and that $Bar_{\mathcal{F}_\beta}(D_{\overline \alpha \delta\circ\eta(n)})$. 

\smallskip
(ii) Assume $\exists n[Bar_{\mathcal{F}_\beta}(D_{\overline \alpha n})]$. Find $n$ such that $Bar_{\mathcal{F}_\beta}(D_{\overline \alpha n})$. Find $p$ such that, for each $s$ in $D_{\overline \alpha n}$, $length(s)< p$. Let $\zeta$ in $[\omega]^\omega$ be given such that $\forall m[\beta\bigl(\zeta(m)\bigr)=0]$.  Consider $\zeta(p)$ and note that $\zeta(p)\ge p$.  Conclude that   $\exists q\le length\bigl(\zeta(p)\bigr)[\overline{\zeta(p)}q \in D_\alpha]$. 
  Conclude that $\forall \zeta \in [\omega]^\omega[\forall m[\beta\bigl(\zeta(m)\bigr)=0]\rightarrow  \exists p\exists q[\overline{\zeta(p)}q \in D_\alpha]]$, i.e.  $Strongbar_{\mathcal{F}_\beta}(D_\alpha)$.  
  \end{proof}

\begin{definition}\label{D:thinbar} $B\subseteq \omega$ is \emph{thin} if and only if $\forall p \in B\forall q\in B[p\neq q \rightarrow p\perp q]$. For every $B\subseteq \omega$, for every $\mathcal{F}\subseteq \omega^\omega$, \emph{$B$ is a thin bar in $\mathcal{F}$}, notation: $Thinbar_\mathcal{F}(B)$, if and only if $B$ is thin and $Bar_\mathcal{F}(B)$. \end{definition} We now prove that $\mathbf{FT_{ext}}$ is, in $\mathsf{BIM}$,  equivalent to the statement that, in an explicit fan, every thin bar is finite:      
      \begin{theorem}\label{T:ftext}
The following statements are equivalent in $\mathsf{BIM}$:
\begin{enumerate}[\upshape (i)]

\item

$\mathbf{FT_{ext}}$: $\forall \beta[\mathit{Fan}^+(\beta) \rightarrow \forall \alpha[Bar_{\mathcal{F}_\beta}(D_\alpha)\rightarrow \exists n[Bar_{\mathcal{F}_\beta}(D_{\overline \alpha n})]]$.

\item $\forall \beta[\mathit{Fan}^+(\beta) \rightarrow  \forall \alpha[Thinbar_{\mathcal{F}_\beta}(D_\alpha)\rightarrow \exists n\forall m >n[\beta(m)= 0\rightarrow \alpha(m)=0]]$.
\end{enumerate}
\end{theorem}
\begin{proof} (i) $\Rightarrow$ (ii). Let $\beta, \alpha$ be given such that $Fan^+(\beta)$ and $Thinbar_{\mathcal{F}_\beta}(D_\alpha)$. Using (i), find $n$ such that $Bar_{\mathcal{F}_\beta}(D_{\overline \alpha n})$. Note that $D_\alpha$ is  thin and that  $\forall s[\beta(s)=0\rightarrow (s\in D_{\overline \alpha n} \leftrightarrow s \in D_\alpha)]$ and $\forall m \ge n[\beta(m)=0\rightarrow \alpha(m) = 0]$. 

\smallskip
  (ii) $\Rightarrow$ (i).  Let $\beta, \alpha$ be given such that $Fan^+(\beta)$ and $Bar_{\mathcal{F}_\beta}(D_\alpha)$. Define $\gamma$ such that, for each $s$, $\gamma(s) \neq 0$ if and only if $\alpha(s) \neq 0$ and $\neg \exists m<\mathit{length}(s)[\alpha(\overline s m)\neq 0]$.   Note: $D_\gamma \subseteq D_\alpha$ and $Thinbar_{\mathcal{F}_\beta}(D_\gamma)$. Using (ii), find $n$ such that $\forall m>n[\beta(m)=0\rightarrow\gamma(m) =0]$ and  conclude $Bar_{\mathcal{F}_\beta}(D_{\overline \gamma n})$  and  $Bar_{\mathcal{F}_\beta}(D_{\overline \alpha n})$.\end{proof}

\subsection{The Heine-Borel Theorem}
We will use some notions introduced in Subsections \ref{SS:reals} and \ref{SS:cov}.\begin{definition}\label{D:hb} For every subset $\mathcal{X}$  of $\mathcal{R}$, for every subset  $C$  of $\mathbb{S}$,  \emph{$C$ covers $\mathcal{X}$}, notation: $Cov_\mathcal{X}(C)$, if and only if $\forall \delta \in\mathcal{X}\exists n\exists s\in C[\delta(n) \sqsubset_\mathbb{S} s]$.

 For every $s$ in $\mathbb{S}$, $L(s) := ( s', \frac{s'+_\mathbb{Q}s''}{2} )$ and $R(s):= ( \frac{s'+_\mathbb{Q}s''}{2}, s'' )$.  $L(s)$ is  \emph{the left half of $s$} and $R(s)$ is \emph{the right half of $s$}.

  $B$ is  a mapping  from $2^{<\omega}$ to $\mathbb{S}$ such that $B(\langle\;\rangle) = (0_\mathbb{Q}, 1_\mathbb{Q})$, and for each $c$ in $2^{<\omega}$, $B(c\ast\langle 0 \rangle) = L\bigl(B(c)\bigr)$ and $B(c\ast\langle 1 \rangle) = R\bigl(B(c)\bigr)$.

For each $k$,  $\underline k$ is the element $\alpha$ of $\omega^\omega$ such that $\forall n[\alpha(n) = k]$. 

 $\mathbb{Q}_2$ is the set of (code numbers of) \emph{binary rationals} $\frac{m}{2^n}$, where $m \in \mathbb{Z}$ and $n \ge 0$.

\smallskip
The \emph{Heine-Borel Theorem} is the statement: 
\[\mathbf{HB}: \;\forall \alpha[Cov_{[0,1]}(D_\alpha) \rightarrow \exists m [Cov_{[0,1]}(D_{\overline \alpha m})]].\]\end{definition}
  \begin{theorem}\label{T:hbft}$\mathsf{BIM}\vdash \mathbf{FT} \leftrightarrow \mathbf{HB}$.  

  \end{theorem}

\begin{proof} Assume $\mathbf{FT}$. Let $\alpha$ be given such that $Cov_{[0,1]}(D_\alpha)$. Define $\beta$  such that $\forall c [\beta(c) \neq 0\leftrightarrow (c \in 2^{<\omega} \;\wedge\;\exists s\le c[s \in D_\alpha\;\wedge\; B(c) \sqsubset_\mathbb{S} s]]$. Note that $Bar_{2^\omega}(D_\beta)$.  Find $m$ such that $Bar_{2^\omega}(D_{\overline \beta m})$.  Conclude that $Cov_{[0,1]}(D_{\overline \alpha m})$. 

Conclude that $\forall \alpha[Cov_{[0,1]}(D_\alpha) \rightarrow \exists m [Cov_{[0,1]}(D_{\overline \alpha m})]]$, i.e. $\mathbf{HB}$.

\smallskip Now assume $\mathbf{HB}$. Let $\alpha$  be given such that  $Bar_{2^\omega}(D_\alpha)$. Define $\beta$  such that,  for all $s$, $\beta(s) \neq 0$ if and only if $s \in \mathbb{S}$ and either \begin{enumerate}[\upshape (1)]  \item $\exists n < s[\alpha(\overline{\underline 0}n) \neq 0 \;\wedge \; s'<_\mathbb{Q}0_\mathbb{Q} <_\mathbb{Q}s''<_\mathbb{Q} B''(\overline{\underline 0}n)]$, or\item $\exists m< s[\alpha(\overline{\underline 1}m) \neq 0 \;\wedge \;  B'(\overline{\underline 1}m)<_\mathbb{Q} s'<_\mathbb{Q} 1_\mathbb{Q} <_\mathbb{Q} s'']$, or \item$ \exists c< s\exists m< s\exists n< s[c\in 2^{<\omega} \;\wedge\; \alpha(\overline{c\ast\langle 0\rangle\ast\underline{ 1}}m)\neq 0  \;\wedge\; \alpha(\overline{c\ast\langle 1\rangle\ast\underline{ 0}}n)\neq 0\;\wedge\;  B'(\overline{c\ast\langle 0\rangle \ast\underline{ 1}}m)<_\mathbb{Q} s'<_\mathbb{Q}s''<_\mathbb{Q}   B''(\overline{c\ast\langle 1\rangle \ast\underline{ 0}}n)]$.\end{enumerate}

 We 
\textit{first} show that  $\forall q \in \mathbb{Q}_2\cap[0,1]\exists s \in D_\beta[ s'<_\mathbb{Q} q <_\mathbb{Q} s'']$ and we do so in three steps.

First consider $0_\mathbb{Q}$. Find $n$ such that $\overline{ \underline 0 }n \in D_\alpha$. Find $s$ in $\mathbb{S}$ such that $n<s$ and $s'<_\mathbb{Q}0_\mathbb{Q} <_\mathbb{Q}s''<_\mathbb{Q} B''(\overline{\underline 0}n)$. Note that $s \in D_\beta$ and that $s'<_\mathbb{Q}  0_\mathbb{Q} <_\mathbb{Q} s''$.

Then consider $1_\mathbb{Q}$. Find $n$ such that $\overline{ \underline 1 }n \in D_\alpha$. Find $s$ in $\mathbb{S}$ such that $m<s$ and $B'(\overline{\underline 1}m )<_\mathbb{Q}s'<_\mathbb{Q} 1_\mathbb{Q} <_\mathbb{Q}s''$. Note that  $s \in D_\beta$ and that $s'<_\mathbb{Q}  1_\mathbb{Q} <_\mathbb{Q} s''$.

Finally, assume $n>0$ and $0 < m < 2^n$ and $m$ is odd and consider $\frac{m}{2^n}$. Find $c$ in $2^{<\omega}$ such that $\mathit{length}(c) = n$ and  $c(n-1) = 1$ and $\sum\limits_{i<n}\frac{c(i)}{2^{i+1}} = \frac{m}{2^n}$. Note that $B(c)= (\frac{m}{2^n}, \frac{m+1}{2^n})$.  Find $p,r$ such that $\overline{\overline c (n-1) \ast \langle 0 \rangle \ast \underline 1}p\in D_\alpha$ and $\overline{c\ast \underline 0} r \in D_\alpha$.  Find $s$ in $\mathbb{S}$ such $\overline c(n-1) <s$ and $p<s$ and $r< s$ and  $B'(\overline{\overline c(n-1)\ast\langle 0\rangle \ast\underline 1}p)<_\mathbb{Q} s'<_\mathbb{Q} \frac{m}{2^n}<_\mathbb{Q} s''<_\mathbb{Q}   B''(\overline{c \ast\underline 0}r)$.  Note that $s \in D_\beta$ and $s'<_\mathbb{Q} \frac{m}{2^n} <_\mathbb{Q} s''$.

Using the above, find $\zeta$ in $\omega^\omega$ such that, for each $q$ in $\mathbb{Q}_2$, if $0_\mathbb{Q} \le_\mathbb{Q} q \le_\mathbb{Q} 1_\mathbb{Q}$, then $\zeta(q)$ is the least $s$ in $D_\beta$ such that $s'<_\mathbb{Q} q <_\mathbb{Q} s''$.

 Let $\delta$ be an element of $[0,1]$. Let $QED$ \footnote{\textit{quod \emph{est} demonstrandum, what  still has to be  proven}} denote the proposition:
$\exists s \in \mathbb{S}\exists m[s \in D_\beta \;\wedge \; \delta(m) \sqsubset_\mathbb{S} s]$.
We define $\gamma$ in $2^\omega$, by recursion, as follows.  
Let $m_0$ be the first $m$ such that {\it either} $\delta(m) \sqsubset_\mathbb{S} \zeta(0_\mathbb{Q})$  {\it or} $\delta(m) \sqsubset_\mathbb{S} \zeta(1_\mathbb{Q})$ {\it or} $\delta(m) \sqsubset_\mathbb{S} (0_\mathbb{Q}, 1_\mathbb{Q})$, and, therefore, 
\textit{either}   $QED$, \textit{or}   $\delta(m_0) \sqsubset_\mathbb{S} (0_\mathbb{Q}, 1_\mathbb{Q}) = B(\langle \; \rangle) = B(\overline \gamma 0)$.
Now let $n$ be given  and assume  we defined, for each $i<n$,  $\gamma(i)$, and ensured that either  $QED$ or  $\exists m[\delta(m) \sqsubset_\mathbb{S} B(\overline \gamma n)]$.  
 We have to define $\gamma(n)$. Consider $s:=B(\overline \gamma n)$ and find  $t:=\zeta( \frac{s'+_\mathbb{Q} s''}{2})$. Let $m_1$ be the first $m$ such that {\it either} $\delta(m) \sqsubset_\mathbb{S} t$ {\it or} $\delta(m) \sqsubset_\mathbb{S} L(s)$ {\it or}  $\delta(m) \sqsubset_\mathbb{S} R(s) $, and, therefore, {\it either}  $QED$ {\it or} $\delta(m) \sqsubset_\mathbb{S} L(s)$ {\it or}  $\delta(m) \sqsubset_\mathbb{S} R(s) $.
Define $\gamma(n) := 0$ if and only if $\delta(m_1) \sqsubset_\mathbb{S} L(s)$, and note that 
 \textit{either} $QED$, \textit{or} $\gamma(n) =0$ and  $\delta(m_1) \sqsubset_\mathbb{S} L(s)$, \textit{or}
 $\gamma(n) =1$ and  $\delta(m_1) \sqsubset_\mathbb{S} R(s)$, that is, {\it either} $QED$ {\it or} $\exists m[\delta(m) \sqsubset_\mathbb{S} B\bigl((\overline \gamma(n+1)\bigr)]$.
This completes the definition of $\gamma$. Now find $n$ such that $\overline \gamma n \in  D_\alpha$. Note that \textit{either $QED$ or $\exists m[\delta(m) \sqsubset_\mathbb{S} B(\overline \gamma n)]$}. In the latter case, find $m, s$ such that $s \in \mathbb{S}$ and $ n < s$ and $B'(\overline \gamma n) <_\mathbb{Q} s'<_\mathbb{Q} \delta'(m) <_\mathbb{Q} \delta''(m) <_\mathbb{Q} s'' <_\mathbb{Q}B''(\overline \gamma n)$. Note that $s \in D_\beta$, and $\delta( m) \sqsubset_\mathbb{S} s$, and, therefore,
 $QED$.  Conclude $QED$. 

We thus see that $\forall \delta \in [0,1] \exists s \in D_\beta\exists m[\delta)m) \sqsubset_\mathbb{S} s]$, i.e. $Cov_{[0,1]}(D_\beta)$.  
Applying $\mathbf{HB}$, we  find $m$ such that $Cov_{[0,1]}(D_{\overline \beta m})$. Assume we find $c$ in $2^{<\omega}$ such that $c>m$ and $\neg \exists k\le\mathit{length}(c)[\overline c k \in D_{\overline \alpha m}]$. Then $\neg \exists s \in D_{\overline \beta m}[ s'<_\mathbb{Q} \frac{B'(c) +B''(c)}{2}<_\mathbb{Q} s'']$. Contradiction. Conclude that $Bar_{2^\omega}(D_{\overline \alpha m})$.

 Conclude that $\forall\alpha[Bar_{2^\omega}(D_\alpha)\rightarrow \exists m[Bar_{2^\omega}(D_{\overline \alpha m}]]$, i.e.  $\mathbf{FT}$.

\end{proof}
The equivalence of $\mathbf{FT}$ and $\mathbf{HB}$ is also proven in \cite{loeb} and \cite{veldman2011b}.
    $\mathbf{HB}$ is  equivalent, in $\mathsf{BIM}$, to the statement:
$\forall \alpha[Cov_{[0,1]}(E_\alpha) \rightarrow \exists m[Cov_{[0,1]}(E_{\overline \alpha m}]]$, see \cite[Corollary 9.8(xii)]{veldman2011b}.

\section{Bar Induction and Open Induction on $[0,1]$}
   
\subsection{Brouwer's argument for the principle of Bar Induction}\label{SS:bi} 

Let $B$ be a subset of $\omega$ such that $Bar_{\omega^\omega}(B)$, i.e. $\forall \alpha \exists n[\overline \alpha n \in B]$.  Brouwer claims, see \cite{brouwer27} and \cite{brouwer54}, that there must exist a \textit{``canonical''} proof of this fact. The canonical  proof  is an arrangement  of statements of the form ``$Bar_{\omega^\omega\cap s}(B)$'', i.e. $\forall \alpha[s\sqsubset \alpha \rightarrow \exists n[\overline \alpha n \in B]]$.   The conclusion of the proof is: $Bar_{\omega^\omega\cap \langle \;\rangle}(B)$, and  the proof uses only  steps of one of the following kinds:

\begin{enumerate}[\upshape(i)]
 \item { \emph `starting points:'}
 
  $s\in B$, \textit{and, therefore,} $Bar_{\omega^\omega\cap s}(B)$.
\item {\emph `forward steps:'} 

$Bar_{\omega^\omega\cap s\ast \langle 0 \rangle}(B), Bar_{\omega^\omega\cap s\ast \langle 1\rangle}(B), Bar_{\omega^\omega\cap s\ast\langle 2 \rangle }(B), \ldots$, 

\textit{and, therefore,} $Bar_{\omega^\omega\cap s}(B)$
\item {\emph `backward steps:'} 

$Bar_{\omega^\omega\cap s}(B)$, \textit{and, therefore,} $Bar_{\omega^\omega\cap s\ast\langle n \rangle}(B)$. 
\end{enumerate}

Unlike in the special case considered in Subsection \ref{SS:fanth},  backward steps can not be missed, see, for instance, \cite{veldman2006b}.

Note that the forward steps have \emph{infinitely many premises} and that the canonical proof can not be visualized as a finite tree. 
\begin{definition} $B\subseteq \omega$  is  \emph{inductive} if and only if $\forall s[\forall n[ s\ast\langle n \rangle \in B]\rightarrow s \in B]$ and   \emph{monotone} if and only if $\forall s\forall n[s \in B \rightarrow s\ast\langle n \rangle \in B]$. 

\smallskip
$\mathbf{BI}$, the \emph{Principle of Induction on Bars in $\omega^\omega$}, is the following statement: 

{\it For all  $B\subseteq\omega$, if $Bar_{\omega^\omega}(B)$  and $B$ is monotone and inductive, then $\langle \; \rangle \in B$.}

\smallskip
$\mathbf{BI}$ may be added to $\mathsf{BIM}$ as an axiom scheme.
 The \emph{Principle of Induction on Enumerable Bars in Baire space} is the following restricted statement:
 
\smallskip $\mathbf{\Sigma}^0_1$-$\mathbf{BI}: \forall \alpha[(Bar_{\omega^\omega}( E_\alpha) \;\wedge\; \forall s[s\in E_\alpha \leftrightarrow \forall n[s\ast\langle n \rangle \in E_\alpha]])\rightarrow 0 \in E_\alpha].$\end{definition}

Let $B\subseteq \omega $ be given such that $Bar_{\omega^\omega}(B)$ and $B$ is inductive and monotone.
If we replace, in a canonical proof of ``$Bar_{\omega^\omega}(B)$'', every statement of the form ``$Bar_{\omega^\omega\cap s}(B)$'' by the statement ``$s \in B$'', we obtain a proof of ``$\langle \;\rangle \in B$''. The intuitionistic mathematician therefore accepts $\mathbf{BI}$.

   \subsection{The Principle of Open Induction on the unit interval $[0,1]$}
We will make use of some definitions in Subsections \ref{SS:reals} and \ref{SS:open}. \begin{definition}\label{D:progressive}   $\mathcal{A}\subseteq\mathcal{R}$ is  \emph{progressive in $[0,1]$} if and only if $\forall \gamma \in [0,1][[0,\gamma )\subseteq \mathcal{A}]\rightarrow \gamma \in \mathcal{A}]$.
  
  \smallskip $\mathbf{OI}([0,1])$,   the  \emph{Principle of Open Induction on $[0,1]$} is the following statement:
  
  {\it For every open subset $\mathcal{H}$ of $\mathcal{R}$, if $\mathcal{H}$ is progressive in $[0,1]$, then $[0,1]\subseteq \mathcal{H}$.}

\end{definition}
For each $s$ in $\mathbb{S}$,  $\overline s := [(s')_\mathcal{R}, (s'')_\mathcal{R}]$, see Subsection \ref{SS:reals}.
\begin{theorem}\label{T:ebioi01} $\mathsf{BIM}\vdash \mathbf{\Sigma}^0_1$-$\mathbf{BI} \rightarrow \mathbf{OI}({[}0,1{]}).$
 
\end{theorem}
\begin{proof} Assume $\mathbf{\Sigma}^0_1$-$\mathbf{BI}$. 
 Let $\alpha$ be given such that  $\forall \gamma \in [0,1][\;[0,\gamma )\subseteq \mathcal{H}_\alpha\rightarrow \gamma \in \mathcal{H}_\alpha]$. 
 Define $\delta$ such that
$\delta(\langle \; \rangle) = ( 0_\mathbb{Q}, 1_\mathbb{Q})$, and,  for each $s$, for each $n$, 
\textit{if}   $[0,\frac{1}{2}(\delta'(s)+_\mathbb{Q}\delta''(s))\bigr)\subseteq \mathcal{H}_{\overline \alpha n}]$, then  $\delta(s\ast\langle n \rangle) = R\bigl(\delta(s)\bigr)$, and, \textit{if not}, then $\delta(s\ast\langle n \rangle) = L\bigl(\delta(s)\bigr)$.
We prove that  for each $s$,  $\exists p[\;[0,\delta'(s)\bigr)\subseteq\mathcal{H}_{\overline\alpha p}]$ and we do so by induction on $length(s)$.
First note: $\delta'(\langle \; \rangle) =_\mathbb{Q} 0_\mathbb{Q}$ and $[0_\mathcal{R}, 0_\mathcal{R}) = \emptyset\subseteq \mathcal{H}_{\overline \alpha 0}$.  
Now let $s,p$ be given such that $[0, \delta'(s)\bigr)\subseteq \mathcal{H}_{\overline\alpha p}$.
Let also $n$ be given.
If $\delta(s\ast \langle n \rangle) = L\bigl(\delta(s)\bigr)$, then $[0, \delta'(s\ast\langle n \rangle)\bigr) = [0, \delta'(s)\bigr)\subseteq \mathcal{H}_{\overline \alpha p}$, and, if $\delta(s\ast \langle n \rangle) = R\bigl(\delta(s)\bigr)$, then $[0, \delta'(s\ast\langle n\rangle)=[0,\frac{1}{2}(\delta'(s)+_\mathbb{Q}\delta''(s))\bigr)\subseteq \mathcal{H}_{\overline \alpha n}$. 
  Conclude that $\forall n\exists p[\;[0, \delta'(s\ast\langle n \rangle)\bigl
  ) \subseteq\mathcal{H}_{\overline \alpha p}]$. 
 We thus see that $\forall s\exists p[\;[0, \delta'(s)\bigr)\subseteq \mathcal{H}_{\overline \alpha p}]$.

Define  $\beta$ be such that, for each $n$, 
\textit{if}  $\overline{\delta(n'')}\subseteq\mathcal{H}_{\overline\alpha n'}$, then  $\beta(n) = n''+1$, and,
\textit{if not},  then $\beta(n) = 0$.
Note that $\forall s[s\in E_\beta\leftrightarrow \exists n[\overline{\delta(s)}\subseteq \mathcal{H}_{\overline\alpha n}]]$.

We first prove that $\mathit{Bar}_{\omega^\omega}(E_\beta)$. 
Let $\varepsilon$ be given.  Define $\gamma$ such that  $\forall n[\gamma(n) = \mathit{double}_\mathbb{S}\bigl(\delta(\overline \varepsilon n)\bigr)]$. Note that  $\gamma \in [0,1]$ and $\forall \eta \in [0,\gamma)\exists n [\eta <_\mathcal{R} \gamma'(n)<_\mathcal{R} \delta'(\overline \varepsilon n)]$. Conclude that  $[0,\gamma)\subseteq\mathcal{H}_\alpha$ and $\gamma \in\mathcal{H}_\alpha$. Find  $n, t$  such that $\alpha(t) \neq 0$ and $\gamma(n) \sqsubset_\mathbb{S} t$ and note that $ \delta(\overline \varepsilon n)\sqsubset_\mathbb{S} \gamma(n) \sqsubset_\mathbb{S} t$ and that $\overline \varepsilon n \in E_\beta$.
Conclude that $\forall\varepsilon\exists n[\overline \varepsilon n\in E_\beta]$, i.e.   $Bar_{\omega^\omega}(E_\beta)$.
Note that $\forall s[s\in E_\beta\rightarrow \forall n[s\ast\langle n \rangle \in E_\beta]]$, i.e. $E_\beta$ is monotone.
We now prove that $E_\beta$ is inductive. Let $s$ be given such that $\forall n[s\ast\langle n \rangle \in E_\beta]$. Find $n$ such that $\beta(n) = s\ast \langle 0 \rangle + 1$ and note that $n'' = s\ast\langle 0 \rangle$ and  $\overline{\delta(s\ast\langle 0 \rangle)}=\overline{L\bigl(\delta(s)\bigr)}\subseteq\mathcal{H}_{\overline \alpha n'}$.  
Conclude that $\delta(s\ast\langle n' \rangle) = R\bigl(\delta(s)\bigr)$.
 Find $p$ such that  $\overline{\delta(s\ast\langle n'\rangle)}\subseteq\mathcal{H}_{\overline \alpha p}$ and define $q:=\max(n',p)$. 
 Conclude that   $\overline{\delta(s)}\subseteq\mathcal{H}_{\overline\alpha q}$ and $s \in E_\beta$.
We thus see that $\forall s[\forall n[s\ast\langle n \rangle \in E_\beta]\rightarrow s \in E_\beta]$, i.e. $E_\beta$ is inductive. 
Using $\mathbf{\Sigma}^0_1$-$\mathbf{BI}$, we conclude that $\langle \; \rangle \in E_\beta$, i.e.    $\exists n[\overline{\delta(\langle\;\rangle)}=[0,1]\subseteq \mathcal{H}_{\overline\alpha n}\subseteq\mathcal{H}_\alpha]$.
\end{proof}

       The above argument is due to Th. Coquand, see \cite{coquand}.

\begin{theorem}\label{T:oihb} $\mathsf{BIM}\vdash\mathbf{OI}({[}0,1{]}) \rightarrow \mathbf{HB}$.
     
\end{theorem}

\begin{proof} Assume $\mathbf{OI}([0,1])$.
Let $\beta$ be given such that $D_\beta$ covers $[0,1]$.   Define $\varepsilon$  such that $\forall s\in\mathbb{S}[\varepsilon(s) \neq 0\leftrightarrow \bigl(0_\mathbb{Q} <_\mathbb{Q} s''\;\wedge\;D_{\overline \beta s}\;\mathrm{covers}\; [ 0, s'']\bigr)]$.    
Note that $\forall \gamma\in [0,1][\gamma\in\mathcal{H}_\varepsilon\leftrightarrow \exists n[D_{\overline\beta n}\;\mathrm{covers}\;[0,\gamma]\;]]$. We  prove that $\forall\gamma\in[0,1][\;[0,\gamma)\subseteq\mathcal{H}_\varepsilon\rightarrow \gamma \in \mathcal{H}_\varepsilon]$. 
 Let $\gamma$ in $[0,1]$ be given such that $[0, \gamma)\subseteq \mathcal{H}_\varepsilon$.  Find $s, n$ such that
$s \in D_\beta$ and $\gamma(n) \sqsubset_\mathbb{S} s$ and  distinguish two cases. 

\textit{Case (1)}.  $s'<_\mathbb{Q} 0_\mathbb{Q}$. Note that $D_{\overline{\beta}( s+1)}$ covers  $[0, \gamma''(n)]$. Find $m>n$ such that $t := \gamma(m) >s$.  Note that $\gamma(m+1) \sqsubset_\mathbb{S} t$ and $D_{\overline \beta t}$ covers $[0, t'']$. Conclude that $t \in D_\varepsilon$ and $\gamma \in \mathcal{H}_\varepsilon$. 

\textit{Case (2)}.  $s'\ge_\mathbb{Q}0_\mathbb{Q}$. Note that $0_\mathcal{R}\le_\mathcal{R} (s')_\mathcal{R} <_\mathcal{R} \gamma$, and, therefore, $(s')_\mathcal{R} \in \mathcal{H}_\varepsilon$. Find $t, n$ such that $t \in D_\varepsilon$ and $(s')_\mathcal{R}(n) \sqsubset_\mathbb{S} t$. Note that $D_{\overline \beta (t+1)}$ covers $[ 0, t'' ]$.
Define $q := \max(s+1, t+1)$ and note that $t\in D_{\overline \beta q}$ and  $D_{\overline \beta q}$ covers $[ 0, \gamma''(n)]$. Find $m>n$ such that $u := \gamma(m) >q$. Note that  $\gamma(m+1) \sqsubset_\mathbb{S} u$ and $D_{\overline \beta u}$ covers $[0, u'']$. Conclude that $u \in D_\varepsilon$ and $\gamma \in \mathcal{H}_\varepsilon$.

We thus see that $\mathcal{H}_\varepsilon$ is progressive in $[0,1]$. Using $\mathbf{OI}([0,1])$,  conclude that $1_\mathcal{R}\in\mathcal{H}_\varepsilon$, and that there exists $n$ such that $D_{\overline \beta n}$ covers $[0,1]$.
Conclude that, for each $\beta$, if $D_\beta$ covers $[0,1]$, then there exists $n$ such that $D_{\overline \beta n}$ covers $[0,1]$, i.e.  $\mathbf{HB}$. 
\end{proof}

  \subsection{Borel's first proof}\label{SS:bfp}    One finds Borel's first proof of the Heine-Borel Theorem  in his \textit{th\`ese} from 1895, see \cite{borel1}.  One might paraphrase his (classical) argument as follows.\footnote{In \cite{borel2}, one finds another proof, more like the one most of us are used to.} \\Assume $D\subseteq\mathbb{S}$  covers $[0,1]$. We define a mapping $f$ from the first uncountable ordinal  $\omega_1$ to $\mathbb{Q}$ such that, \begin{enumerate}[\upshape (i)] \item for all $\alpha, \beta$ in $\omega_1$, if $\alpha < \beta$, and $f(\alpha) <_\mathbb{Q} 1$,  then $f(\alpha) <_\mathbb{Q} f(\beta)$, and, \item for every $\alpha$ in $\omega_1$, there is a finite  subset of $D$ covering $[0, f(\alpha)]$. \end{enumerate}
 Define $f(0) := 0$. Note that there exists indeed a finite subset of $D$ covering $[0,f(0)]=[0,0]=\{0\}$. Let $\alpha$ be an element of $\omega_1$ such that, for every $\beta <\alpha$, $f(\beta)$ has been defined already.
Calculate $\gamma := \sup \{f(\beta)\mid\beta < \alpha\}$. 
 If $\gamma>_\mathcal{R} 1$, let $\beta_0$ be the least $\beta$ such that $f(\beta) >_\mathbb{Q} 1$ and define $f(\alpha):=f(\beta_0)$.
 If $\gamma\le_\mathcal{R} 1$,  
 find $s$ in $D$ and  $n$ in $\omega$ such that $\gamma(n) \sqsubset_\mathbb{S} s$.  Find $\beta < \alpha$ such that $s'< f(\beta)$. Find a finite subset $D'$ of $D$ covering $[0, f(\beta)]$. Define $f(\alpha) := \bigl(\gamma(n)\bigr)''$. Note that the finite set  $D'\cup\{s\}$ covers $[0, f(\alpha)]$.

 Now observe there must exist $\alpha <\omega_1$ such that $f(\alpha)>_\mathbb{Q}1$, because, if there is no such $\alpha$, $f$ is an injective map from the uncountable set $\omega_1$ into the countable set $\mathbb{Q}$.
It follows that there exists a finite subset of $D$ covering $[0,1]$.

\subsection{Achilles} In his argument for the Heine-Borel Theorem sketched in Subsection \ref{SS:bfp},  Borel actually uses and proves the principle of open induction on $[0,1]$.  His indirect argument of course is not convincing from a constructive point of view. We are confronted with the following question.  If Achilles starts from $0$ and makes a step of positive length in the direction of $1$ from every position that he reaches, and if he also reaches the limit of every convergent sequence of reached positions, 
 can we see that  he will arrive at $1$? 

Achilles
first will arrive at $f(1) > f(0)=0$, then, unless $f(1) \ge 1$,  at $f(2)> f(1)$, and then, unless $f(2) \ge 1$, at $f(3) \ge f(2)$, $\ldots$, and then, taking a bold leap, at $p:=\lim_{n\rightarrow \infty} f(n)$,  a position  beyond every $f(n)$. If $p\le 1$, one may find a rational $q>p$ that is reached by Achilles and define $f(\omega):=q$. If $p>1$, one defines $f(\omega) = f(n_0)$, where $n_0$ is the least $n$ such that $f(n)>1$.  \\ It is not so easy, alas,  to calculate $p:=\lim_{n\rightarrow \infty}f(n)$.\footnote{Dedekind's Theorem, i.e.  the theorem that every bounded infinite and monotone sequence of real numbers converges is not true constructively, see Subsection \ref{SS:dt}.} And even if one finds $p$, it may be difficult to decide: $p< 1$ or $1\le p$.   It thus is not clear how to actually find $f(\omega)$.  And if one should find it, and tries to  continue the process, many more problems of the same kind will arise.
       
       How might Achilles convince himself \textit{constructively} that he will reach the point $1$?
       He would first have to convince himself that he will reach $\frac{1}{2}$.
        The problem of finding out if and when he reaches $\frac{1}{2}$, however,  does not seem easier than the problem of finding out if and when he reaches $1$.
        
        Achilles will be  very surprised and happy seeing Theorem \ref{T:ebioi01},  Brouwer's prediction  that he will indeed arrive at $1$.
      \subsection{A stronger formulation}\begin{definition} The   \emph{Stronger Principle of Open Induction on $[0,1]$}  is the following statement.

      $\mathbf{OI^+}([0,1]): \forall \alpha[\forall \gamma \in [0,1][[0,\gamma)\subseteq \mathcal{H}_\alpha
\rightarrow \gamma \in \mathcal{H}_\alpha]\rightarrow\exists n  [[0,1]\subseteq \mathcal{H}_{\overline \alpha n}]].$ \end{definition}

From  the proof of Theorem \ref{T:ebioi01}, one may conclude that $\mathsf{BIM}\vdash\mathbf{\Sigma}^0_1$-$\mathbf{BI} \rightarrow  \mathbf{OI^+}([0,1])$   and not only that $\mathsf{BIM}\vdash \mathbf{\Sigma}^0_1$-$\mathbf{BI} \rightarrow  \mathbf{OI}([0,1])$. Theorem \ref{T:oihb} makes one see that $  \mathsf{BIM}\vdash\mathbf{OI}([0,1]) \leftrightarrow  \mathbf{OI^+}([0,1])$.

\section{The contraposition of Dedekind's Theorem}

 \subsection{Preliminaries}       \begin{definition}$[\omega]^\omega$ is the set of all $\zeta$ such that $\forall n[ \zeta(n) < \zeta(n+1)]$.
       
       $\mathbb{Q}^\omega$ is the set of all $\gamma$  such that $\forall n[\gamma(n) \in \mathbb{Q}]$.

$\gamma \in \mathbb{Q}^\omega$
     \emph{converges} if and only if $\forall n\exists m\forall p \ge m[|\gamma(m+p) -_\mathbb{Q} \gamma(m)| \le_\mathbb{Q} \frac{1}{2^n}]$.
     
      $\gamma\in \mathbb{Q}^\omega$ \emph{converges explicitly}
 if and only if $\exists \delta\forall n \forall p[|\gamma\bigl(\delta(n) + p\bigr) - \gamma\circ\delta(n)| \le_\mathbb{Q} \frac{1}{2^n}]$. 
  
  $\gamma \in \mathbb{Q}^\omega$ \emph{has a converging subsequence} if and only if \[\exists\zeta \in[\omega]^\omega\forall n[ |\gamma\circ\zeta(n+1) -_\mathbb{Q} \gamma\circ\zeta(n)| \le_\mathbb{Q}\frac{1}{2^n}],\] and  
 \emph{positively fails to have a converging subsequence} if and only if
        \[\forall  \zeta \in[\omega]^\omega\exists n[ |\gamma\circ \zeta(n+1) -_\mathbb{Q} \gamma\circ\zeta(n)| >_\mathbb{Q} \frac{1}{2^n}].\]
   \indent     $\gamma\in \mathbb{Q}^\omega$ is  \emph{non-decreasing} if and only if $\forall n[\gamma(n) \le_\mathbb{Q} \gamma(n+1)]$. 
        
        If $\gamma\in \mathbb{Q}^\omega$ is non-decreasing, then $\gamma$ \emph{positively fails to converge} if and only if  \[\forall  \zeta \in[\omega]^\omega\exists n[ |\gamma\circ \zeta(n+1) -_\mathbb{Q} \gamma\circ\zeta(n)| >_\mathbb{Q} \frac{1}{2^n}]\] \end{definition}$\mathsf{BIM}$ proves that $\gamma \in \mathbb{Q}^\omega $ converges explicitly if and only if $\exists \delta \in \mathcal{R}[\lim_{n \rightarrow \infty} \gamma(n) = \delta]$, i.e. $\exists \delta \in \mathcal{R} \forall n \exists m \forall p\ge m[|\delta -_\mathcal{R} \bigl(\gamma(m+p)\bigr)_\mathcal{R}| <_\mathcal{R} \frac{1}{2^n}]$. Using  $\mathbf{Weak}$-$\mathbf{\Pi}^0_1$-$\mathbf{AC}_{0,0}$\footnote{See Subsection \ref{SSS:AC00}.} one may prove in $\mathsf{BIM}$ that, if $\gamma\in \mathbb{Q}^\omega$ converges, then $\gamma$ converges explicitly.
  $\mathsf{BIM}$ also proves that,  if $\gamma\in \mathbb{Q}^\omega$ is non-decreasing, then  $\gamma$   converges explicitly if and only if $\gamma$ has a converging subsequence. 

\begin{definition}Let $ \varepsilon$ be an element of $\mathbb{Q}^\omega$ such that, for each $n$, $\varepsilon(n) >_\mathbb{Q} 0_\mathbb{Q}$. We let $\mathit{Sum}(\varepsilon)$ be the element of $\mathbb{Q}^\omega$ such that,  for each $n$,  $\bigl(\mathit{Sum}(\varepsilon)\bigr)(n) =_\mathbb{Q} \sum_{i=0}^{i=n}\varepsilon(i)$.   \\$\varepsilon$ is called \emph{(explicitly) summable} if and only if  $\mathit{Sum(\varepsilon)}$  converges (explicitly). \end{definition}      
\begin{lemma}\label{L:ftconvprep} Let $\gamma$ be an element of $\mathbb{Q}^\omega$ and let $\lambda,\varepsilon$ be explicitly summable elements of $\mathbb{Q}^\omega$ such that, for all $n$, $\varepsilon(n) >_\mathbb{Q} 0_\mathbb{Q}$ and $\lambda(n) >_\mathbb{Q} 0_\mathbb{Q}$.
 $\mathsf{BIM}$ proves the following.
\begin{enumerate}[\upshape (i)]
\item If $\exists \zeta\in [\omega]^\omega\forall n[| \gamma\circ\zeta(n+1) -_\mathbb{Q} \gamma\circ\zeta(n)| <_\mathbb{Q} \varepsilon(n)]$, \\then $\exists\zeta\in [\omega]^\omega\forall n[ |\gamma\circ \zeta(n+1) -_\mathbb{Q}\gamma\circ\zeta(n)| <_\mathbb{Q} \lambda(n)]$.
\item If $\forall\zeta\in [\omega]^\omega\exists n[ |\gamma\circ\zeta(n+1) -_\mathbb{Q}\gamma\circ\zeta(n)| >_\mathbb{Q} \lambda(n)]$, \\then $\forall\zeta\in [\omega]^\omega\exists n[ |\gamma\circ\zeta(n+1) -_\mathbb{Q} \gamma\circ\zeta(n)| >_\mathbb{Q} \varepsilon(n)]$.

\end{enumerate}\end{lemma}
\begin{proof} Let $\lambda, \varepsilon$ be explicitly summable elements of $\mathbb{Q}^\omega$. 
 Find $\delta$ in $[\omega]^\omega$ such that $\forall n \forall p[\sum_{i=\delta(n)}^{i = \delta(n) + p} \varepsilon(i) \le_\mathbb{Q} \lambda(n)]$. Note that $\forall n[\varepsilon(n)>_\mathbb{Q} 0_\mathbb{Q}]$ and conclude that
 
 \noindent $\forall n \forall p \forall q[\sum_{i=\delta(n) + p}^{i = \delta(n) +p+q}\varepsilon(i) <_\mathbb{Q} \lambda(n)]$. 
 
 \smallskip
 (i) Let $\zeta$ be an element of $[\omega]^\omega$ such that $\forall n[ |\gamma\circ\zeta(n+1) -_\mathbb{Q} \gamma\circ\zeta(n)| <_\mathbb{Q} \varepsilon(n)]$.  Find $\eta$ in $[\omega]^\omega$ such that, for each $n$, $\zeta\bigl(\eta(n)\bigr) > \delta(n)$ and define $\zeta^\ast = \zeta \circ \eta$. Note that, for each $n$, $|\gamma \circ\zeta^\ast(n+1)-\gamma\circ\zeta^\ast(n)|= |\gamma \circ\zeta\circ \eta(n+1)-\gamma\circ\zeta\circ \eta(n)|\le_\mathbb{Q} \sum_{i=\eta(n)}^{i=\eta(n+1) -1}|\gamma\circ\zeta(i+1) -_\mathbb{Q} \gamma \circ \zeta(i)| \le_\mathbb{Q} \sum_{i=\eta(n)}^{i=\eta(n+1) -1}\varepsilon(i)<_\mathbb{Q} \lambda(n)$. 

\smallskip
(ii) Assume $ \forall\zeta\in [\omega]^\omega\exists n[ |\gamma\circ\zeta(n+1) -_\mathbb{Q}\gamma\circ\zeta(n)| >_\mathbb{Q} \lambda(n)]$.  Let $\zeta$ be an element of $[\omega]^\omega$.  Find $\eta$ in $[\omega]^\omega$ such that, for each $n$, $\zeta\bigl(\eta(n)\bigr) > \delta(n)$ and define $\zeta^\ast = \zeta\circ \eta$. Find $n$ such that $|\gamma\circ\zeta^\ast(n+1) -_\mathbb{Q} \gamma\circ\zeta^\ast(n)|>_\mathbb{Q} \lambda(n)$. Conclude that $\sum_{i=\eta(n)}^{i=\eta(n+1)-1} |\gamma\circ\zeta(i+1) -_\mathbb{Q} \gamma\circ\zeta(i)| >_\mathbb{Q} \lambda(n) >_\mathbb{Q} \sum_{i=\eta(n)}^{i = \eta(n+1)-1} \varepsilon(i)$ and find $i$ such that $\eta(n) \le i < \eta(n+1)$ and $\gamma\circ\zeta(i+1) -_\mathbb{Q} \gamma\circ\zeta(i)| >_\mathbb{Q} \varepsilon(i)$.
\end{proof}

\subsection{Dedekind's Theorem fails constructively}\label{SS:dt}  R. Dedekind wrote his \cite{dedekind} in order to justify the \textit{``intuitively clear''} statement:
 \textit{Every infinite sequence  of reals that is bounded and nondecreasing converges.}
\begin{definition} We  call the following  statement \emph{Dedekind's Theorem}.

$\mathbf{Ded}$: {\it For all $\gamma$ in $\mathbb{Q}^\omega$, \begin{center}if
$\forall n[\gamma(n)\le_\mathbb{Q} \gamma(n+1)\le_\mathbb{Q} 1_\mathbb{Q}]$, then $ \exists \zeta \in [\omega]^\omega\forall n [\gamma\circ\zeta(n+1) \le_\mathbb{Q} \gamma\circ\zeta(n) +_\mathbb{Q} \frac{1}{2^n}]]$.\end{center}}\end{definition}

\smallskip The following Theorem shows that $\mathbf{Ded}$ is constructively false.\footnote{See Subsubsection \ref{SSS:lpo}.}

\begin{theorem}\label{T:dedfalse} $\mathsf{BIM}\vdash \mathbf{Ded} \rightarrow \mathbf{LPO}$. \end{theorem}
\begin{proof} Assume $\mathbf{Ded}$.  Let $\alpha$ be given. Define $\gamma$ in $\mathbb{Q}^\omega$ such that  for all $n$, if $\forall i<n[\alpha(i) =0]$, then $\gamma(n) = 0_\mathbb{Q}$, and, if $\exists i<n[\alpha(i) \neq 0]$, then $\gamma(n) = 1_\mathbb{Q}$. Using $\mathbf{Ded}$, find $\zeta$ in $[\omega]^\omega$ such that  $\forall n [\gamma\circ\zeta(n+1) \le_\mathbb{Q} \gamma\circ\zeta(n) +_\mathbb{Q} \frac{1}{2^n}]$. Define $m:=\zeta(1)$ and note: $\forall n> m[\gamma(n) <_\mathbb{Q} \gamma(m) +_\mathbb{Q} \frac{1}{2}]$.
\textit{Either} $  \gamma(m)=1_\mathbb{Q}$ and: $\exists i<m[\alpha(i) \neq 0]$ \textit{or:} $\gamma(m) = 0_\mathbb{Q} $ and $\forall n[\gamma(n)<_\mathbb{Q} 1_\mathbb{Q}]$ and $\forall i[\alpha(i)= 0]$. We thus see that $\forall \alpha[\exists i[\alpha(i)\neq 0]\;\vee\;\forall i[\alpha(i)=0]]$, i.e.  $\mathbf{LPO}$.
\end{proof}
\subsection{A contraposition of Dedekind's Theorem}$\;$\\
We introduce a  \textit{contraposition} of Dedekind's Theorem. 

\begin{definition} The \emph{contrapositive of Dedekind's Theorem} is the following statement:

$\overleftarrow{\mathbf{Ded}}$: 
\textit{For all $\gamma$ in $\mathbb{Q}^\omega$, if $\forall n[\gamma(n)\le_\mathbb{Q} \gamma(n+1)]$ and $\forall \zeta \in [\omega]^\omega\exists n [ \gamma\circ\zeta(n) +_\mathbb{Q} \frac{1}{2^n}<_\mathbb{Q}\gamma\circ\zeta(n+1)]$, then $\exists n[ 1_\mathbb{Q}<_\mathbb{Q} \gamma(n)]$,}

i.e.  every non-decreasing infinite sequence of rationals that positively fails to converge grows beyond $1$.\end{definition}

\begin{theorem}\label{T:oided}
 $\mathsf{BIM}\vdash \mathbf{OI}({[}0,1{]})\leftrightarrow \overleftarrow{\mathbf{Ded}}$.

\end{theorem}
	 
\begin{proof}
(i) Assume $\mathbf{OI}([0,1])$.  
Let $\gamma$ in $\mathbb{Q}^\omega$ be given such that $\forall n[\gamma(n)\le_\mathbb{Q}\gamma(n+1)]$ and $\forall \zeta\in [\omega]^\omega\exists n[\gamma\circ\zeta(n)+_\mathbb{Q}\frac{1}{2^n}<_\mathbb{Q}\gamma\circ\zeta(n+1)]$. 
Define $\alpha$  such that  $\forall s\in \mathbb{S}[\alpha(s) \neq 0\leftrightarrow \exists n<s[s''<_\mathbb{Q} \gamma(n)]]$. 
Note that $\forall \eta \in \mathcal{R}[\eta \in \mathcal{H}_\alpha\leftrightarrow \exists n[\eta<_\mathcal{R}\bigl(\gamma(n)\bigr)_\mathcal{R}]]$. We now prove that $\mathcal{H}_\alpha$ is progressive in $[0,1]$. 
Let $\eta$ be an element of $[0,1]$ such that $[0, \eta) \subseteq \mathcal{H}_\alpha$.
 Note that $\forall n \exists m[\eta -_\mathcal{R}\frac{1}{2^n} <_\mathcal{R}\bigl(\gamma(m)\bigr)_\mathcal{R}]$.
 Define $\beta$ such that $\forall n[\beta(n)=\mu p[\eta''(p') -_\mathbb{Q} \frac{1}{2^n}<_\mathbb{Q} \bigl(\gamma(p'')\bigr)'(p')]$. 
  Note that  $\forall n[\eta-_\mathcal{R} \frac{1}{2^n} <_\mathcal{R} \bigl(\gamma\circ\beta''(n)\bigr)_\mathcal{R}].$
  Define $\zeta$  such that $\zeta(0) = \beta''(0)$ and $\forall n[\zeta(n+1) = \max\bigl(\zeta(n) + 1, \beta''(n+1)\bigr)]$. 
    Note that $\zeta\in[\omega]^\omega$ and $\forall n[\eta<_\mathcal{R}\bigl(\gamma\circ\zeta(n)\bigr)_\mathcal{R}+_\mathcal{R} \frac{1}{2^n}]$. 
   Find $n$ such that $ \gamma\circ\zeta(n)+ \frac{1}{2^n}<_\mathbb{Q}\gamma\circ\zeta(n+1) $.
   Conclude that $\eta<_\mathcal{R} \bigl(\gamma\circ\zeta(n)\bigr)_\mathcal{R}+ \frac{1}{2^n}<_\mathcal{R}\bigl(\gamma\circ\zeta(n+1)\bigr)_\mathcal{R}$ and  $\eta\in\mathcal{H}_\alpha$.
   We thus see that  $\forall \eta \in [0,1][\;[0,\eta)\subseteq \mathcal{H}_\alpha\rightarrow \eta \in \mathcal{H}_\alpha]$.
Using $\mathbf{OI}({[}0,1{]})$, we conclude that $[0,1]\subseteq \mathcal{H}_\alpha$, and  $\exists n[1_\mathbb{Q} <_\mathbb{Q}\gamma(n) ]$. We thus see that, for each $\gamma$ in $\mathbb{Q}^\omega$, if $\gamma$ is nondecreasing and positively fails to converge, then $\exists n[1_\mathbb{Q}<_\mathbb{Q}\gamma(n)]$, i.e. 
$\overleftarrow{\mathbf{Ded}}$.

\smallskip (ii) Assume $\overleftarrow{\mathbf{Ded}}$.
Let $\alpha$ be given such that $\mathcal{H}_\alpha$ is progressive in $[0,1]$. 
Define $\gamma$ in $\mathbb{Q}^\omega$ such that, for each $n$, $\gamma(n) = \max_\mathbb{Q}(\{q\in \mathbb{Q}\mid [0, q)\subseteq \mathcal{H}_{\overline \alpha n}\})$.  
Note that $\forall n[\gamma(n)\le_\mathbb{Q}\gamma(n+1)]$. 
We now prove that $\gamma$ positively fails to converge. Let $\zeta$ in $[\omega]^\omega$ be given. 
Define $\delta$ in $\mathbb{Q}^\omega$ such that $\delta(0)=\gamma(0)$, and, for each $n$,  $\delta(n+1)=\min_\mathbb{Q} \bigl(\delta(n)+\frac{1}{2^n}, \gamma\circ\zeta(n+1)\bigr)$. 
Note that the infinite sequence $\delta$ converges and find $\varepsilon:=\lim_{n\rightarrow \infty} \delta(n)$ in $\mathcal{R}$.
Note that, for each $n$, $[0,\delta(n))\subseteq \mathcal{H}_\alpha$ and $[0, \varepsilon)=\bigcup_n[0, \delta(n)\bigr)$.  Conclude that $[0, \varepsilon)\subseteq  \mathcal{H}_\alpha$, and that $\varepsilon \in \mathcal{H}_\alpha$. Find $s,m$ such that $s'<_\mathbb{Q} \varepsilon'(m)<_\mathbb{Q}\varepsilon''(m)<_\mathbb{Q} s''$ and $\alpha(s)\neq 0$. Find $p$ such that $[0, s')\subseteq \mathcal{H}_ {\overline \alpha p}$. Define $q:=\max(s+1, p)$ and note that $[0,s'')\subseteq \mathcal{H}_{\overline \alpha q}$. Conclude that $\varepsilon<_\mathcal{R} \bigl(\gamma(q)\bigr)_\mathcal{R}$ and note that $\gamma(q)\le_\mathbb{Q} \gamma \circ \zeta (q)$. 
Conclude that $\delta(q+1) <_\mathbb{Q} \gamma \circ \zeta(q+1)$ and $\exists i\le q[\gamma\circ\zeta(i) +_\mathbb{Q} \frac{1}{2^1} <_\mathbb{Q} \gamma\circ \zeta(i+1)]$. 
  We thus see that $\forall \zeta \in [\omega]^\omega \exists i[\gamma\circ\zeta(i)+_\mathbb{Q} \frac{1}{2^1} <_\mathbb{Q} \gamma\circ \zeta(i+1)]$. Using $\overleftarrow{\mathbf{Ded}}$, we find $n$ such that $\gamma(n)>_\mathbb{Q} 1_\mathbb{Q}$ and conclude that $[0,1]\subseteq \mathcal{H}_{\overline \alpha n}$. We thus see that, for each $\alpha$, if $\mathcal{H}_\alpha$ is progressive in $[0,1]$, then $1\in \mathcal{H}_\alpha$, i.e.  $\mathbf{OI}([0,1])$. 
\end{proof}

\subsection{The contraposition of Dedekind's Theorem in $\omega^\omega$.}\label{SS:contr(ded)N}$\;$ 

\begin{definition}For each $n$, for all $a, b$, 

$a \neq_n b\leftrightarrow \exists j<n[j<\min\bigl(length(a), length(b)\bigr)\;\wedge\; a(j) \neq b(j)]$, and $a=_n b\leftrightarrow \neg (a\neq_n b)$.\footnote{ See Subsection \ref{SS:codfs}.} \end{definition} 

The following Lemma may be compared to Lemma \ref{L:ftconvprep}(ii).
\begin{lemma}\label{L:ftconvprep2}  
 $\mathsf{BIM}$ proves the following.

For all $\gamma$, for all $\eta$ in $[\omega]^\omega$, if $\forall\zeta\in [\omega]^\omega\exists n[ \gamma\circ\zeta(n+1) \neq_{\eta(n)}\gamma\circ\zeta(n)]$, then $\forall\zeta\in [\omega]^\omega\exists n[ \gamma\circ\zeta(n+1) \neq_n \gamma\circ\zeta(n)]$.

\end{lemma}
\begin{proof} Let $\gamma, \eta$ be given such that  $\eta$ in $[\omega]^\omega$ and  $\forall\zeta\in [\omega]^\omega\exists n[ \gamma\circ\zeta(n+1) \neq_{\eta(n)}\gamma\circ\zeta(n)]$.  Let $\zeta$ be an element of $[\omega]^\omega$. Note that $\zeta\circ \eta \in [\omega]^\omega$ and find $n$ such that  $\gamma\circ\zeta\circ\eta(n+1) \neq_{\eta(n)}\gamma\circ\zeta\circ\eta(n)$. Find $i$ such that $\eta(n)\le i <\eta(n+1)$ and $\gamma\circ\zeta(i)\neq_{\eta(n)}\gamma\circ\zeta(i+1)$. Note that $\eta(n)\le i$ and $\gamma\circ\zeta(i)\neq_i\gamma\circ\zeta(i+1)$.  We thus see that $\forall\zeta\in [\omega]^\omega\exists i[ \gamma\circ\zeta(i+1) \neq_i \gamma\circ\zeta(i)]$.
\end{proof}

\begin{definition}For all $a,b$, 

$a<_{lex}b\leftrightarrow \exists i<\min\bigl(length(a), length(b)\bigr)[ \overline a i =\overline bi \;\wedge\;a(i)<b(i)]$.   

$\gamma\in \omega^\omega$  is  \emph{$1$-non-decreasing} if and only if $\forall n[\gamma(n)<_\mathit{lex} \gamma(n+1)\;\vee\;\gamma(n)\sqsubseteq \gamma(n+1)]$. If $\gamma\in \omega^\omega$ is $1$-nondecreasing, then $\gamma$ is \emph{ explicitly $1$-convergent} if and only if $\exists \zeta\in [\omega]^\omega \forall n [ \gamma\circ\zeta(n) = _n \gamma\circ\zeta(n+1)]$.
 If $\gamma\in \omega^\omega$ is $1$-nondecreasing, then    $\gamma$ \emph{positively fails to be $1$-convergent} if and only if $\forall \zeta\in[\omega]^\omega\exists n[\gamma\circ\zeta(n) \neq_n \gamma\circ\zeta(n+1)]$.
 
\smallskip  
  
The \emph{contrapositive of Dedekind's Theorem for Baire space} is the  statement:

 $\overleftarrow{\mathbf{Ded}_{\omega^\omega}}$: 
For all  $\gamma$,  
if $\forall n[\gamma(n) <_{lex} \gamma(n+1)\;\vee\;\gamma(n)\sqsubseteq \gamma(n+1)]$ and  $\forall \zeta \in [\omega]^\omega\exists n[\gamma\circ\zeta(n) \neq_n \gamma\circ\zeta(n+1)]$, 
then  $\exists m [\gamma(m) \notin 2^{<\omega}]$.   

i.e. \textit{every $1$-non-decreasing element of $\omega^\omega$ that positively fails be $1$-convergent will leave $2^{<\omega}$.} \end{definition}
\begin{theorem}\label{T:deddedN} $\mathsf{BIM}\vdash \overleftarrow{\mathbf{Ded} }\leftrightarrow \overleftarrow{\mathbf{Ded}_{\omega^\omega}}$.

\end{theorem}

\begin{proof} (i) Assume $\overleftarrow{\mathbf{Ded}}$. 

Define ${\bm \delta_0}$ such that, for all $a$ in $ 2^{<\omega}$,  
${\bm \delta_0} (a) = \sum_{i<\mathit{length}(a)} \frac{2a(i)+1}{5^i}$. 
Note that, for all $n$, for all $a,b$ in $2^{<\omega}$, if  $ a  <_{lex} b $ and $a\neq_n b$, then ${\bm \delta_0}(a) +\frac{1}{5^n} <_\mathbb{Q} {\bm \delta_0}(b)$.

Let $\gamma$ be given such that $\forall n[\gamma(n) <_{lex} \gamma(n+1)\;\vee \gamma(n)\sqsubseteq \gamma(n+1)]$ and  $\forall\zeta  \in [\omega]^\omega\exists n[\gamma\circ\zeta(n)\neq_n \gamma\circ\zeta(n+1)]$. 
Define $\eta$ in $\mathbb{Q}^\omega$ such that,  for each $p$, \textit{either} $\forall m \le p  [\gamma(m) \in 2^{<\omega}]$ and $\eta(p) = {\bm \delta_0}\bigl(\gamma(p)\bigr)$, \textit{or} $\exists m \le p[\gamma(m) \notin 2^{<\omega}]$ and $\eta(p) = (p+1)_\mathbb{Q}$. 
Note that $\forall n[\eta(n) \le_\mathbb{Q} \eta(n+1)]$. We now prove that $\eta$ positively fails to have a converging subsequence. 
Let $\zeta \in [\omega]^\omega$ be given.  Find $p$ such that $\gamma\circ\zeta(p) \neq_p 
 \gamma\circ\zeta(p+1)$ and distinguish two cases.
 \begin{enumerate}[\upshape (1)]
  \item  $\forall m \le \zeta(p+1)[\gamma(m) \in 2^{<\omega}]$. 
Then $\eta\circ\zeta(p) = {\bm \delta_0}\bigl(\gamma\circ\zeta(p)\bigr)$ and $\eta\circ\zeta(p+1) = {\bm \delta_0}\bigl(\gamma\circ\zeta(p+1)\bigr)$ and $\eta\circ\zeta(p) +_\mathbb{Q} \frac{1}{5^p} \le_\mathbb{Q}\eta\circ\zeta(p+1) $.

 \item $\exists m \le \zeta(p+1)[\gamma(m) \notin 2^{<\omega}]$. Then  $\eta\circ\zeta(p+1) =_\mathbb{Q} \bigl(\zeta(p+1)+1\bigr)_\mathbb{Q}$ and $\eta\circ\zeta(p+2) =_\mathbb{Q} \bigl(\zeta(p+2)+1\bigr)_\mathbb{Q}$ and $\eta\circ\zeta(p+2) -_\mathbb{Q} \eta\circ\zeta(p+1) \ge_\mathbb{Q} 1_\mathbb{Q}>_\mathbb{Q} \frac{1}{ 2^{p+1}}$. 
 \end{enumerate}
We thus see that $\forall \zeta \in [\omega]^\omega\exists p[ \eta\circ\zeta(p)+_\mathbb{Q} \frac{1}{5^{p+1}}<_\mathbb{Q} \eta\circ\zeta(p+1)]$.
Using Lemma \ref{L:ftconvprep} and $\overleftarrow{\mathbf{Ded}}$, we find $p$ such that $ 1_\mathbb{Q}<_\mathbb{Q} \eta(p)$. 
 Conclude that $\eta(p)  =_\mathbb{Q} (p+1)_\mathbb{Q}$ and $\exists m \le p [\gamma(m) \notin 2^{<\omega}]$.  
 We thus see that, if $\gamma \in \omega^\omega$  is $1$-nondecreasing and positively fails to be $1$-convergent, then $\exists m  [\gamma(m) \notin 2^{<\omega}]$, i.e.  $\overleftarrow{\mathbf{Ded}_{\omega^\omega}}$.
 
 \smallskip (ii) 
 Assume $\overleftarrow{\mathbf{Ded}_{\omega^\omega}}$. 
 
 Define ${\bm \delta_1}$ such that, for all $a$ in $ 2^{<\omega}$,  ${\bm \delta_1} (a) = \sum_{i<\mathit{length}(a)} \frac{a(i)}{2^{i+1}}$. 
Note that, 
for all $n$, for all $a,b$ in $2^{<\omega}$, if  $a=_n b$ then $|{\bm \delta_1}(a)-{\bm \delta_1}(b)|<\frac{1}{2^{n}} $. 

Let $\gamma$ in $\mathbb{Q}^\omega$ be given such that $\forall n[\gamma(n)\le_\mathbb{Q}\gamma(n+1)]$ and $\forall \zeta\in [\omega]^\omega\exists n[\gamma\circ \zeta(n)+_\mathbb{Q} \frac{1}{2^n} <_\mathbb{Q}\gamma\circ\zeta(n+1)]$. 
 We want to prove that $\exists n[1_\mathbb{Q}<_\mathbb{Q}\gamma(n)  ]$. First note that $\forall m\exists n>m[\gamma(n) <_\mathbb{Q} \gamma(n+1)]$.
Then find $\zeta$ in $[\omega]^\omega$ such that  $\forall n[\gamma\circ \zeta(n)<_\mathbb{Q} \gamma\circ\zeta(n+1)]$ and define $\gamma^\ast = \gamma\circ \zeta$.
Note that $\forall n[\gamma^\ast(n)<_\mathbb{Q} \gamma^\ast(n+1)]$ and $\forall \eta \in [\omega]^\omega\exists n[\gamma^\ast\circ \eta(n)+_\mathbb{Q}\frac{1}{2^n}<_\mathbb{Q}\gamma^\ast\circ\eta(n+1) ]$.  
Define $\rho$ such that,  
 for each $n$, if $\exists i\le n+1[1_\mathbb{Q}<_\mathbb{Q}\gamma^\ast(i) ]$, then $\rho(n)=\langle n +2\rangle$, and, if $\forall i\le n+1[\gamma^\ast(i) \le_\mathbb{Q} 1_\mathbb{Q}]$, then $\rho(n)$ is the least $a$ in $2^{<\omega}$ such that $\gamma^\ast(n)<_\mathbb{Q} {\bm \delta_1}(a) <_\mathbb{Q} \gamma^\ast(n+1)$ and ${\bm \delta_1}(a)-_\mathbb{Q} \gamma^\ast(n) <_\mathbb{Q} \frac{1}{2} \bigl(\gamma^\ast (n+1)-_\mathbb{Q} \gamma^\ast(n)\bigr)$.  Note that $\rho$ is $1$-non-decreasing. We now prove that $\rho$ positively fails to be 1-convergent.
Assume that $\eta \in [\omega]^\omega$. Find $n$ such that $\gamma^\ast\circ\eta(n) +_\mathbb{Q}\frac{1}{2^n}<_\mathbb{Q} \gamma^\ast\circ\eta(n+1)$ and distinguish two cases.
 \begin{enumerate}[\upshape (1)]
 \item  $\forall i \le \eta(n+1)+1[\gamma^\ast(i) \le_\mathbb{Q} 1_\mathbb{Q}]$. 
Then $|\rho\circ\eta(n+1)-_\mathbb{Q} \rho\circ\eta(n)|=|{\bm \delta_1}\circ\gamma^\ast\circ\eta(n+1) -_\mathbb{Q}{\bm \delta_1}\circ\gamma^\ast\circ\eta(n)|>_\mathbb{Q} \frac{1}{2}\bigl(\gamma^\ast\circ\eta(n+1)-_\mathbb{Q} \gamma^\ast\circ \eta(n)\bigr) > \frac{1}{2^{n+1}}$. Conclude that $\rho\circ\eta(n+1)\ne_{n+1}\rho\circ\eta(n)$. 
\item $\exists i \le \eta(n+1)+1[\gamma^\ast(i)>_\mathbb{Q} 1_\mathbb{Q} ]$. Then  $\rho\circ\eta(n+1) =_\mathbb{Q} \langle\eta(n+1)+2\rangle$ and $\rho\circ\eta(n+2) =\langle \bigl(\eta(n+2)+2\rangle$ and $\rho\circ\eta(n+2) \neq_{1} \rho\circ\eta(n+1) $. \end{enumerate}
 We thus see that $\forall \eta \in [\omega]^\omega\exists n[ \rho\circ\eta(n+1) \neq_{n+1} \rho\circ\eta(n)]$.
Using Lemma \ref{L:ftconvprep2} and $\overleftarrow{\mathbf{Ded}_{\omega^\omega}}$, we find $p$ such that $\eta(p)\notin 2^{<\omega}$. Conclude  that $\eta(p)  =\langle p+2\rangle$ and $\exists n\le p [1_\mathbb{Q} <_\mathbb{Q}\gamma^\ast(n)]$ and $\exists n[1_\mathbb{Q}<_\mathbb{Q}\gamma(n) ]$.   We thus see that, for all $\gamma$ in $\mathbb{Q}^\omega$, if $\gamma$ is non-decreasing and positively fails to  converge, then $\exists n[1_\mathbb{Q}<_\mathbb{Q}\gamma(n) ]$, i.e.  $\overleftarrow{\mathbf{Ded}}$.
\end{proof}

\subsection{The Principle of Open Induction on Cantor space $2^\omega$}\label{SS:oicantorspace} 

\begin{definition}\label{D:progrcantorspace}
 For all $\mathcal{A}\subseteq \omega^\omega$, for all $\gamma$ in $\mathcal{A}$, $\mathcal{A}_{<_{lex}\gamma}=\{\delta \in \mathcal{A}\mid \delta<_{lex}\gamma\}$.
 
 $\mathcal{A}\subseteq \omega^\omega$ is  \emph{progressive in $2^\omega$} if and only if $\forall \gamma \in 2^\omega[(2^\omega)_{<_{lex}\gamma}\subseteq \mathcal{A} \rightarrow \gamma \in\mathcal{A}]$.  
 
 \smallskip The \emph{Principle of Open Induction on Cantor space} is the following statement:

$\mathbf{OI}(2^\omega)$: \textit{For every open subset $\mathcal{G}$ of $\omega^\omega$, if $\mathcal{G}$ is progressive  in $2^\omega$, then} $2^\omega\subseteq \mathcal{G}$.
\end{definition}
\begin{theorem}\label{T:oic} $\mathsf{BIM} \vdash \overleftarrow{\mathbf{Ded}_{\omega^\omega}} \leftrightarrow \mathbf{OI}(2^\omega)$.
 
\end{theorem}

\begin{proof} (i) Assume $\overleftarrow{\mathbf{Ded}_{\omega^\omega}}$. 
Let $\alpha$ be given such that $\forall \gamma \in 2^\omega[(2^\omega)_{ <_{lex} \gamma}\subseteq \mathcal{G}_\alpha\rightarrow \gamma \in \mathcal{G}_\alpha]$. 
Define $\gamma$ such that $\gamma(0):=\langle\;\rangle$ and, for each $n$, (1) if    $2^\omega \nsubseteq \mathcal{G}_{\overline\alpha n}$ then  $\gamma(n)$  is the least element  $s$ of $2^{<\omega}$ satisfying  $(2^\omega)_{<_{lex}s}\subseteq \mathcal{G}_{\overline \alpha n}$ and $2^\omega\cap s\nsubseteq \mathcal{G}_{\overline \alpha n}$, and,  (2) if  $2^\omega \subseteq \mathcal{G}_{\overline\alpha n}$, then $\gamma(n)=\langle n +2\rangle$. 
Note that $\forall n[(2^\omega)_{<_{lex}\gamma(n)}\subseteq \mathcal{G}_{\overline \alpha n}]$  and $\forall n[\gamma(n)<_{lex}\gamma(n+1)\;\vee\;\gamma(n)\sqsubseteq \gamma(n+1)]$. 
 We now prove that $\gamma$ positively fails to be $1$-convergent. Let $\zeta$ in $[\omega]^\omega$ be given.
Define $\delta$  such that $\delta(0)=\gamma\circ\zeta(0)$, and, for each $n$, if  $\forall i\le n[\gamma\circ \zeta(i) =_i \gamma\circ\zeta(i+1)$, then $\delta(n+1)=\gamma\circ\zeta(n+1)$, and, if not, then $\delta(n+1)= \delta(n)\ast\langle 0\rangle$. Note that, for each $n$, $\delta(n)=_n\delta(n+1)$.
  Define $\varepsilon$ such that $\forall n[\varepsilon =_n \delta(n)]$ and note that, for all $\beta$ in $2^\omega$, if $\beta<_{lex} \varepsilon$, then  $\exists n[\beta<\delta(n)]$, and $\exists n[\beta<_{lex} \gamma(n)]$, and $\beta \in \mathcal{G}_\alpha$. We thus see that $2^\omega_{<_{lex}\varepsilon}\subseteq \mathcal{G}_\alpha$ and conclude that $\varepsilon \in \mathcal{G}_\alpha$. 
Find $n$ such that $\alpha(\overline \varepsilon n)\neq 0$. Define $p:=\max(\zeta(n), \overline \varepsilon n +1)$  and note that $\varepsilon <_{lex}  \gamma(p)$.
Conclude that  $\exists i\le n[\gamma\circ\zeta(i) \neq_i \gamma\circ \zeta(i+1)]$. 
   We thus see: $\forall \zeta \in [\omega]^\omega \exists i[\gamma\circ\zeta(i)\neq_i \gamma\circ \zeta(i+1)]$. 
Using $\overleftarrow{\mathbf{Ded}_{\omega^\omega}}$, we find $n$ such that $\gamma(n)\notin 2^{<\omega}$ and conclude: $2^\omega\subseteq \mathcal{G}_{\overline \alpha n}$. 

We thus see that, for all $\alpha$, if $\mathcal{G}_\alpha$ is progressive in $2^\omega$, then $2^\omega\subseteq \mathcal{G}_\alpha$, i.e.  $\mathbf{OI}(2^\omega)$.  

\smallskip (ii) Assume $\mathbf{OI}(2^\omega)$.
Let $\gamma$ be given such that $\forall n[\gamma(n) <_{lex} \gamma(n+1)\;\vee\; \gamma(n)\sqsubseteq \gamma(n+1)]$ and $\forall \zeta \in [\omega]^\omega\exists n[\gamma\circ\zeta(n)\neq_n \gamma\circ\zeta(n+1)]$. Note that $\exists n [\gamma(n) <_{lex} \gamma(n+1)]$ and conclude that $\exists n\exists m[\underline {\overline 0}m <_{lex} \gamma(n)]$. Also note that $\forall m \exists p>m[\gamma(p)<_{lex} \gamma(p+1) ]$.

 Define $\alpha$ such that, for each $s$, $\alpha(s)\neq 0$ if and only if $\exists n < s  [s<_{lex}  \gamma(n)]$.
Note that $\mathcal{G}_\alpha=\{\beta\mid \exists n[\beta<_{lex} \gamma(n)]\}$.
We now prove that $\mathcal{G}_\alpha$ is progressive in $2^\omega$. 
 Let $\delta$ in $2^\omega$ be given such that $(2^\omega)_{<_{lex}\delta}\subseteq  \mathcal{G}_\alpha$. 
We want to prove that $\delta \in \mathcal{G}_\alpha$. 
We  first prove that,  for all $n$, $\overline \delta n\ast \underline 0$ belongs to $G_\alpha$. We do so by induction. 
First observe that $\overline \delta 0\ast\underline 0=\underline 0$ belongs to $\mathcal{G}_\alpha$, as $\exists n\exists m[\underline {\overline 0}m <_{lex} \gamma(n)]$.
 Now let $n$ be given such that $\overline \delta n\ast \underline 0$ belongs to $\mathcal{G}_\alpha$.
If $\delta(n)=0$, then $\overline\delta(n+1)\ast\underline 0=\overline\delta n\ast \underline 0$ belongs to $\mathcal{G}_\alpha$. 
Now assume $\delta(n)=1$ and note that  $\overline \delta n\ast\langle 0\rangle\ast \underline 1$ belongs to $\mathcal{G}_\alpha$. Find $m$ such that $\overline \delta n\ast\langle 0\rangle\ast \underline 1 <_{lex} \gamma(m)$. Conclude that $\overline \delta n\ast\langle 0\rangle <_{lex} \gamma(m)$ and: {\it either} $\overline \delta n<_{lex}\gamma(m)$ and: $\overline \delta(n+1) \ast \underline 0<_{lex} \gamma(m)$, {\it or} $\overline \delta(n+1)\sqsubseteq \gamma(m)$, and, for some $p>m$, $\gamma(p)<_{lex} \gamma(p+1)$ and also: $\delta(n+1)\ast \underline 0 <_{lex} \gamma(p+1)$. We may conclude that, for all $n$, $\overline \delta n\ast \underline 0$ belongs to $G_\alpha$. 
Now find $\zeta$ in $[\omega]^\omega$ such that, for each $n$, 
$\overline \delta n\ast\underline  0 <_{lex}\gamma\circ\zeta(n)$. Find $n$ such that $\gamma\circ\zeta(n)\neq_n \gamma\circ\zeta(n+1)$ and note that $\overline \delta n <_{lex} \gamma\circ\zeta(n+1)$.
Conclude that $\delta<_{lex} \gamma\circ\zeta(n+1)$ and $\delta \in \mathcal{G}_\alpha$.
We thus see that $\mathcal{G}_\alpha$ is progressive in $2^\omega$.
Conclude that $2^\omega\subseteq \mathcal{G}_\alpha$ and  $\underline 1 \in \mathcal{G}_\alpha$ and  $\exists m[\gamma(m)\notin 2^{<\omega}]$.

We thus see that, for each $\gamma$, if $\gamma$ is $1$-non-decreasing and positively fails to converge, then $\exists n[\gamma(n)\notin 2^{<\omega}]$, i.e. $\overleftarrow{\mathbf{Ded}_{\omega^\omega}}$.

\end{proof}	
\begin{corollary}\label{C:oided} $\mathsf{BIM}\vdash \mathbf{OI}([0,1])\leftrightarrow \mathbf{OI}(2^\omega)\leftrightarrow \overleftarrow{\mathbf{Ded}}\leftrightarrow\overleftarrow{\mathbf{Ded}_{\omega^\omega}}$. \end{corollary}
\begin{proof} See Theorems \ref{T:oided}, \ref{T:deddedN} and \ref{T:oic}. \end{proof}
       
   \section{ Enumerable subsets of $\omega$ that positively fail to be decidable}     For each $\alpha$,  $D_\alpha:=\{n|\alpha(n)\neq 0\} $ and $E_\alpha:=\{ m|\exists n[\alpha(n) = m+1]\} $. $D_\alpha$ is  \textit{the subset of $\omega$ decided by $\alpha$} and $E_\alpha$ is  \textit{the subset of $\omega$ enumerated by $\alpha$}.\footnote{See Subsection \ref{SS:decen}.}

       \begin{lemma}
         $\mathsf{BIM}\vdash  \forall \alpha \exists \gamma[D_\alpha = E_\gamma]$.

       \end{lemma}

\begin{proof}  Let $\alpha$  be given. Define $\gamma$ such that, for each $n$, if $\alpha(n) = 0$, then $\gamma(n) = 0$, and, if $\alpha(n)\neq 0$, then $\gamma(n) = n+1$. Clearly, $D_\alpha = E_\gamma$.
 \end{proof}

 \begin{lemma}\label{L:consbcp}
        $\mathsf{BIM} +\mathbf{BCP} \vdash \neg \forall \gamma\exists\alpha[E_\gamma= D_\alpha]$.

       \end{lemma}

\begin{proof}   Assume that $\forall \gamma\exists\alpha[E_\gamma= D_\alpha]$. Conclude that $\forall \gamma[  0 \in E_\gamma \;\vee\; 0 \notin E_\gamma]$. Using Brouwer's Continuity Principle $\mathbf{BCP}$\footnote{For  $\mathbf{BCP}$, see Subsubsection \ref{SSS:cp}.}, find $p$ such that {\it either} $\forall \gamma[\overline \gamma p = \underline{\overline 0}p\rightarrow 0 \in E_\gamma]$, {\it or}  $\forall \gamma[\overline \gamma p = \underline{\overline 0}p\rightarrow 0 \notin E_\gamma]$. Note that $0 \notin E_{\underline 0}$ and $0 \in E_{\underline{\overline 0} p\ast \underline 1}$, and conclude that both alternatives are false.
\end{proof}
\begin{lemma}\label{L:ct} $\mathsf{BIM} + \mathbf{CT} \vdash \exists \gamma\forall \alpha[ D_\alpha \subseteq E_\gamma \rightarrow \exists n \in E_\gamma[n \notin D_\alpha]]$. 
\end{lemma}

\begin{proof} Using Church's Thesis $\mathbf{CT}$\footnote{ For  $\mathbf{CT}$. see Subsubsection \ref{SSS:ct}.},  find $\tau$, $\psi$ such that $\forall \alpha \exists e\forall n \exists z[\tau(e, n, z) \neq 0 \;\wedge\; \forall i<z[\tau(e, n, i) =0 ]\;\wedge\; \psi(z) = \alpha(n)]$. Define, for each $n$, $W_n:=\{m\mid\exists z[\tau(n,m,z)\neq 0]\}$. Define the so-called  \textit{(self-)halting problem} $K := \{ n\mid n \in W_n\}$.  Define $\gamma$ such that, for each $n$, \textit{if} $\tau(n',n',n'') \neq 0$ and $\forall i < n''[\tau(n',n', i) =0]$, \textit{then} $\gamma(n) = n'+1$, and, \textit{if not, then} $\gamma(n) = 0$, and note that $K= E_\gamma$. Let $\alpha$ be given such that $D_\alpha \subseteq E_\gamma$. Define $\delta$ such that, for each $n$, if $\alpha(n) =0$, then $\delta(n) = 0$ and, if $\alpha(n) \neq 0$, then $\delta(n) = \psi(z) +1$ where $z$ is the least number $i$ such that $\tau(n,n,i) \neq 0$. Find $e$ such that $\forall n \exists z[\tau(e, n, z) \neq 0 \;\wedge\; \forall i<z[\tau(e, n, i) =0]\;\wedge\; \psi(z) = \delta(n)]$. Note that $ e \in K=E_\gamma$.
Assume $\alpha(e) \neq 0$. Let $z$ be the least $i$ such that $\tau(e,e,i) = 1$ and conclude that $\psi(z) = \psi(z) +1$. Contradiction. Conclude that $\alpha(e) = 0$ and $ e\in E_\gamma\setminus D_\alpha$.
 \end{proof}
    \begin{definition} Let $D,E$ be subsets of $\omega$ such that $D\subseteq E$.   $D$  \emph{ is properly contained in}   $E$ or $D$ \emph{is surpassed by $E$} if and only if  $\exists n\in E[n \notin D]$. $X\subseteq \omega$  \emph{positively fails to be decidable} if and only if for every $\alpha$,   if $D_\alpha\subseteq X$, then $D_\alpha$ is  properly contained in $X$. \end{definition}  
    
    Lemma \ref{L:ct} shows that $\mathbf{CT}$ enables one to find $\gamma$ such that $E_\gamma$  \textit{positively fails to be decidable}.

\begin{definition}  We introduce the following statement.
 \[\mathbf{EnDec?!}:\forall \gamma[\forall \alpha[(D_\alpha \subseteq E_\gamma \; \wedge \; \exists n[n \notin D_\alpha]) \rightarrow \exists p[p \notin D_\alpha \; \wedge \; p \in E_\gamma]] \rightarrow E_\gamma = \omega].\] 
i.e.  \emph{for every $\gamma$,  \emph{if,} for each $\alpha$,  if $D_\alpha\subseteq E_\gamma$  is surpassed by $\omega$, then $D_\alpha$ is surpassed by  $E_\gamma$, \emph{then} 
        $E_\gamma = \omega$.}\end{definition}

  Note that $\mathbf{EnDec?!}$ implies:   \textit{``There is no $\gamma$ such that $E_\gamma$ positively fails to be decidable.''} and thus contradicts $\mathbf{CT}$.

      \begin{theorem}\label{T:oiendec}
 $\mathsf{BIM}\vdash\mathbf{OI}(2^\omega)\rightarrow\mathbf{EnDec?!}$.

      \end{theorem}

\begin{proof} Assume $\mathbf{OI}(2^\omega)$.

 Let  $\gamma,n$ be  given such that   $\forall \alpha  \in 2^\omega[ ( D_\alpha \subseteq E_\gamma\;\wedge\; n  \notin D_\alpha) \rightarrow \exists p[p \notin
      D_\alpha\; \wedge p\; \in E_\gamma]]$.
     Define $\zeta$  such that,  for all $a$ in $2^{<\omega}$, $\zeta(a) \neq 0$ if and only if $\exists p <length(a)\exists i<length(a)[\bigl(a(p) = 0 \;\wedge\; \gamma(i) = p+1\bigr)\vee \gamma(i) = n+1]$. 
        Note that $\mathcal{G}_\zeta=\{\alpha\in2^\omega\mid \exists p[p \notin D_\alpha \;\wedge\; p \in E_\gamma]\;\vee\; n \in E_\gamma\}$. 
  We now prove that $\mathcal{G}_\zeta$ is progressive in $2^\omega$.
        Let $\alpha$ in $2^\omega$ be given such that  $(2^\omega)_{<_{lex}\alpha}\subseteq \mathcal{G}_\zeta$.
        Let also $k$ in $D_\alpha$ be given.
Define $\beta$ in $2^\omega$ such that $\beta(k) = 0$ and $\forall m \neq k[\beta(m) = \alpha (m)]$. 
Note that $D_\beta = D_\alpha\setminus \{k\}$ and $\beta <_{lex} \alpha$ and, therefore,  $\beta \in \mathcal{G}_\zeta$.  
Find $p$ such that either $p \in E_\gamma\setminus D_\beta$ or $n \in E_\gamma$.  
Note that (1) if $p \in E_\gamma\setminus D_\beta$ and $p\neq k$, then $p \in E_\gamma \setminus D_\alpha$ and  $\alpha \in \mathcal{G}_\zeta$, and, (2) 
if $p \in E_\gamma\setminus D_\beta$ and $p = k$, then   $k \in E_\gamma$, and 
(3) if $n \in E_\gamma$, then $\alpha \in \mathcal{G}_\zeta$. 
We thus see that  $\forall k\in D_ \alpha[\alpha \in \mathcal{G}_\zeta\;\vee\; k \in E_\gamma]$, i.e.
 $\forall k\in D_\alpha\exists q[ \overline \alpha q \in D_\zeta\;\vee\;\gamma(q) = k+1]$.
Define $\delta$  such that $\forall k \in D_\alpha[\delta(k)=\mu q[\overline \alpha q \in D_\zeta\;\vee\;\gamma(q) = k+1]$.
 Define $\alpha^\ast$  such that  $\forall k[\alpha^\ast(k) \neq 0 \leftrightarrow \bigl(\alpha(k) \neq 0\;\wedge\;\gamma(\delta(k))=k+1\bigr)]$.
     Note that $D_{\alpha^\ast}\subseteq  E_\gamma$ and find $p$ in $E_\gamma\setminus D_{\alpha^\ast}$.
      \textit{Either} $\alpha(p)=0$ and $p \in E_\gamma\setminus D_\alpha$ and $\alpha \in \mathcal{G}_\zeta$, \textit{or} $\alpha(p) \neq 0$ and $\gamma(\delta(p)) \neq p+1$ and: $\overline \alpha\delta(p) \in D_\zeta$ and $\alpha \in \mathcal{G}_\zeta$.
       In both cases, $\alpha \in \mathcal{G}_\zeta$.
       We thus see that $\forall \alpha\in 2^\omega[(2^\omega)_{<_{lex}\alpha}\subseteq \mathcal{G}_\zeta\rightarrow\alpha\in\mathcal{G}_\zeta]$.
    Using $\mathbf{OI}(2^\omega)$  we conclude that $2^\omega\subseteq \mathcal{G}_\zeta$ and,  in particular,  $\underline 1\in\mathcal{G}_\zeta$, and $n \in E_\gamma$. Conclude that, for all $\gamma, n$, if $\forall \alpha  \in 2^\omega[ ( D_\alpha \subseteq E_\gamma\;\wedge\; n  \notin D_\alpha) \rightarrow \exists p[p \notin
      D_\alpha\; \wedge p\; \in E_\gamma]]$, then $n\in E_\gamma$.

Now let  $\gamma$ be given such that $ \forall \alpha \in 2^\omega[(D_\alpha \subseteq E_\gamma \; \wedge \; \exists n[n \notin D_\alpha]) \rightarrow \exists p[p \notin D_\alpha \; \wedge \; p \in E_\gamma]]$.
 Conclude that $\forall n[n \in E_\gamma]$ and   $  E_\gamma = \omega$.
 
  Conclude that, for each $\gamma$, if every decidable subset of $\omega$ that is contained in $E_\gamma$ and properly contained in $\omega$  is also properly contained in $E_\gamma$, then $E_\gamma=\omega$, i.e.   $\mathbf{EnDec?!}$.
 \end{proof}

     \begin{theorem}\label{T:endoi}
 $\mathsf{BIM}\vdash\mathbf{EnDec?!}\rightarrow \mathbf{OI}([0,1])$.  \end{theorem}

      \begin{proof} Assume $\mathbf{EnDec?!}$.
      
      Let $\alpha$ be given such that  $\forall \delta\in[0,1][\;[0,\delta)\subseteq \mathcal{H}_\alpha \rightarrow \delta\in\mathcal{H}_\alpha]$.
      We shall prove that $[0,1]\subseteq \mathcal{H}_\alpha$.
      Define $\zeta$  such that $\forall n[\zeta(n)=\mu i[i \in \mathbb{Q}\;\wedge\;0_\mathbb{Q} \le_\mathbb{Q} i \le_\mathbb{Q} 1_\mathbb{Q}\;\wedge\; \forall j<i[\zeta(j) \neq \zeta(i)]]]$. 
       Note that $\{\zeta(n)\mid n \in \omega\} =\{q \in \mathbb{Q}|0_\mathbb{Q} \le_\mathbb{Q} q \le_\mathbb{Q} 1_\mathbb{Q}\}$.
       Define $\gamma$  such that, for each $n$,  \textit{if}  $[0, \zeta(n'')]\subseteq \mathcal{H}_{\overline\alpha n'}$, then $\gamma(n) = n'' + 1$, and, \textit{if not}, then $\gamma(n) = 0$. 
       Note that  $\forall n[n \in E_\gamma\leftrightarrow\exists m[\;[0, \zeta(n)]\subseteq \mathcal{H}_{\overline\alpha m}]]$. We shall prove that  every decidable subset of $\omega$ that is surpassed by $\omega$ and contained in $E_\gamma$  is also surpassed by $E_\gamma$.
       Let $\beta,n$ be given such that $n \notin D_\beta$ and $D_\beta \subseteq E_\gamma$.
      We shall prove that $ \exists p[p \in E_\gamma \setminus  D_\beta]$.
       Define $k_0:=\mu i[\zeta(i) = 0_\mathbb{Q}]$. Note that $0_\mathcal{R} \in \mathcal{H}_\alpha$ and $k_0 \in E_\gamma$. 
  Note that,     if $k_0 \notin D_\beta$, we are done.
        Now assume that $k_0 \in D_\beta$.
       Define  $\delta$   as follows, by induction.   Define $\delta(0) := (k_0,n)$. Note that  $\delta'(0)\in D_\beta$ and $\delta''(0)\notin D_\beta$ and  $\zeta\circ\delta'(0) <_\mathbb{Q}\zeta\circ\delta''(0)$.
        Let  $m$ be given such that  $\delta(m)$ has been defined and $\delta'(m)\in D_\beta$ and $\delta''(m)\notin D_\beta$ and  $\zeta\circ\delta'(m) <_\mathbb{Q}\zeta\circ\delta''(m)$.
       Find $l$ such that $\zeta(l) = \frac{1}{2}\bigl(\zeta\circ\delta'(m)+_\mathbb{Q} \zeta\circ\delta''(m)\bigr)$. 
        \textit{If} $l\in D_\beta$,  define $\delta(m+1)  = \bigl(l, \delta''(m)\bigr)$, and,  \textit{if} $l \notin D_\beta$,  define $\delta(m+1):= \bigl(\delta'(m), l\bigr)$.
        Note that, for each $m$,  $\delta'(m)\in D_\beta \subseteq E_\gamma$, and  $\exists r[\;[0, \zeta\circ\delta'(m)]\subseteq \mathcal{H}_{\overline\alpha r}]$.
       Also note that, for each $m$, $\delta''(m)\notin D_\beta$.
      Define $\varepsilon$  such that $\forall n[\varepsilon(n) = \mathit{double}_\mathbb{S}\bigl((\zeta\circ\delta'(n),\zeta\circ\delta''(n))\bigr)]$.
        Note that $\varepsilon \in [0,1]$ and  
        $\forall\eta \in [0,\varepsilon)\exists m[\eta <_\mathcal{R} (\zeta \circ \delta'(m))_\mathcal{R}]$ and $\forall\eta \in [0,\varepsilon)\exists r[\eta \in \mathcal{H}_{\overline\alpha r}\subseteq\mathcal{H}_\alpha]$, so $[0,\varepsilon)\subseteq \mathcal{H}_\alpha$.    Conclude that $\varepsilon \in \mathcal{H}_\alpha$. 
       Find $s, p$ such that $s \in \mathbb{S}$ and $s \in D_\alpha$ and $\varepsilon(p) \sqsubset_\mathbb{S} s$. 
        Note that $s'<_\mathbb{Q}\varepsilon'(p)<_\mathbb{Q} \zeta\circ\delta'(p)<_\mathbb{Q}\zeta\circ\delta''(p)<_\mathbb{Q} \varepsilon''(p)<_\mathbb{Q} s''$. Find $r$ such that $D_{\overline \alpha r}$ covers $[0, \zeta\circ\delta'(p)]$. Define $t :=\max(r, s+1)$ and  note that $D_{\overline \alpha t}$ covers $[0, \zeta\circ\delta''(p)]$.    Conclude that $\delta''(p) \in E_\gamma\setminus D_\beta$.
We thus see that  $\forall \beta[ (D_\beta\subseteq E_\gamma \;\wedge\;\;
\exists n[n\notin  D_\beta])\rightarrow \exists n [n \in E_\gamma \setminus D_\beta]]$.
Using $\mathbf{EnDec?!}$,  we conclude that $E_\gamma =\omega$. 
 Define $k_1:=\mu i[\zeta(k_1) =1_\mathbb{Q}]$ and find $n$ such that $\gamma(n) = k_1 +1$. 
 Conclude that  $n'' = k_1$ and  $D_{\overline \alpha n'}$ covers $[0_\mathbb{Q},1_\mathbb{Q}]$ and $[0_\mathbb{Q},1_\mathbb{Q}] \subseteq \mathcal{H}_{\overline \alpha n '} \subseteq \mathcal{H}_\alpha$. 
 
 We thus see that, for each $\alpha$, if $\mathcal{H}_\alpha$ is progressive in $[0,1]$, then $[0,1]\subseteq \mathcal{H}_\alpha$, i.e. $\mathbf{OI}([0,1])$.
     \end{proof}

    \begin{corollary}\label{Cor:sumup} $\mathsf{BIM} \vdash  \mathbf{OI}([0,1]) \leftrightarrow \overleftarrow{\mathbf{Ded}}\leftrightarrow 
  \overleftarrow{\mathbf{Ded}_{\omega^\omega}} \leftrightarrow \mathbf{OI}(2^\omega)\leftrightarrow \mathbf{EnDec?!}$. \end{corollary}

       \begin{proof}  Use Theorems \ref{T:oiendec} and \ref{T:endoi} and Corollary \ref{C:oided}.\end{proof}
       
 \section{Below a bar, $<_{KB}$ is a well-ordering}
  
 \begin{definition} $<_{KB}$ is a binary relation on $\omega$ satisfying  $\forall s\forall t[s<_{KB}t\leftrightarrow (t \sqsubset s\;\vee\; s <_{lex} t)]$.

  The ordering $<_{KB}$ is called the \emph{Kleene-Brouwer ordering} of $\omega$.
  
   For every decidable\footnote{$X\subseteq \omega$ is {\it decidable} if and only if $\exists \alpha[X=D_\alpha]$.} $X\subseteq \omega$,  $Below(X):=\{s \mid \forall t\sqsubseteq s[t \notin X]\}$.
  
   For every decidable $Y\subseteq \omega$,  \emph{ $<_{KB}$ well-orders $Y$} if and only if, for all $\zeta$, if $\forall n[\zeta(n+1)<_{KB} \zeta(n)]$, then $\exists n[\zeta(n)\notin Y]$.
  \end{definition}
  \begin{lemma}\label{L:wellorders2times} Let $Y$ be a decidable subset of $\omega$. The following statements are equivalent in $\mathsf{BIM}$. 
  \begin{enumerate}[\upshape(i)] \item $<_{KB}$ well-orders $Y$. \item for all $\zeta$, there exists $n$ such that either  $\neg\bigl(\zeta(n+1)<_{KB} \zeta(n)\bigr)$or $\zeta(n)\notin Y$. \end{enumerate} \end{lemma}
  
  \begin{proof} (i) $\Rightarrow$ (ii). Let $\zeta$ be given. Define $\zeta^\ast$ such that $\zeta^\ast(0)=\zeta(0)$ and, for each $n$, if $\forall i\le n[\zeta(i+1) <_{KB}\zeta(i)]$, then $\zeta^\ast(n+1)=\zeta(n+1)$, and, if  not, then $\zeta^\ast(n+1) = \zeta^\ast(n)\ast\langle 0\rangle$. 
  Note that $\forall n[\zeta^\ast(n+1)<_{KB} \zeta^\ast(n)]$ and find $n$ such that $\zeta^\ast (n) \notin Y$. 
  Either $\zeta(n) =\zeta^\ast(n)\notin Y$ or $\exists i<n[\neg(\zeta(i+1)<_{KB} \zeta(i)\bigr)]$.

\smallskip (ii) $\Rightarrow$ (i). Obvious.  \end{proof}
 \begin{definition}\label{D:finiteseqbinary} For all $s,n$, if $n \ge \mathit{length}(s)$, then $s(n) = 0$. \footnote{See Subsection \ref{SS:codfs}.}

 We define ${\bm
\delta}$ such that ${\bm \delta}(\langle\;\rangle)=\langle\;\rangle$ and,   $\forall s\forall n[{\bm \delta}(s\ast\langle n \rangle) = {\bm \delta}( s)\ast \underline{\overline 0}n\ast\langle 1 \rangle]$. \end{definition}
Note that $\forall s[{\bm \delta}(s) \in 2^{<\omega}]$ and $\forall s \forall t[s<_{KB} t\leftrightarrow \bigl({\bm \delta}(t) <_{lex} {\bm \delta}(s)\;\vee \;{\bm \delta}(t)\sqsubset{\bm \delta}(s)\bigr)]$. 
Note that,  for all $\gamma$ in $2^\omega$, $\exists \beta[\gamma={\bm \delta}|\beta]$ if and only if  $\forall m \exists n>m[\beta(n)=1]$.
\begin{theorem}\label{T:dedKB1} The following statements are equivalent in $\mathsf{BIM}$:

\begin{enumerate}[\upshape (i)] \item 
 $\overleftarrow{\mathbf{Ded}_{\omega^\omega}}$. \item  
 
 For all $\alpha$, if $\mathit{Bar}_{\omega^\omega}(D_\alpha)$, then $<_{KB}$ well-orders $Below(D_\alpha)$. \item $\mathbf{EnDec?!}$.
 \end{enumerate}
 
\end{theorem}	
\begin{proof} (i) $\Rightarrow$ (ii).  Assume $\overleftarrow{\mathbf{Ded}_{\omega^\omega}}$. 
		
		Let $\alpha,\zeta$ be given such that $Bar_{\omega^\omega}(D_\alpha)$  and $\forall n[\zeta(n+1)<_{KB}\zeta(n)]$.
		 We shall prove that  $ \exists n\exists m\le\mathit{length}\bigl(\zeta(n)\bigr)[\overline{\zeta(n)}m\in D_\alpha]$.
		 
		 Define $\zeta^\ast$ such that, for each $n$, if $\neg\exists i\le n\exists m \le\mathit{length}\bigl(\zeta(i)\bigr)[\overline{\zeta(i)}m \in D_\alpha]$, then $\zeta^\ast(n)={\bm \delta}\circ\zeta(n)$, and,  if $\exists i\le n\exists m \le\mathit{length}\bigl(\zeta(i)\bigr)[\overline{\zeta(i)}m \in D_\alpha]$, then $\zeta^\ast(n)=\langle n+2\rangle$.
		Note that, for all $n$,  $\zeta(n+1)<_{lex} \zeta(n)\;\vee\;\zeta(n)\sqsubset\zeta(n+1)$ and $\zeta^\ast(n)<_{lex}\zeta^\ast(n+1)\;\vee\;\zeta^\ast(n)\sqsubset\zeta^\ast(n+1)$. We now first prove that $\forall n\exists j[\zeta^\ast(n+j)<_{lex}\zeta^\ast(n+j+1)]$.
		 
Let $n$ be given. 
Define $\beta$ such that $\beta(0)=\zeta^\ast(n)$ and, for each $j$, if $\forall i\le j[\zeta^\ast (n+i)\sqsubset \zeta^\ast(n+i+1)]$ then $\beta(j+1)=\zeta^\ast(n+j+1)$, and, if not, then $\beta(j+1)=\beta(j)\ast\langle 1 \rangle$. 
Note that, for all $j$, $\beta(j)\sqsubset \beta(j+1)$ and $length\bigl(\beta(j)\bigr)\ge j$. 
Find $\gamma$ in $2^\omega$ such that $\forall j[ \beta(j)\sqsubset \gamma]$. Note that, for each $j$, $length\bigl(\beta(j)\bigr)>0$ and $\bigl(\beta(j)\bigr)(length(\beta(j))-1)=1$ and conclude  that $\forall i\exists j>i[\gamma(j)=1]$. Find $\varepsilon$ such that ${\bm \delta}|\varepsilon=\gamma$. Find $k$ such that $\overline \varepsilon k \in D_\alpha$. Find $q:=\mathit{length}\bigl({\bm \delta}(\varepsilon(k)\bigr)$. 
Assume that $\beta(q)=\zeta^\ast(n+q)$. Then $\overline\varepsilon k=\overline{\zeta(n+q)}k\in D_\alpha$ and $\zeta^\ast(n+q)=\langle n+q+2\rangle$. Contradiction.
Conclude that $\beta(q)\neq \zeta^\ast(n+q)$ and that there exists  $j\le q$ such that 
$\neg(\zeta^\ast(n+j) \sqsubset \zeta^\ast(n+j+1))$ and  $\zeta^\ast(n+j) <_{lex} \zeta^\ast(n+j+1)$.

 Now define $\theta$ such that $\theta(0)=\mu j[\zeta^\ast(j) <_{lex}\zeta^\ast(j+1)]$ and $\forall j[\theta(i+1) =\mu j[j>\theta(i)\;\wedge\;\zeta^\ast(j)<_{lex}\zeta^\ast(j+1)]]$. 
Define $\zeta^\diamond:=\zeta^\ast \circ \theta$ and note that $\forall j[\zeta^\diamond(j)<_{lex}\zeta^\diamond(j+1)]$. 
We now  prove that $\forall \eta \in [\omega]^\omega \exists n[\zeta^\diamond\circ \eta(n) \neq_n \zeta^\diamond\circ\eta(n+1)]$.	

Let $\eta$ in $[\omega]^\omega$ be given.
Define $\gamma$ in $2^\omega$ such that, for each $j$, 
 if
		$\forall i\le j[\zeta^\diamond\circ \eta(i) =_{i+1} \zeta^\diamond\circ\eta(i+1)]$, then $ \gamma (j) = \bigl(\zeta^\diamond\circ\eta(j)\bigr)(j)$, and, 
		if not, then $\gamma(j) = 1$. 
		We prove that $\forall n \exists m >n[\gamma(m) = 1]$. 
		
		Let $n$ be given. There are several cases to consider. 
		
		\smallskip
		\textit{Case (i)}.   $\forall i \le n[\zeta^\diamond\circ\eta(i) =_{i+1} \zeta^\diamond \circ \eta(i+1)]$.  Note  $\overline \gamma(n+1) =_{n+1}\zeta^\diamond\circ\eta(n) =_{n+1}\zeta^\diamond\circ\eta(n+1)$ and $\zeta^\diamond\circ \eta(n)<_{lex} \zeta^\diamond\circ \eta (n+1)$. 
		Find $p:=\mu i[\bigl(\zeta^\diamond\circ\eta(n)\bigr)(n+i+1)=0\;\wedge\; \bigl(\zeta^\diamond\circ\eta(n+1)\bigr)(n+i+1)=1]$. 
		Note that $\bigl(\zeta^\diamond\circ\eta(p)\bigr)(n+p+1)=1$ and consider $\overline \gamma(n+p+2)$.

		    \textit{Subcase (i)a.}  $\forall i \le n+p+1[\zeta^\diamond\circ\eta(i) =_{i+1}\zeta^\diamond\circ\eta(i+1)]$. 
		    Conclude that $\overline \gamma (n+p+2) =_{n+p+2} \zeta^\diamond\circ\eta(n+p+2)$ and $\gamma(n+p+1)=1$.

		    \textit{Subcase (i)b.} $\exists  i \le n+p+1[\zeta^\diamond\circ\eta(i) \neq_{i+1}\zeta^\diamond\circ\eta(i+1)]$. 
		   Conclude that $\gamma(n+p+1) =1$.
		 
	\smallskip	  \textit{Case (ii)}.  $\exists i \le n+1[\zeta^\ast\circ\eta(i) \neq_i \zeta^\ast \circ \eta(i+1)]$. Conclude that $\gamma(n+1) = 1$.

\smallskip	  
We thus see that	$\forall n\exists m>n[\gamma(m) =1]$. 	

		\smallskip
		Find $\varepsilon$ such that $\forall n[{\bm \delta}(\overline \varepsilon n) \sqsubset \gamma]$ and  find $n$ such that $\overline \varepsilon n \in D_\alpha$. Define $q:= \mathit{length}\bigl({\bm \delta}(\overline \varepsilon (n))\bigr)$. 
		Assume that  $\forall i \le q[\zeta^\diamond\circ\eta(i) =_{i+1} \zeta^\diamond\circ\eta(i+1)]$. 
		 Then $\overline \gamma q = {\bm \delta}\bigl(\overline\varepsilon(n)\bigr)  \sqsubseteq \zeta^\diamond\circ\eta(q)$ and $\overline\varepsilon(n) \sqsubseteq \zeta\circ\theta\circ\eta(q)$.
 As $\overline \varepsilon n \in D_\alpha$ and $q\le\theta\circ\eta(q)$, $\zeta^\diamond\circ\eta(q)=\langle \theta\circ\eta(q)+2) \rangle \neq\langle \theta\circ\eta(q)+3\rangle=\zeta^\diamond\circ\eta(q+1)$ and $\zeta^\diamond\circ\eta(q) \neq_1 \zeta^\diamond\circ\eta(q+1)$. Contradiction.
		Conclude that $\exists i \le q[\zeta^\diamond\circ\eta(i) \neq_{i+1} \zeta^\diamond\circ\eta(i+1)]$.
	We thus see that $\forall \eta \in [\omega]^\omega\exists i[\zeta^\diamond\circ \eta(i) \neq_i \zeta^\diamond\circ\eta(i+1)]$.
		Using $\overleftarrow{\mathbf{Ded}_{\omega^\omega}} $ we conclude that 
$\exists n[\zeta^\diamond(n)\notin 2^{<\omega}]$  and
	$\exists n\exists m\le\mathit{length}\bigl(\zeta(n)\bigr)[\overline{\zeta(n)}m\in D_\alpha]$. 
	
	We thus see that, for all $\alpha$,  $\zeta$, if $\mathit{Bar}_{\omega^\omega}(D_\alpha)$ and $\forall n[\zeta(n+1)<_{KB}\zeta(n)]$, then $\exists n\exists m[\overline{\zeta(n)}m\in D_\alpha]]$ and $\exists n[\zeta(n)\notin Below(D_\alpha)]$, i.e. $<_{KB}$ well-orders $Below(D_\alpha)$.

		\medskip (ii) $\Rightarrow$ (iii). 
Assume:	for all $\alpha$,  if $\mathit{Bar}_{\omega^\omega}(D_\alpha)$, then $<_{KB}$ well-orders $Below(D_\alpha)$.  

Let $\gamma, n$ be given such that $\forall \alpha \in 2^\omega[(n \notin D_\alpha \;\wedge\; D_\alpha \subseteq E_\gamma) \rightarrow \exists p[p \in E_\gamma \setminus D_\alpha]].$ 
		 We shall prove that $n \in E_\gamma$,
		 i.e. $\exists p[\gamma(p) = n+1]$.
		 
		 Note that $\forall m \exists \alpha[E_{\overline \gamma m} = D_\alpha]$. It follows that $\forall m \exists p >m[E_{\overline \gamma m} \subsetneq E_{\overline \gamma p} \;\vee\; n \in E_{\overline \gamma m} ]$.
		 
		 Define $\eta$  such that $\eta(0) =0$ and $\forall m[\eta(m+1)=\mu i> \eta(m)[E_{\overline \gamma(\eta(m))}\subsetneq E_{\overline \gamma  i}\;\vee\;n \in E_{\overline \gamma i)}].$ 
		 Note that $\forall m[n \notin E_{\overline \gamma (\eta(m+1))}\rightarrow \exists s \in [\omega]^{m}\forall p[p \in E_{\overline \gamma(\eta(m+1))}\leftrightarrow\exists i<m[s(i)=p]]$.    
		 Define $\zeta$ such that $\zeta(0) = \langle \;\rangle$ and, for each $m$, either: \begin{enumerate}[\upshape (1)] \item $n \in E_{\overline \gamma (\eta( m+1))}$ and $\zeta(m+1)=\zeta(m)\ast\langle n \rangle$, or: \item $n \notin E_{\overline \gamma (\eta( m+1))}$ and $\zeta(m+1)=\mu s \in [\omega]^{m+1}\forall p[p \in E_{\overline \gamma(\eta(m+1))}\leftrightarrow \exists j<m[s(j) = p]]$. \end{enumerate}
	Note that, for all $m$, $\zeta(m) \in \omega^m$ and $\forall i<m[\bigl(\zeta(m)\bigr)(i) \in E_\gamma]$ and $\zeta(m+1) <_{KB} \zeta(m)$.
		 Define $\lambda$ such that, for all $m$,  $\lambda(m)\in\omega^{2m}$  and $\forall i<m[\bigl(\lambda(m)\bigr)(2i) = \bigl(\zeta(m)\bigr)(i)\;\wedge\;\bigl(\lambda(m)\bigr)(2i+1)=\mu q[\gamma(q) = \bigl(\zeta(m)\bigr)(i) +1]$.
		  Note that, for all $m$, $\lambda(m+1) <_{KB} \lambda(m)$.

		 For all $k,s$  such that $s\in \omega^{2k}$, 
		we define: $s$ \textit{is fine}  if and only if \begin{enumerate}[\upshape (1)]\item
		  $\forall i<k[\gamma\bigl(s(2i+1)\bigr) = s(2i) +1]$ and \item
		  $\forall i[i+1<k\rightarrow \bigl(s(2i) < s(2i+2)\;\vee\;s(2i+2)=n\bigr)]$ and \item$\forall i<k[s(2i)=n \rightarrow s(2i+2)=n]$.\end{enumerate}
		  Note that,  if $s\in \omega^{2k}$ is fine, then $\forall i<k[s(2i)\in E_\gamma]$  and, if $\forall i<k[ s(2i) \neq n]$ then $\forall i[i+1 < k \rightarrow s(2i) <s(2i+2)]$. Note that one may decide, given any $s$ in $\omega^{2k}$, if $s$ is fine  or not. Note that, if $s$ is fine, then for every $j<k$, $\overline s(2j)$ is fine. Note that, for every $m$, $\lambda(m)$ is fine.

		 For all $k,s$  such that $s\in \omega^{2k}$, 
		we define: $s$ \textit{is forgetful}  if and only if $s$ is fine  and, for some $j<k$, $\gamma(j)>0$ and {\it either}  $\gamma(j) - 1 < s(0)$ {\it or}  $\exists i[i+1 <k\;\wedge\;s(2i) < \gamma(j) -1<s(2i+2)]$, {\it or} $\gamma(j) = n+1$.
		  If $s\in \omega^{2k}$ is forgetful,  then \textit{either}  the finite sequence $\bigl(s(0), s(2), \ldots, s(2k-2)\bigr)$ is \textit{not} the list of the first $k$ elements of $E_\gamma$ \textit{or} $n \in  E_\gamma$. 
		  Note that one may decide, given any $s$ in $\omega^{2k}$, if $s\in \omega^{2k}$ is forgetful  or not. Note that, for all $m$, if $\neg\exists i<m[\gamma(i)=n+1]$, then $\lambda(m)$ is not forgetful.

		 Define $B:=\{s \in \bigcup_k \omega^{2k}\mid s\; \mathit{is\; not\; fine}\; \vee \;\mathit{s\; is\; forgetful}\}$. $B$ is a decidable subset of $\omega$ as one may define $\tau$ such that $B=D_\tau$.
		 We now prove that $Bar_{\omega^\omega}(B)$.
		 
		  Let $\beta$ be given. 
		 Define $\nu$ such that, for each $m$, if $\overline \beta(2m+2)$ is fine, then
		 $\overline\nu(2m+2) = \overline\beta(2m+2)$, and, 
		 if $\overline \beta(2m+2)$ is not fine,  then  $\nu(2m+1)= \mu q[\gamma(q) >0\;\wedge\;\bigl(\gamma(q)-1= n\;\vee\;\gamma(q) -1 > \nu(2m-2)\bigr)]$ and $\nu(2m)=\gamma\bigl(\nu(2m+1)\bigr)-1$.
		 Note that, for all $m$,  $\overline \nu( 2m)\;\mathit{is \;fine}$
		and either $\nu(2m) < \nu(2m+2)$ or $\nu(2m+2) = n$. 
		Define $\alpha$ such that, for each $m$,  $\alpha(m) \neq 0\leftrightarrow\bigl( m\neq n\;\wedge\;  \exists j\le m[m = \nu(2j)\bigr)$.
		 Note that $D_\alpha\subseteq E_\gamma$ and that $n \notin D_\alpha$. Find $p=\mu q[\gamma(q') = q'' +1\;\wedge\; q'' \notin D_\alpha]$ and distinguish two cases.
		  \textit{Case (a)}. $p'' = n$. Note that $\overline \nu(2p')$ is forgetful.
		 \textit{Case (b)}. $p'' \neq n$. Find $i$ such that either $p''<\nu(0)$ or $\nu(2i)<p''<\nu(2i+2)$. Define $m:=\max(i+1, p'+1)$. Note that $\overline \nu(2m)$ is forgetful.
		In both cases we thus find $m$ such that $\overline \nu(2m)$ is forgetful.
		{\it Either} $\overline \beta(2m) \neq \overline \nu(2m)$ and $\overline\beta (2m)$ is not fine, {\it or} $\overline \beta(2m)=\overline \nu(2m)$ is forgetful. Conclude that $\overline \beta(2m)\in B$.
		 We thus see that, for each $\beta$, there exists $m$ such that $\overline\beta(2m)\in B$, i.e. $Bar_{\omega^\omega}(B)$.
		 
		 Using the fact that $<_{KB}$ well-orders $Below(B)$, find $m,p$ such that $\overline{\lambda(m)}p\in B$. \\As $\lambda(m)$ is fine,   $\overline{\lambda(m)}p$ is forgetful  and also  $\lambda(m)$ itself is  forgetful. Conclude that $\exists j<m[\gamma(j)=n+1]$ and $n\in E_\gamma$.

We thus see that,	for each $n$, 	for each $\gamma$,   if $ \forall \alpha \in 2^\omega[(n \notin D_\alpha \;\wedge\; D_\alpha \subseteq E_\gamma) \rightarrow \exists p[p \in E_\gamma \;\wedge\; p \notin D_\alpha]]$, then  $n \in E_\gamma$. We conclude that,    
		 for each $\gamma$, if $\forall \alpha \in 2^\omega[(D_\alpha \subseteq E_\gamma \; \wedge \; \exists n[n \notin D_\alpha]) \rightarrow \exists p[p \notin D_\alpha \; \wedge \; p \in E_\gamma]]$, then $ E_\gamma = \omega$, i.e. 
		$\mathbf{EnDec?!}$.
		 
		 \smallskip (iii) $\Rightarrow$ (i). See Corollary \ref{Cor:sumup}. 
		\end{proof} Theorem  \ref{T:dedKB1} is a counterpart to  \cite[Lemma V.1.3 and Exercise V.1.11]{Simpson}.

 In Definition \ref{D:thinbar}, we defined:
  		
		 $B\subseteq \omega$  is  \emph{thin} if and only if $\forall p\in B\forall q\in B[p\neq q \rightarrow p\perp q]$. 
	
	  $B\subseteq\omega$ is a {\it thin bar in $\mathcal{F}\subseteq \omega^\omega$}, notation $Thinbar_\mathcal{F}(B)$  if and only if $B$ is thin and  $Bar_\mathcal{F}$.

	\begin{corollary}The following are equivalent in $\mathsf{BIM}$ \begin{enumerate}[\upshape (i)] \item $\mathbf{OI}(2^\omega)$. \item For all $\alpha$, if $Bar_{\omega^\omega}(D_\alpha)$, then $<_{KB} $ well-orders $Below(D_\alpha)$. \item For all $\alpha$, if $Thinbar_{\omega^\omega}(D_\alpha)$, then $<_{KB} $ well-orders $D_\alpha$. \end{enumerate} 
	\end{corollary}
	\begin{proof} (i)$\Leftrightarrow$(ii). See Theorem 	\ref{T:dedKB1} and  Corollary \ref{Cor:sumup}.
	
	\smallskip (ii)$\Rightarrow$(iii). Assume (ii). 
Let $\alpha$ be given such that $Thinbar(D_\alpha)$. Define $\beta$ such that $\beta(\langle\;\rangle)=0$ and $\forall s\forall n[\beta(s\ast \langle n \rangle) \neq 0 \leftrightarrow \alpha(s)\neq 0]$. Note that $Bar_{\omega^\omega}(D_\beta)$ and $D_\alpha\subseteq Below(D_\beta))$. As $<_{KB}$ well-orders  $Below(D_\beta)$, $<_{KB}$ also well-orders $D_\alpha$.  Note that, as $D_\alpha$ is thin, the relations $<_{KB}$ and $<_{lex}$ coincide on $D_\alpha$. Conclude that (iii) holds.

\smallskip (iii)$\Rightarrow$ (ii). Assume (iii). Let $\alpha$ be given such that $Bar_{\omega^\omega}(D_\alpha)$. Define $\beta$ such that $\forall s[\beta(s) \neq 0 \leftrightarrow (s \in D_\alpha\;\wedge\;\forall t\sqsubset s[t \notin D_\alpha])]$. Note that $Thinbar_{\omega^\omega}(D_\beta)$. Let $\zeta$ be given.  We have to prove  $QED:= \exists n [\neg \bigl(\zeta(n+1) <_{KB} \zeta (n)\bigr)\;\vee\; \zeta(n) \notin Below(D_\alpha]$. We claim that  $\forall i \exists j\ge i[ \zeta(j+1)<_{lex} \zeta(j)\;\vee\;QED]$.
We prove this claim as follows. Let $i$ be given.  Define $\zeta^\ast$ such that $\zeta^\ast(0) = \zeta(i)$ and, for each $j$, {\it if} $\forall i\le j[\zeta(i)\sqsubset \zeta(i+1)]$, then $\zeta^\ast(j+1)=\zeta(j+1)$, and, {\it if not}, then $\zeta^\ast(j+1)= \zeta^\ast(j)\ast\langle 0 \rangle$. Note that, for all $j$, $\zeta^\ast(j)\sqsubset \zeta^\ast(j+1)$, and find $\gamma$ such that $\forall n[\zeta^\ast(n)\sqsubset \gamma]$. Then  find $p$ such that $\overline \gamma p \in D_\alpha$ and $n$ such that $\overline \gamma  p \sqsubseteq \zeta^\ast(n)$. Note: $\zeta^\ast(n)\notin Below( D_\alpha)$.
If $\zeta^\ast(n) = \zeta(n)$, then $QED$, and, if not, $\exists k<n[\zeta(i+k+1)<_{lex} \zeta(i+k)]$. Conclude that $\forall i \exists j\ge i[ \zeta(j+1)<_{lex} \zeta(j)\;\vee\;QED]$.
Find $\eta$ in $[\omega]^\omega$ such that  $\forall n[ \zeta\circ \eta(n+1) <_{lex} \zeta\circ\eta(n)\;\vee\;QED]$.
Find $s_0$ such that $s_0 \notin D_\beta$. Now define $\rho$ such that, for each $n$, if $\forall i\le n[\zeta\circ\eta(i) \in Below(D_\alpha)]$, then $\zeta\circ \eta(n)\sqsubset\rho(n)$ and $\rho(n)\in D_\beta$, and, if not, then $\rho(n) = s_0$. 
Using (iii), find $n$ such that either $\neg \bigl(\rho(n+1)<_{lex}\rho(n)\bigr)$ or $\rho(n)\notin D_\beta$. Conclude that {\it either} $\neg \bigl(\zeta\circ\eta(n+1) <_{lex} \zeta\circ\eta(n)\bigr)$ and $\exists n [\neg \bigl(\zeta(n+1) <_{lex} \zeta(n)\bigr)]$ {\it or} $\exists i\le n+1[\zeta\circ \eta(i)\notin Below(D_\alpha)]$, so, in any case, $QED$.  \end{proof}
 One may introduce some of the usual countable ordinals as elements $\beta$ of $\omega^\omega$ with the property: $D_\beta$ is a thin bar in $\omega^\omega$. We sketch how this could be done.

\begin{definition}\label{D:ordinals}
Define ${\bm \varphi}:2^\omega\rightarrow 2^\omega$  such that, for every $\alpha$, $({\bm \varphi}|\alpha)(\langle\;\rangle) = 0$, and, for every $n$, for every $t$, $({\bm \varphi}|\alpha)(\langle n \rangle \ast t) \neq 0$ if and only if  $\exists s[\mathit{length}(s) = n \;\wedge\;\forall i < n[\alpha\bigl(s(i)\bigr) \neq 0] \;\wedge\; t = s(0)\ast s(1) \ast \ldots \ast s(n-1)]$. 

 $\tilde{\bm \omega}$ is the element of $2^\omega$ satisfying 
 $\forall s[\tilde{\bm \omega}(s) = 1\leftrightarrow \exists n[s=\langle n \rangle]]$.
  
  $ \tilde{\bm \varepsilon}_0$ is the element of $2^\omega$   such that $\tilde{\bm \varepsilon}_0(\langle \;\rangle) = 0$ and $(\tilde{\bm \varepsilon}_0)^0 = \mathbf{\omega}$ and $\forall n[(\tilde{\bm \varepsilon}_0)^{n+1} = {\bm \varphi}|(\tilde{\bm \varepsilon}_0)^n]$. \end{definition}

\begin{theorem}\label{T:epsilon}$\mathsf{BIM}$ proves the following. \begin{enumerate}[\upshape(i)]\item $\forall \alpha[Thinbar_{\omega^\omega}(D_\alpha) \rightarrow Thinbar_{\omega^\omega}(D_{{\bm \varphi}|\alpha})]$. \item  $Thinbar(D_{\tilde{\bm  \omega}}) \;\wedge\;Thinbar(D_{\tilde{\bm \varepsilon}_0})$. \item $ \mathbf{OI}(2^\omega) \rightarrow \;\;<_{lex} $ well-orders $D_{\tilde{\bm \varepsilon}_0}$. \end{enumerate}\end{theorem}

\begin{proof} The proof of (i) and (ii) is left to the reader. For (iii), use Theorem \ref{T:dedKB1}. \end{proof}

Theorem \ref{T:epsilon} is a sharpening of a result obtained by W.~Howard and G.~Kreisel, see \cite{howard-kreisel}, Appendix 2,  to the effect:  $\mathsf{BIM}\vdash \mathbf{BI}\rightarrow  \;\;<_{lex}$ well-orders $D_{\tilde{\bm \varepsilon}_0}$. Using the fact: $<_{lex}$ well-orders $D_{\tilde{\bm \varepsilon}_0}$, G.~Gentzen was able to prove the consistency of arithmetic, that is, the consistency of $\mathsf{BIM}$.  A.S. Troelstra, see \cite{troelstra},  proved, making use of techniques developed for eliminating choice sequences, that
     the Fan Theorem {\bf FT} is conservative over Heyting Arithmetic, that is, over $\mathsf{BIM}$. Using G\"odel's Second Incompleteness Theorem we obtain:

\begin{corollary}\label{Cor:incomplete}$\mathsf{BIM}\nvdash \;\;<_{lex}$ well-orders $D_{\tilde{\bm \varepsilon}_0}$ and $\mathsf{BIM}+\mathbf{FT}\nvdash \mathbf{OI}(2^\omega)$. \end{corollary}
  \section{The Almost-Fan Theorem and the Approximate-Fan Theorem}\label{S:sft?}

  We need a statement stronger than $\mathbf{FT}$, that plays, in the intuitionistic context of $\mathsf{BIM}$, a r\^ole comparable to the r\^ole fulfilled by $\mathbf{KL}$ in the classical context of $\mathsf{RCA_0}$.  $\mathbf{FT^+_{ext}}$ probably is not a good candidate, see Subsection \ref{SS:ftwkl}.

   \subsection{Notions of finiteness}\label{SS:notionfinite}    \begin{definition}\label{D:almostfinite} For each $n$, $\omega^n:=\{s\mid length(s)=n\}$
    and $[\omega]^n:=\{s\in \omega^n \mid \forall i[i+1 <n  \rightarrow s(i)< s(i+1)]\}$ \footnote{ See Subsection \ref{SS:inc}.}.
    
    For each $\alpha$,  $D_\alpha := \{n \in \omega\mid \alpha(n) \neq 0\}$. $D_\alpha$ 
     is  \emph{finite} if and only if  $\exists n\forall m\ge n[\alpha(m) = 0]$. For each $n$, $D_\alpha$ \emph{has at most $n$ members} if and only if $\forall s \in [\omega]^{n+1} \exists i \le n[ s(i) \notin D_{\alpha}]$.
        $D_\alpha$ is \emph{bounded-in-number}  if and only if there exists $n$ such that $D_\alpha$ has at most $n$ members. 
        $D_{\alpha}$ is \emph{almost-finite} if and only if $\forall \zeta \in [\omega]^\omega\exists n[\zeta(n) \notin D_\alpha]]$. \end{definition}

       Obviously, for each $\alpha$,  if $D_\alpha$ is finite, then $D_\alpha$ is bounded-in-number, and, if $D_\alpha$ is bounded-in-number, then $D_\alpha$ is almost-finite. The converse statements are not true, see \cite{veldman1995}.
       Almost-finite subsets of $\omega$ are  also studied in \cite{veldman1999}, \cite{veldman2005} and \cite{veldman2022}.

     $B\subseteq \omega$ is called a \emph{decidable subset of $\omega$} if and only if $\exists \alpha[B=D_\alpha]$. 
      
       \begin{lemma}\label{L:almostfinite} One may prove in $\mathsf{BIM}$: \begin{enumerate}[\upshape (i)]\item For all decidable subsets $A,B$ of $\omega$, if $A\subseteq B$ and $B$ is almost-finite, then also $A$ is almost-finite.  \item For all decidable subsets $A,B$ of $\omega$,  if $A,B$ are almost-finite, then  $A \cup B$ is almost-finite.
   \item For every $k$, for all decidable subsets $B_0, B_1, \dots, B_k$ of $\omega$, if, for each $n\le k$, $B_n$ is almost-finite, then $\bigcup\limits_{ n\le k} B_n$ is almost-finite.   \item For every infinite sequence $A, B_0, B_1, \ldots$ of decidable subsets of $\omega$, if  $A=\bigcup\limits_{n\in \omega}B_n$  and,  for each $n$, $B_n$ is almost-finite, and  $\forall \zeta \in [\omega]^\omega \exists n[B_n = \emptyset]$, then  $A$ is almost-finite. \end{enumerate} \end{lemma}
         \begin{proof} (i) The proof is left to the reader.
         
         \smallskip (ii) Let $A,B$ be given decidable and almost-finite subsets of $\omega$. Assume $\zeta \in [\omega]^\omega$. Define $\eta$ as follows, by induction. $\eta(0) := \mu p[\zeta(p) \notin A]$ and, for each $n$, $\eta(n+1) := \mu p[p>\eta(n) \;\wedge\; \zeta(p) \notin A]$. Note that $\zeta  \in [\omega]^\omega$ and $\forall n[\zeta \circ \eta (n) \notin A]$. Find $p$ such that $\zeta \circ \eta (p) \notin B$. Define $q:= \eta(p)$ and note that $\zeta(q) \notin A \cup B$. We thus see that $\forall \zeta \in [\omega]^\omega\exists q[\zeta(q) \notin A \cup B]$, i.e. $A\cup B$ is almost-finite.   
         
         \smallskip (iii) Use (ii) and induction.
         
         \smallskip (iv) Let $A, B_0, B_1, \ldots$ be an infinite sequence of decidable  subsets of $\omega$ such that, for each $n$, $B_n$ is almost-finite and $A=\bigcup\limits_{n\in \omega}B_n$  and  $\forall \zeta \in [\omega]^\omega \exists n[B_n = \emptyset]$. Assume $\zeta \in [\omega]^\omega$. Using (ii), define $\eta$  as follows, by induction. $\eta(0) := \mu p[\zeta(p) \notin B_0]$, and, for each $n$,   $\eta(n+1):= \mu p[p > \eta(n) \;\wedge\; \forall i\le n+1[\zeta(p) \notin B_i]]$.  Note that $\zeta\circ \eta \in [\omega]^\omega$ and $\forall n \forall i\le n[\zeta\circ \eta(n) \notin B_i]$. Define $\varepsilon$ as follows. For each $n$, if $\zeta\circ \eta(n) \notin A$, $\varepsilon(n) = n+1$, and, if $\zeta\circ\eta(n) \in A$, then $\varepsilon(n) :=\mu p[\zeta\circ \eta(n) \in B_p]$.  Note that $\forall n[\varepsilon(n) >n \;\wedge\;\bigl( \zeta \circ \eta(n) \in A \rightarrow \zeta\circ \eta(n) \in B_{\varepsilon(n)}\bigr)]$. Define $\lambda$ as follows. $\lambda(0) = \varepsilon(0)$ and, for each $n$, $\lambda(n+1) := \mu p[\varepsilon(p)>\varepsilon\circ\lambda(n)]$. Note that $\varepsilon \circ \lambda \in [\omega]^\omega$. Find $n$ such that $B_{\varepsilon\circ\lambda(n)} = \emptyset$. Define $q:= \eta\circ \lambda(n)$ and note that $\zeta(q) \notin A$. We thus see that $\forall \zeta \in [\omega]^\omega\exists q[\zeta(q) \notin A]$, i.e. $A$ is almost-finite. \end{proof} 
    \subsection{Fans, approximate fans and almost-fans}\label{SS:appalmostfans}
    \begin{definition}\label{D:fanapproximatealmostfan} 
  $\beta\in \omega^\omega$ is an \emph{approximate-fan-law}, notation: $Appfan(\beta)$,  if and only if $Spr(\beta)$ and, for each $n$, the set $\{t \in \omega^n \mid \beta(t) = 0 \}$ is bounded-in-number. 
  
  $\beta$ is an \emph{explicit} approximate-fan-law, notation: $Appfan^+(\beta)$,  if and only if $Spr(\beta)$ and there exists $\gamma$ such that   for each $n$, the set $\{t\in \omega^n\mid\beta(t) = 0 \}$ has at most $\gamma(n) $ members. 
  
  $\beta$ is an \emph{almost-fan-law}, notation: $Almfan(\beta)$,  if and only if $Spr(\beta)$ and, for each $s$, the set $\{n\mid\beta(s\ast\langle n \rangle) = 0\}$ is almost-finite.
  
   $\mathcal{F}\subseteq \omega^\omega$ is an \emph{(explicit) approximate fan/almost-fan}  if and only if there exists an (explicit) approximate-fan-law/almost-fan-law $\beta$ such that $\mathcal{F}=\mathcal{F}_\beta$.\end{definition}

\textbf{Weak}-$\mathbf{\Pi}^0_1$-$\mathbf{AC}_{0,0}$, see Subsection \ref{SSS:AC00}, does not seem to be sufficient for proving that every approximate-fan-law is an explicit approximate-fan-law, but $\mathbf{\Pi}^0_1$-$\mathbf{AC}_{0,0}$ clearly is sufficient. 

Note that every  fan is an approximate fan, but that, conversely, it is not true that every approximate fan is a fan.

    \begin{theorem}\label{T:almostfans}  A spread-law $\beta$ is an almost-fan-law if and only if, for each $n$, the set $\{t\in \omega^n\mid \beta(t)=0\}$ is almost-finite. \end{theorem}
    
   \begin{proof} First, let $\beta$ be an almost-fan-law. We prove by induction that, for each $n$, the set $\{t\in \omega^n\mid \beta(t)=0\}$ is almost-finite. Obviously, the set $\{t\in \omega^0\mid \beta(t)=0\}$ either is empty or coincides with $\{\langle\;\rangle\}$ and so is almost-finite. Now let $n$ be given such that the set $\{t\in \omega^n\mid \beta(t)=0\}$ is almost-finite. Note that, for each $t$ in $\omega^n$ such that $\beta(t)=0$, the set $\{s\in \omega^{n+1}\mid \beta(s)=0\;\wedge\;t\sqsubset s\}$ is almost-finite, like the set $\{n\mid \beta(t\ast\langle n\rangle)=0\}$.  Note that the set $\{s\in \omega^{n+1}\mid \beta(s)=0\} = \bigcup \{\{s\in \omega^{n+1}\mid t\sqsubset s \;\wedge\; \beta(s)=0\}\mid t\in\omega^n\;\wedge\;\beta(t)=0\}$. Using Lemma \ref{L:almostfinite}(iv), conclude that the set $\{s\in \omega^{n+1}\mid \beta(s)=0\}$ is almost-finite. 
   
   \smallskip As to the converse, let $\beta$ be a spread-law such that, for each $n$, the set $\{t\in \omega^n\mid \beta(t)=0\}$ is almost-finite. Let $t,n$ be given such that $t\in \omega^n$ and $\beta(t)=0$. Note that the set $\{s \in \omega^{n+1} \mid  \beta(s)=0\;\wedge\; t\sqsubset s\}$ is a subset of the set $\{s \in \omega^{n+1}\mid \beta(s)=0\}$ and thus, by Lemma \ref{L:almostfinite})(i), almost-finite. Conclude that the set $\{k\mid \beta(t\ast\langle k \rangle=0\}$ is also almost-finite. \end{proof}

     One may ask if every spread-law $\beta$ satisfying the condition that, for each $s$, the set $\{n\mid \beta(s\ast\langle n \rangle)=0\}$ is bounded-in-number, is an approximate-fan-law.   The answer is no: assuming this statement we prove $\mathbf{LPO}$. 
     
      Let $\alpha$ be given. Define $\beta$ be such that, for each $s$, $\beta(s)=0$ if and only if  \begin{enumerate}[\upshape (i)]  \item if $length(s)>0$, then   $s(0)=0$ or $s(0)=\mu n[\alpha(n)\neq 0]$, and \item if $length(s)>1$ then  $s(0)\le s(1)\le 2\cdot s(0)$, and \item if $length(s)>2$, then $\forall j>1[j<length(s)\rightarrow s(j)=0]$. \end{enumerate} Note that $\beta$ is a spread-law, and that, for each $s$, if $\beta(s)=0$, then either $s=\langle\;\rangle$ and the set $\{n\mid \beta(s\ast \langle n \rangle)=0\}$ has at most 2 members or $length(s)=1$ and   the set $\{n\mid \beta(s\ast \langle n \rangle)=0\}$ has $s(0)+1$ members or $length(s)>1$ and  the set $\{n\mid \beta(s\ast \langle n \rangle)=0\}$ has $1$ member. By our assumption,  $\beta$ is an approximate-fan-law. Find $k$ such that the set $\{s\in \omega^2\mid \beta(s)=0\}$ has at most $k$ members. Conclude that, for each $m$, if $\beta(\langle m \rangle)=0$, then $m<k$. Conclude that, for all $m$, if $m=\mu n[\alpha(n)\neq 0]$, then $m<k$.  Conclude that $\exists n[\alpha(n) \neq 0]$ if and only if $\exists n<k[\alpha(n)\neq 0]$. Conclude that $\exists n[\alpha(n)\neq 0]\;\vee\;\forall n[\alpha(n)=0]$.
      
      We thus see that, for each $\alpha$, $\exists n[\alpha(n)\neq 0]\;\vee\;\forall n[\alpha(n)=0]$, i.e. $\mathbf{LPO}$. 
     
     \smallskip Note that every approximate fan is an almost-fan, but that,  conversely, it is not true that every almost-fan is an approximate fan.
       
\begin{definition}        The following statement is called the \emph{Almost-fan Theorem}:
   \[\mathbf{AlmFT}:\forall \beta \forall \alpha[\bigl(Almfan(\beta) \;\wedge\; Thinbar_{\mathcal{F}_\beta}(D_\alpha) \;\wedge\; \forall s[s \in D_\alpha \rightarrow \beta(s) = 0]\bigr)\rightarrow \]\[ \forall \zeta \in [\omega]^\omega\exists n[\zeta(n) \notin D_\alpha]],\]
    
    i.e. \textit{in an almost-fan, every thin bar is almost-finite}.  
    
    \smallskip
      We also introduce the \emph{Approximate-fan Theorem}:
   \[\mathbf{AppFT}:\forall \beta \forall \alpha[\bigl(Appfan^+(\beta) \;\wedge\; Thinbar_{\mathcal{F}_\beta}(D_\alpha) \;\wedge\; \forall s[s \in D_\alpha \rightarrow \beta(s) = 0]\bigr)\rightarrow \]\[ \forall \zeta \in [\omega]^\omega\exists n[\zeta(n) \notin D_\alpha]],\]
    
    i.e. \textit{in an explicit approximate fan, every thin bar is almost-finite}.\end{definition}

    Note that $\mathsf{BIM}\vdash \mathbf{AlmFT} \rightarrow \mathbf{AppFT}.$ One may also see that $\mathbf{AppFT}$ extends
 $\mathbf{FT}$, see Theorem \ref{T:ftext}(ii). 
   
    $\mathbf{AlmFT}$ has been introduced in \cite{veldman2001a} and \cite{veldman2001b}. It was shown in these papers that $\mathbf{AlmFT}$ is a consequence of Brouwer's Thesis on Bars in $\omega^\omega$ and that $\mathbf{AlmFT}$ implies  $\mathbf{OI}([0,1])$.
      As we shall see, also $\mathbf{AppFT}$ implies $\mathbf{OI}([0,1])$. $\mathbf{AppFT}$ seems to be a little bit weaker than $\mathbf{AlmFT}$.   
      \subsection{Proving $\mathbf{AlmFT}$} \begin{definition}\label{D:co-analytic} For each $\beta$,  $CA_\beta:=\{s|\forall \gamma \exists n[\beta ( s, \overline  \gamma n)\neq 0]\}.$ $X\subseteq \omega$ is  \emph{co-analytic} or $\mathbf{\Pi}^1_1$ if and only if $\exists \beta[X=CA_\beta].$ 
      
The  \emph{Principle of Induction on co-analytic bars in Baire space} is the following statement:

$\mathbf{\Pi}_1^1$-$\mathbf{BI}:\;\;\forall \beta[(Bar_{\omega^\omega}(CA_\beta) \;\wedge\; \forall s[s \in CA_\beta \leftrightarrow \forall i[s\ast\langle i \rangle \in CA_\beta]]) \rightarrow 0\in CA_\beta]$.\end{definition}
\begin{theorem}\label{T:bipi11sigma01}$\mathsf{BIM}\vdash \mathbf{\Pi}_1^1$-$\mathbf{BI} \rightarrow \mathbf{\Sigma}^0_1$-$\mathbf{BI}$.\end{theorem}
\begin{proof} Note that every enumerable subset of $\omega$ is co-analytic as $\forall \beta \exists \delta[E_\beta = CA_\delta]$. One proves the latter fact as follows. Given $\beta$, define $\delta$ such that, for each $a$, for each $s$, $\delta(s, a) \neq 0$ if and only if $\gamma\bigl(\mathit{length}(a)\bigr) = s+1$.\end{proof}

\begin{theorem} $\mathsf{BIM}\vdash \mathbf{\Pi}^1_1$-$\mathbf{BI} \rightarrow \mathbf{AlmFT}$.\end{theorem}

\begin{proof} Let $\beta, \alpha$ be given such that $Almfan(\beta)$ and $Thinbar_{\mathcal{F}_\beta}(D_\alpha)$  and $\forall s[s \in D_\alpha\rightarrow\beta(s) =0]$. Define $B
:=\{s\mid\forall \zeta \in [\omega]^\omega\exists n[s\sqsubseteq \zeta(n) \rightarrow \zeta(n) \notin D_\alpha]\}$, i.e. $B$ is the set of all $s$ such that the set $\{t\mid s\sqsubseteq t \;\wedge\; t \in D_\alpha\}$ is almost-finite. Note that $B$ is a co-analytic subset of $\omega$. 

We claim that $B$ is a bar in  $\omega^\omega$. For let $\gamma$  be given. Define $\gamma^\ast$ such that, for each $n$, if $\beta(\overline{\gamma^\ast}n\ast\langle \gamma(n)\rangle) = 0$, then $\gamma^\ast(n) = \gamma(n)$, and, if not, then $\gamma^\ast(n)=\mu p[\beta(\overline{\gamma^\ast}n\ast\langle p \rangle) = 0]$. Note that $\gamma^\ast \in \mathcal{F}_\beta$. Find $m$ such that $\overline{\gamma^\ast}m \in D_\alpha$. {\it Either} $\overline \gamma m = \overline{\gamma^\ast}m$ and, as $D_\alpha$ is thin,  $\{t\mid\overline \gamma m \sqsubseteq t \;\wedge\; t\in D_\alpha\}=\{\overline \gamma m\}$ is  almost-finite, {\it or} $\overline \gamma m \neq \overline{\gamma^\ast}m$ and $t\mid\overline \gamma m \sqsubseteq t \;\wedge\; t \in D_\alpha\}=\emptyset$  is almost-finite.  In any case, $\overline \gamma m \in B$. 

We now prove that $B$ is inductive. Let $s$ be given such that, for each $n$, $s\ast\langle n \rangle \in B$. Define $A := \{t\mid s\sqsubseteq t \;\wedge\; t \in D_\alpha\}$, and, for each $n$, $B_n :=\{t\mid s\ast\langle n \rangle \sqsubseteq t \;\wedge\; t \in D_\alpha\}$.   First assume that $\exists t\sqsubseteq s[t \in D_\alpha]$. Then $A$ has at most one element. Next, assume that $\neg\exists t\sqsubset s[t \in D_\alpha]$. Then $A=\bigcup\limits_{n \in \omega} B_n$, and $\forall \zeta \in [\omega]^\omega\exists n[\beta(s\ast\langle \zeta(n) \rangle)\neq 0]$. Note that $\forall n[\beta(s\ast\langle n \rangle) \neq 0 \rightarrow B_n = \emptyset]$, and, for each $n$, $B_n$ is almost-finite.  One may conclude, using Lemma \ref{L:almostfinite}(iv), that $A$ is almost-finite. In any case, $A$ is almost-finite, i.e. $s \in B$.  

Note that $B$ is also monotone. Applying $\mathbf{\Pi}^1_1$-$\mathbf{BI}$,  conclude that $\langle \;\rangle \in B$ i.e. $D_\alpha$ is almost-finite. \end{proof}

 \section{The Semi-approximate-fan Theorem and the contrapostion of the Bolzano-Weierstrass Theorem}

\subsection{The Semi-approximate-fan Theorem} 

\begin{definition}\label{D:boundedinnumberenumerable}Let $\gamma, n$ be given.  The set $E_\gamma=\{m\mid\exists p[\gamma(p)=m+1]\}$  \emph{has at most $n$ members} if and only if $\forall s \in [\omega]^{n+1} \exists i \le n\exists j \le n[\gamma\circ s(i) = 0 \;\vee\; \bigl(i<j \;\wedge\; \gamma\circ s(i) = \gamma\circ s(j)\bigr)]$. The set $E_\gamma$ is \emph{bounded-in-number} if and only if, for some $n$, the set $E_\gamma$ has at most $n$ members.\end{definition}
\begin{definition}\label{D:semiappfan}For every $\delta$, $\mathcal{S} \mathcal{F}_\delta:=\{\gamma\mid\forall n[\overline \gamma n \in E_\delta]\} =\{\gamma\mid\forall n \exists p[\delta(p) = \overline \gamma n + 1]\}.$
$\delta\in\omega^\omega$ is a \emph{semi-spreadlaw}, notation: $\mathit{Semispr}(\delta)$,  if and only if $\forall s[s \in E_\delta \leftrightarrow \exists n[s\ast\langle n \rangle \in E_\delta]]$. 
 $\mathcal{F}\subseteq\omega^\omega$ is  a \emph{semi-spread}  if and only if there exists a semi-spreadlaw $\delta$ such that $\mathcal{F} = \mathcal{SF}_\delta$.

 $\delta\in\omega^\omega$ is a \emph{semi-approximate-fan-law}, notation: $Semiappfan(\delta)$, if and only if $Semispr(\delta)$ and, for each $n$, the enumerable\footnote{$X\subseteq \omega$ is enumerable if and only if $\exists \gamma[X=E_\gamma]$.} set $\{s \in \omega^n\mid  s\in E_\delta \;\wedge\; \mathit{length}(s) = n\}$ is bounded-in-number. 
 $\delta$ is an \emph{explicit} semi-approximate-fan-law, notation: $\mathit{Semiappfan}^+(\delta)$, if and only if $\mathit{Semiappfan}(\delta)$ and there exists $\gamma$ such that, for each $n$, the  set $\{s \in \omega^n\mid s\in E_\delta \}$ has at most $\gamma(n)$ members. $\mathcal{F}\subseteq \omega^\omega$ is an \emph{(explicit) semi-approximate fan} if and only if there exists an (explicit) semi-approximate-fan-law $\delta$ such that $\mathcal{F}=\mathcal{SF}_\delta$. \end{definition} 
 
Semi-spreads are called  {\it closed-and-semilocated} subsets of $\omega^\omega$ in \cite{veldman2011b}.  It is shown in \cite{veldman2011b} that an inhabited subset $\mathcal{F}$ of $\omega^\omega$ is closed-and-semilocated if and only if $\mathcal{F}$ is \textit{closed-and-separable}, i.e. $\mathcal{F}$ is closed\footnote{A subset $\mathcal{F}$ of $\omega^\omega$ is \textit{closed} if and only if, for some $\beta$, $\mathcal{F}=\mathcal{F}_\beta=\{\alpha\mid\forall n[\beta(\overline\alpha n)=0]\}$, see Subsection \ref{SS:fans}.}    and $\exists \alpha\forall \gamma[\gamma \in \mathcal{F} \leftrightarrow \forall n \exists m[\overline \gamma n = \overline{\alpha^m}n]]$.

  Using $\mathbf{\Pi}^0_1$-$\mathbf{AC}_{0,0}$\footnote{See Subsubsection \ref{SSS:AC00}.} one may prove that every semi-approximate-fan-law is an explicit semi-approximate-fan-law.
  
\begin{definition} The following statement is the \emph{semi-approximate-fan theorem}:
 \[\mathbf{SemiappFT}:\forall \delta \forall \alpha[\bigl(Semiappfan^+(\delta) \;\wedge\; Thinbar_{\mathcal{SF}_\delta}(D_\alpha) \;\wedge\; D_\alpha\subseteq E_\delta\bigr)\rightarrow \]\[ \forall \zeta \in [\omega]^\omega\exists n[\zeta(n) \notin D_\alpha]],\]
    i.e. \textit{in an explicit  semi-approximate-fan, every decidable thin bar is almost-finite}. \end{definition}

    \begin{theorem}$\mathsf{BIM} \vdash \mathbf{AppFT} \rightarrow \mathbf{SemiappFT}$.\label{T:appsemiapp}

      \end{theorem} 
      
      \begin{proof} 
   Let $\delta, \alpha$ be given such that $Semiappfan(\delta)$  and $Thinbar_{\mathcal{SF}_\delta}(D_\alpha)$ and $ D_\alpha  \subseteq E_\delta$.  We have to prove that  $D_\alpha$ is almost-finite.
 We   introduce $\beta$ and $\eta$ such that $Appfan(\beta)$  and $\eta$ is a one-to-one mapping of the set $\{s\mid\beta(s) =0\}$ onto $E_\delta$, as follows.
 We define  
  $\beta(\langle \;\rangle) := 0$ and $\eta(\langle \;\rangle) := \langle \;\rangle$ and
  for each $s$,  for each $n$, we distinguish two cases. \begin{enumerate}[\upshape (i)] \item  $\beta(s) = 0$, and  there exists $p$ such that $n=\mu i[\delta(i) = \eta(s)\ast\langle p \rangle + 1]$. We define  $\beta(s\ast\langle n \rangle) = 0$ and $\eta(s\ast\langle n\rangle) = \eta(s)\ast\langle \delta(n)-1 \rangle$.
  \item  {\it Either} $\beta(s) \neq 0$ {\it or} $\beta(s) = 0$ and there is no $p$ such that $n=\mu i[\delta(i) = \eta(s)\ast\langle p \rangle + 1]$.  We define $\beta(s\ast\langle n \rangle) = 1$ and $\eta(s\ast\langle n \rangle) = \eta(s) \ast\langle 0 \rangle$.\end{enumerate}

 Note that, for every $s$, if $\beta(s) = 0$, then $\eta(s) \in E_\delta$. Note that, for every $t$ in $E_\delta$, there exists $s$  such that $\beta(s) = 0$ and  $\eta(s) = t$.
 Also note that $\forall t[t \in E_\delta \leftrightarrow \exists p[t\ast \langle p \rangle  \in E_\delta]]$ and conclude that $\forall s[\beta(s) = 0 \leftrightarrow \exists n[ \beta(s\ast n \rangle) = 0]]$.  
   Find $\varepsilon$ such that, for each $n$,  the set $E_\delta \cap\omega^n$ has at most $\varepsilon(n)$ elements and conclude that, for each $n$,   the set $\{t \in \omega^n\mid \beta(s) = 0\}$ has at most $\varepsilon(n)$ elements. Conclude that $\beta$ is an explicit approximate-fan-law.  
  Let $\theta$ be an element of $\mathcal{F}_\beta$. Define $\nu$ such that, for each $n$, $\overline \nu n = \eta(\overline \theta n)$. Note that $\nu$ belongs to $\mathcal{S} \mathcal{F}_\delta$ and find $n$ such that $\overline \nu n \in D_\alpha$. Conclude that $\overline \theta n \in D_{\alpha\circ \eta}$.
We thus see that $Bar_{\mathcal{F}_\beta}(D_{\alpha \circ \eta})$. Note that $D_{\alpha \circ \eta}$ is thin, and, 
 using $\mathbf{AppFT}$,  conclude that $D_{\alpha \circ \eta}$ is almost-finite. 
 We now prove that also $D_\alpha$ itself is almost-finite.
Let $\zeta$ be an element of $[\omega]^\omega$. Define  $\kappa$ such that, for each $n$, \textit{either} $\zeta(n) \in D_\alpha$ and $\kappa(n)$ is the least $s$  such that $\beta(s) = 0$ and $\eta(s) = \zeta(n)$, (and therefore: $\eta \circ \kappa (n) = \zeta(n)$), \textit{or}  $\zeta(n) \notin D_\alpha$  and $\kappa(n)$ is the least $s$  such that  $\beta(s) \neq 0$ and $\forall i < n[\kappa(i) < s]$. Note that $\kappa$ is well-defined, as there are infinitely many numbers $s$ such that $\beta(s) \neq 0$,  and note that $\kappa$ is one-to-one, as $\eta$ is a one-to-one function from the set $\{t\mid\beta(t) = 0\}
 $ to $E_\delta$. Define $\rho$ such that $\rho(0) = \kappa(0)$ and, for each $n$, $\rho(n+1)$ is  the least $j> \rho(n) $ such that $\kappa\bigl(\rho(n)\bigr) <  \kappa(j)$.   Note that $\kappa \circ \rho \in [\omega]^\omega$ and find   $n$ such that $\kappa \circ \rho(n) \notin D_{\alpha\circ \eta}$. Conclude that $\eta \circ \kappa \circ \rho(n) \notin D_\alpha$. If $\zeta\bigl(\rho(n)\bigr)\in D_\alpha$, then $\eta\circ\kappa\circ\rho(n) = \zeta \bigl(\rho(n)\bigr)$ and $\zeta\bigl(\rho(n)\bigr)\notin D_\alpha$.  Conclude that $\zeta\bigl(\rho(n)\bigr)\notin D_\alpha$. Conclude  that $\forall \zeta \in [
 \omega]^\omega \exists m[\zeta(m) \notin D_\alpha]$, i.e. $D_\alpha$ is almost-finite. 
  \end{proof} 
 The next Definition extends Definition \ref{D:strongbar}. \begin{definition}\label{D:strongbar2}     
       For every semi-spread-law $\delta$, for every $B\subseteq\omega$,  \emph{$B$ is s strong bar in $\mathcal{SF}_\delta$}, notation: $Strongbar_{\mathcal{SF}_\delta}(B)$,  if and only if $\forall \zeta \in [\omega]^\omega[\forall n[\zeta(n) \in E_\delta ] \rightarrow \exists m [\overline{\zeta(n)}m \in B]]$. 
  \end{definition}
   Lemma  \ref{L:appfanbelow} and Corollary \ref{C:almostfinitestrongbar} extend Lemma \ref{L:fankl1}. 
  \begin{lemma}\label{L:appfanbelow} $\mathsf{BIM}$ proves that,  for every semi-spread-law $\delta$, for every $\alpha$,
  
    if
  $Bar_{\mathcal{SF}_\delta}(D_\alpha)$ and  $D_\alpha$ is almost-finite, then $Strongbar_{\mathcal{SF}_\delta}(D_\alpha)$. \end{lemma}
  
  \begin{proof}  Let $\delta$, $\alpha$  be given such that $Semispr(\delta)$ and $Bar_{\mathcal{SF}_\delta}(D_\alpha)$ and $D_\alpha$ is almost-finite.   
      Assume that $\zeta\in[\omega]^\omega$ and $\forall n[\zeta(n)\in E_\delta]$. For each $n$, for each $s$ in $\omega^n\cap E_\delta$, define $s^+$ in $\omega^{n+1}\cap E_\delta$, as follows. Given $n,s$, find $k_0:=\mu k[\delta(k)>0 \;\wedge\; s\sqsubset \delta(k)-1]$ and define $s^+:=\overline{\delta(k_0)-1}(n+1)$. Define $\gamma$ such that, for each $n$,  $\zeta(n)\sqsubset \gamma^n$ and and, for each $p\ge length\bigl(\zeta(n))$, $\overline{\gamma^n}(p+1)=\bigl(\overline{\gamma^n}(p)\bigr)^+$. Note that, for each $n$, $\zeta(n)\sqsubset \gamma^n$ and $\gamma^n \in \mathcal{SF}_\delta$. Define $\eta$ such that, for each $n$, $\eta(n)=\overline{\gamma^n}t$ where $t=\mu q[\overline{\gamma^n}q\in D_\alpha]$.  Note that, for all $m,n$ either $\eta(m)=\eta(n)$ or $\eta(m) \perp \eta(n)$. Note that, for each $n$, either $\zeta(n)\sqsubset \eta(n)$ or $\eta(n)\sqsubseteq \zeta(n)$. Define $QED\footnote{QED: `quod \emph{est} demonstrandum', `what \emph{still has} to be proven'.}:=\exists m\exists n[\overline{\zeta(n)}m\in D_\alpha]$ and note that $QED\leftrightarrow \exists n[\eta(n)\sqsubseteq \zeta(n)]$. Note that, for each $n$,   the set $\{k\mid  \zeta(k) 
      \sqsubseteq \eta(n)\}$ has at most $p:=length\bigl(\eta(n)\bigr)$ elements.    We define $\rho$ in $[\omega]^\omega$. We define  $\rho(0)=\eta(0)$ and, for each $n>0$, we distinguish two possibilities. \begin{enumerate}[\upshape (i)] \item  There exists $s$ in $[\omega]^n$ such that $\overline \rho n = \eta \circ s$ and $\forall i<n[\zeta\circ s(i)\sqsubset \eta\circ s(i)]$. We  calculate $N:=\sum_{i<n}length\bigl(\rho(i)\bigr)$ and $p:=s(n-1)$ and determine  $k<N$ such that {\it either}, for all $i<n$,  $\eta(i)\perp \eta(p+k+1)$, {\it or} $\eta(p+k+1)\sqsubseteq \zeta(p+k+1)$ and define $\rho(n) = \eta(p+k+1)$.
      
      \item If (i) does not apply, $n>0$ and we define $\rho(n)=\rho(n-1)\ast\overline{\underline 0}k$, where $k=\mu p\forall i<n[\rho(i)< \rho(n-1)\ast\overline{\underline 0}p]$.\end{enumerate} 
 
 Note that, for each $n$ either there exists $s$ in $[\omega]^n$ such that $\overline \rho n =  \eta\circ s$ or  $\exists i<n \exists  j[\zeta(j)\sqsubseteq \eta(j)=\rho(i)]$ and $QED$. Note that $\forall m\forall n[m<n\rightarrow \rho(m)\neq \rho(n)]$ and find $\theta$ in $[\omega]^\omega$ such that $\rho\circ\theta \in [\omega]^\omega$. Using the fact that $D_\alpha$ is almost-finite, find $n$ such that $\rho\circ \theta(n)\notin D_\alpha$. Find $p$ such that $\rho\circ\theta(n)\sqsubseteq\zeta(p)$ and conclude that $\zeta(p)\sqsubseteq \eta(p)\in D_\alpha$. Conclude $QED$.   
 
 We thus see that $\forall\zeta\in[\omega]^\omega[\forall n[\zeta(n) \in E_\delta]\rightarrow \exists m\exists n[\overline{\zeta(n)}m \in D_\alpha]$, i.e.
 
$Strongbar_{\mathcal{SF}_\delta}(D_\alpha)$. \end{proof}

\begin{corollary}\label{C:almostfinitestrongbar}$\mathsf{BIM}$ proves that the following two statements are equivalent. \begin{enumerate}[\upshape (i)] \item   $\mathbf{AppFT}$. \item For all $\beta , \alpha$, if $Appfan^+(\beta)$ and $\forall s[s\in D_\alpha\rightarrow \beta(s)=0]$ and $Bar_{\mathcal{F}_\beta}(D_\alpha) $, then  $ Strongbar_{\mathcal{F}_\beta}(D_\alpha)$.\end{enumerate} \end{corollary}   

\begin{proof} (i)$\Rightarrow$(ii)  Let $\beta,\alpha$ be given such that $Appfan^+(\beta)$ and $\forall s[s\in D_\alpha\rightarrow \beta(s)=0]$ and $Bar_{\mathcal{F}_\beta}(D_\alpha)$. Using $\mathbf{AppFT}$, conclude that $D_\alpha$ is almost-finite. Note that there exists $\delta$ such that $\mathcal{F}_\beta=\mathcal{SF}_\delta$.  Using Lemma \ref{L:appfanbelow}, conclude 
 that $Strongbar_{\mathcal{F}_\beta}(D_\alpha)$.  
 
 \smallskip (ii)$\Rightarrow$(i).  Let $\beta,\alpha$ be given such that $Appfan^+(\beta)$ and  $\forall s[s\in D_\alpha\rightarrow \beta(s)=0]$ and $Thinbar_{\mathcal{F}_\beta}(D_\alpha)$.  Define $\gamma$ such that for all $t$, $\gamma(t)\neq 0$ if and only if $\exists s\exists n[t= s\ast\langle n\rangle\;\wedge\;s\in D_\alpha\;\wedge\; \beta(s\ast\langle n \rangle)=0 ]$.   Note that $Bar_{\mathcal{F}_\beta}(D_\gamma)$ and $\forall s[s\in D_\gamma\rightarrow \beta(s)=0]$. Conclude that $Strongbar_{\mathcal{F}_\beta}(D_\gamma)$. Let $\zeta$ in $[\omega]^\omega$ be given. Find $n,m$ such that $\overline{\zeta(n)}m\in D_\gamma$ and conclude that $\zeta(n)\notin D_\alpha$. Conclude that $\forall \zeta \in [\omega]^\omega \exists n[\zeta(n)\notin D_\alpha]$, i.e. $D_\alpha$ is almost-finite. \end{proof}
\subsection{The contraposition of the Bolzano-Weierstrass Theorem in $\omega^\omega$}   
  \begin{definition}   The \emph{contrapositive of the Bolzano-Weierstrass Theorem in $\omega^\omega$} is the following statement:  
  
  $\overleftarrow{\mathbf{BW}_{\omega^\omega}}$:
      For all $\gamma$, if $\forall \zeta \in [\omega]^\omega\exists n[\gamma\circ\zeta(n) \neq_n \gamma\circ\zeta(n+1)]$, then $\exists n[\gamma(n)\notin 2^{<\omega}]$. \end{definition}

\begin{definition}\label{D:finiteseqbinary2} Define\footnote{Note that ${\bm \tau}$ subtly differs from the function ${\bm \delta}$ given in Definition   \ref{D:finiteseqbinary}.} ${\bm \tau}$ such that ${\bm \tau}(\langle \;\rangle):= \langle \;\rangle$ and, for  all $s$, for all $n$,  ${\bm \tau}(s\ast\langle n\rangle) := {\bm \tau}(s) \ast\overline{\underline 0}(n+1)\ast\langle 1 \rangle$. \end{definition}
\begin{theorem}\label{T:semiappcontrbw} $\mathsf{BIM}\vdash \mathbf{SemiappFT} \rightarrow \overleftarrow{\mathbf{BW}_{\omega^\omega}}$. \end{theorem}

\begin{proof}                                                                                                                                                                                                                                                                                         Let ${\bm \tau}$ be as in Definition \ref{D:finiteseqbinary2}. Observe the following. \begin{enumerate}[\upshape(i)] \item $\forall s[{\bm \tau}(s) \in 2^{<\omega}]$ \item $\forall s\forall i[\bigl(i+1 <length(s) \;\wedge\; \bigl({\bm \tau}(s)\bigr)(i)=1\bigr)\rightarrow  \bigl({\bm \tau}(s)\bigr)(i+1)=0]$ and \item $\forall s \in  2^{<\omega}[\bigl(i+2<\mathit{length}\bigl({\bm \tau}(s)\bigr)\;\wedge\; \bigl({\bm \tau}(s)\bigr)(i)=\bigl({\bm \tau}(s)\bigr)(i+1)=0\bigr)\rightarrow \\\bigl({\bm \tau}(s)\bigr)(i+2)=1]$ and \item $\forall s \in 2^{<\omega}[\mathit{length}\bigl({\bm \tau}(s)\bigr) \le 3\cdot\mathit{length}(s)]$.\end{enumerate}
Let $\gamma$ be given such that $\forall \zeta \in [\omega]^\omega\exists n[\gamma\circ\zeta(n) \neq_n \gamma\circ\zeta(n+1)]$. 
Define $\delta$ such that, for each $n$, \textit{if} $n'\sqsubseteq {\bm \tau}\circ\gamma(n'')\ast\underline 1$,  then $\delta(n) =n' +1$, and, \textit{if not}, then $\delta(n) =0$.  
Note that $\forall s[s\in E_\delta\leftrightarrow \exists m  [s \sqsubseteq {\bm \tau}\circ\gamma(m)\ast\underline{ 1}]]$ and conclude that $\delta$ is a semi-spread-law. Note that  $\forall s[s \in E_\delta \leftrightarrow \exists n<2[s\ast\langle n \rangle \in E_\delta]]$ and that, for each $n$, the set $\{s\in \omega^n\mid s \in E_\delta\}$ has at most $2^n$ members, and conclude that  $\delta$ is an explicit semi-approximate-fan-law.   
Define $B:=\{s \in 2^{<\omega} \mid \exists i[i+2 < length(s) \;\wedge\;s(i)=s(i+1)=s(i+2)=1]\;\vee\;\exists i<length(s)[\gamma(i)\notin 2^{<\omega}]\}$. 
We  now prove that $B$ is a bar in $\mathcal{SF}_\delta$.

Let $\beta$ in $\mathcal{SF}_\delta$ be given. Define $\zeta$ as follows.
 For each $p$, having to define $\zeta(p)$, we distinguish two cases. 
 
 \textit{Case (1)}.  $\neg \exists i[i+2<p \;\wedge\; \beta(i)=\beta(i+1)=\beta(i+2)=1]$.  Find $k:=\mu n [\delta(n) = \overline \beta p +1]$ and define $\zeta(p) = k''$.  
   Note that $k'=\overline \beta(p) \sqsubseteq {\bm \tau}\bigl(\gamma\circ\zeta(p)\bigr)\ast\underline 1$.
  
   \textit{Case (2)}. $\exists i[i+2<p \;\wedge\; \beta(i)=\beta(i+1)=\beta(i+2)=1 ]$. Note that $p>0$ and define $\zeta(p) :=\zeta(p-1) +1$. 
  
This completes the definition of $\zeta$.     Note that $\forall p \exists q[\zeta (q)>p]$. 
   Define $\lambda$ such that $\lambda(0) =0$ and $\forall n[\lambda(n+1)= \mu q>\lambda(n)[\zeta(q)> \zeta\bigl(\lambda(n)\bigr)]$.  Consider $\eta:=\zeta \circ \lambda$   and note that $\eta \in [\omega]^\omega$.  Find $n$ such that $\gamma\circ\eta(3n) \neq_n\gamma\circ\eta(3n+3)$. Assume that $\gamma\circ\eta(3n) \in 2^{<\omega}$ and $\gamma\circ\eta(3n+3) \in 2^{<\omega}$ and that $\overline \beta\bigl(\lambda(3n+3)\bigr)\notin B$.   Note that $\overline \beta\bigl(\lambda(3n)\bigr)\sqsubseteq {\bm \tau}\bigl(\gamma \circ\eta(3n)\bigr)\ast\underline 1$
    and $\overline \beta\bigl(\lambda(3n)\bigr)\sqsubseteq \overline \beta\bigl(\lambda(3n+3)\bigr)\sqsubseteq {\bm \tau}\bigl(\gamma \circ\lambda(3n+3)\bigr)\ast\underline 1$ and conclude that ${\bm \tau}\bigl(\gamma \circ\eta(3n)\bigr)=_{\lambda(3n)} {\bm \tau}\bigl(\gamma \circ\eta(3n+3)\bigr)$.  As $\lambda(3n) \ge 3n$, also ${\bm \tau}\bigl(\gamma \circ\eta(3n)\bigr)=_{3n} {\bm \tau}\bigl(\gamma \circ\eta(3n+3)\bigr)$. As $\gamma\circ\eta(3n) \in 2^{<\omega}$ and $\gamma\circ\eta(3n+3) \in 2^{<\omega}$, also $\gamma \circ\eta(3n)=_n \gamma \circ\eta(3n+3)$. Contradiction. Conclude that {\it either} $\overline\beta\bigl(\lambda(3n+3)\bigr)\in B$ {\it or}  $\gamma\circ\eta(3n) \notin 2^{<\omega}$ {\it or} $\gamma\circ\eta(3n+3) \notin 2^{<\omega}$. In the latter two cases, $\overline \beta\bigl((\eta(3n+3)+1\bigr) \in B$.  
 We thus see that $\forall \beta \in \mathcal{SF}_\delta\exists n[\overline \beta n  \in B]$, i.e. $Bar_{\mathcal{SF}_\delta}(B)$.
Note that $\forall m \exists n >m [\gamma(n) > \gamma (m)]$. 
 Define $\eta^\ast$ in $[\omega]^\omega$ such that  $\forall n[\gamma\circ\eta^\ast(n+1)>\gamma \circ\eta^\ast(n)]$ and define: $\gamma^\ast:=\gamma\circ \eta^\ast$.
 Using $\mathbf{SemiappFT}$ and Lemma \ref{L:appfanbelow}, find $m,n$ such that: $\overline{{\bm \tau}\circ\gamma^\ast(n)}m \in B$.
 Conclude that $\exists n[\gamma(n) \notin 2^{<\omega}]$.
 
 We thus see that, for all $\gamma$, if $\forall \zeta \in [\omega]^\omega\exists n[\gamma\circ\zeta(n) \neq_n \gamma\circ\zeta(n+1)]$, then $\exists n[\gamma(n) \notin 2^{<\omega}]$, i.e. $\overleftarrow{\mathbf{BW_{\omega^\omega}}}$. 
\end{proof}    

\subsection{The contraposition of the Bolzano-Weierstrass Theorem in $\mathcal{R}$.}

\begin{definition} The following statement is a version of the classical \emph{Bolzano-Weierstrass Theorem}:

      $ \mathbf{BW}$: for all $\gamma$ in $\mathbb{Q}^\omega$,  \\if $\forall n[| \gamma(n)|\le_\mathbb{Q} 1_\mathbb{Q}]$, then $\exists \zeta \in [\omega]^\omega\forall n[|\gamma\circ\zeta(n+1) -_\mathbb{Q} \gamma\circ\zeta(n)| \le_\mathbb{Q} \frac{1}{2^n}]].$ 
      
       i.e. every infinite sequence $\gamma$ of rationals in $[-1,1]$  has a convergent subsequence:

The following  statement is the \emph{contrapositive  of the Bolzano-Weierstrass Theorem}:

$\overleftarrow{\mathbf{BW}}$:
{\it For all $\gamma$ in $\mathbb{Q}^\omega$,   \\if
$\forall \zeta \in [\omega]^\omega\exists n[|\gamma\circ\zeta(n+1) -_\mathbb{Q} \gamma\circ\zeta(n)| 
>_\mathbb{Q} \frac{1}{2^n}]$, then $\exists n[|1_\mathbb{Q}<_\mathbb{Q} |\gamma(n)| ].$} 

i.e. \textit{every infinite  sequence of rationals that positively fails to have a convergent subsequence leaves $[-1,1]$}.  \end{definition}\smallskip $\mathbf{BW}$   extends Dedekind's Theorem $\mathbf{Ded}$, see Subsection \ref{SS:dt}, and thus is constructively false.
\begin{theorem}\label{T:bwbwr} $\mathsf{BIM}\vdash \overleftarrow{\mathbf{BW}_{\omega^\omega}} \leftrightarrow \overleftarrow{\mathbf{BW}}$. \end{theorem}

\begin{proof} (i) Assume $\overleftarrow{\mathbf{BW}_{\omega^\omega}}$. 

Define ${\bm \delta_1}$ such that, for all $a$ in $ 2^{<\omega}$,  ${\bm \delta_1} (a) = \sum_{i<\mathit{length}(a)} \frac{a(i)}{2^{i+1}}$. 
Note that, for all $n$, for all $a,b$ in $2^{<\omega}$, if  $a=_n b$ then $|{\bm \delta_1}(a)-{\bm \delta_1}(b)|<\frac{1}{2^{n}} $. 

\smallskip Let $\gamma$ in $\mathbb{Q}^\omega$ be given such that  $\forall \zeta\in [\omega]^\omega\exists n[|\gamma\circ \zeta(n+1)-\gamma\circ\zeta(n)|>_\mathbb{Q} \frac{1}{2^n}]$. 
 We want to prove: $\exists n[1_\mathbb{Q}<_\mathbb{Q}|\gamma(n)|  ]$. 
 Define $\eta$ such that,  
for each $n$, if $\exists i\le n+1[1_\mathbb{Q}<_\mathbb{Q}|\gamma(i)| ]$, then $\eta(n)=\langle n +2\rangle$, and, if $\forall i\le n+1[|\gamma(i)| \le_\mathbb{Q} 1_\mathbb{Q}]$, then $\eta(n)$ is the least $a$ in $2^{<\omega}$ such that $|{\bm \delta_1}(a)-_\mathbb{Q} \gamma(n)| <_\mathbb{Q} \frac{1}{2^{n+2}} $.  
Assume that $\zeta \in [\omega]^\omega$. Find $n$ such that $|\gamma\circ\zeta(n+1) -_\mathbb{Q}\gamma\circ\zeta(n)|>_\mathbb{Q}\frac{1}{2^n}$ and distinguish two cases.
 \begin{enumerate}[\upshape (1)]
 \item  $\forall i \le \zeta(n+1)+1[|\gamma(i)| \le_\mathbb{Q} 1_\mathbb{Q}]$. 
Observe that $|{\bm \delta_1}\circ\gamma\circ\zeta(n+1) -_\mathbb{Q}{\bm \delta_1}\circ\gamma\circ\zeta(n)|>_\mathbb{Q} \\|\gamma\circ\zeta(n+1)-_\mathbb{Q} \gamma\circ \zeta(n)| -_\mathbb{Q} \frac{1}{2^{\zeta(n+1)+2}} -_\mathbb{Q} \frac{1}{2^{\zeta(n)+2}} >_\mathbb{Q} \frac{1}{2^n}-_\mathbb{Q}2\cdot \frac{1}{2^{n+2}}=\frac{1}{2^{n+1}}$. Conclude that $\eta\circ\zeta(n+1)\ne_{n+1}\eta\circ\zeta(n)$. 
\item $\exists i \le \zeta(n+1)+1[\gamma(i)<_\mathbb{Q} |\gamma(i)| ]$. Note that  $\eta\circ\zeta(n+1) =_\mathbb{Q} \langle\zeta(n+1)+2\rangle$ and $\eta\circ\zeta(n+2) =\langle \bigl(\zeta(n+2)+2\rangle$ and $\eta\circ\zeta(n+2) \neq_{1} \eta\circ\zeta(n+1) $. \end{enumerate}
 We thus see that $\forall \zeta \in [\omega]^\omega\exists n[ \eta\circ\zeta(n+1) \neq_{n+1} \eta\circ\zeta(n)]$.
Using Lemma \ref{L:ftconvprep2} and $\overleftarrow{\mathbf{BW}_{\omega^\omega}}$, we find $p$ such that $\eta(p)\notin 2^{<\omega}$. Conclude that $\eta(p)  =\langle p+2\rangle$ and $\exists n[1_\mathbb{Q}<_\mathbb{Q}|\gamma(n) |]$.

We thus see that, for each $\gamma$ in $\mathbb{Q}^\omega$, if $\gamma$ positively fails tio have a convergent subsequence, then  $\exists n[1_\mathbb{Q}<_\mathbb{Q}|\gamma(n) |]$, i.e. $\overleftarrow{\mathbf{BW}}$.

\smallskip  (ii)  Assume $\overleftarrow{\mathbf{BW}}$.

 Define ${\bm \delta_0}$ such that, for all $a$ in $ 2^{<\omega}$,  ${\bm \delta_0} (a) = \sum_{i<\mathit{length}(a)} \frac{2a(i)+1}{5^i}$. 
Note that, for all $n$, for all $a,b$ in $2^{<\omega}$, if   $a\neq_n b$, then $\frac{1}{5^n} <_\mathbb{Q} |{\bm \delta_0}(a)-_\mathbb{Q} {\bm \delta_0}(b)|$.

Let $\gamma$ be given such that  $\forall\zeta  \in [\omega]^\omega\exists n[\gamma\circ\zeta(n)\neq_n \gamma\circ\zeta(n+1)]$. 
Define $\eta$ in $\mathbb{Q}^\omega$ such that,  for each $p$, \textit{either} $\forall m \le p  [\gamma(m) \in 2^{<\omega}]$ and $\eta(p) = {\bm \delta_0}\bigl(\gamma(p)\bigr)$, \textit{or} $\exists m \le p[\gamma(m) \notin 2^{<\omega}]$ and $\eta(p) = (p+1)_\mathbb{Q}$. 
Assume $\zeta \in [\omega]^\omega$. Find $p$ such that $\gamma\circ\zeta(p) \neq_p 
 \gamma\circ\zeta(p+1)$ and distinguish two cases.
 \begin{enumerate}[\upshape (1)]
  \item  $\forall m \le \zeta(p+1)[\gamma(m) \in 2^{<\omega}]$. 
Conclude that $\eta\circ\zeta(p) = {\bm \delta_0}\bigl(\gamma\circ\zeta(p)\bigr)$ and $\eta\circ\zeta(p+1) = {\bm \delta_0}\bigl(\gamma\circ\zeta(p+1)\bigr)$ and $\frac{1}{5^p}<_\mathbb{Q} |\eta\circ\zeta(p) -_\mathbb{Q} \eta\circ\zeta(p+1)| $.

 \item $\exists m \le \zeta(p+1)[\gamma(m) \notin 2^{<\omega}]$. Then  $\eta\circ\zeta(p+1) =_\mathbb{Q} \bigl(\zeta(p+1)+1\bigr)_\mathbb{Q}$ and $\eta\circ\zeta(p+2) =_\mathbb{Q} \bigl(\zeta(p+2)+1\bigr)_\mathbb{Q}$ and $|\eta\circ\zeta(p+2) -_\mathbb{Q} \eta\circ\zeta(p+1)| \ge_\mathbb{Q} 1_\mathbb{Q}>_\mathbb{Q} \frac{1}{ 2^{p+1}}$. 
 \end{enumerate}
We thus see that $\forall \zeta \in [\omega]^\omega\exists p[ \frac{1}{5^{p+1}}<_\mathbb{Q} |\eta\circ\zeta(p)-_\mathbb{Q}  \eta\circ\zeta(p+1)|]$.
Using Lemma \ref{L:ftconvprep} and $\overleftarrow{\mathbf{BW}}$, we find $p$ such that $ 1_\mathbb{Q}<_\mathbb{Q} |\eta(p)|$. 
 Conclude: $\eta(p)  =_\mathbb{Q} (p+1)_\mathbb{Q}$ and $\exists m \le p [\gamma(m) \notin 2^{<\omega}]$.  
 
 We thus see that, for each $\gamma$, if $\gamma$ positively fails to have a $1$-convergent subsequence, then $\exists m [\gamma(m) \notin 2^{<\omega}]$, i.e.  $\overleftarrow{\mathbf{BW}_{\omega^\omega}}$.
\end{proof}

\subsection{Extending $\protect\overleftarrow{\mathbf{BW}_{\omega^\omega}}$}
\begin{definition}  The \emph{Extended contrapositive of the Bolzano-Weierstrass Theorem} is the following  statement:

   $(\overleftarrow{\mathbf{BW}_{\omega^\omega}})^+$: 
  For all $\beta,\gamma$, if $Appfan^+(\beta)$ and $\forall \zeta \in [\omega]^\omega\exists n[\gamma\circ\zeta(n) \neq_n \gamma\circ\zeta(n+1)]$, then  $\exists n[\beta\bigl(\gamma(n)\bigr) \neq 0].$ \end{definition}

$(\overleftarrow{\mathbf{BW}_{\omega^\omega}})^+$ generalizes $\overleftarrow{\mathbf{BW}_{\omega^\omega}}$ from Cantor space $2^\omega$ to any explicit approximate fan $\mathcal{F}_\beta$.
\begin{theorem}\label{T:bwbw+} $\mathsf{BIM}\vdash \overleftarrow{\mathbf{BW}_{\omega^\omega}} \rightarrow (\overleftarrow{\mathbf{BW}_{\omega^\omega}})^+$. \end{theorem}

\begin{proof} Let $\beta$ be given such that $Appfan^+(\beta)$. Find $\delta$  such that   for all $n$, the set  $\{s\in \omega^n\mid \beta(s) = 0\}$   has at most $\delta(n)$ elements.  
Define $\eta$ such that,  for each $n$, $\eta(n) =\sum\limits_{i\le n} \delta(i) +1$.  Define  ${\bm\psi}$ such that \begin{enumerate}[\upshape (i)]
 \item $\mathbf{{\bm \psi}}(\langle \; \rangle) =  \langle \;\rangle $, and, \item
 for every $t$,  $j$,  $k$, if $\beta(t\ast\langle j \rangle) =0$, and there are exactly $k$ numbers $l$ such that $l<j$ and $\beta(t\ast\langle l\rangle)=0$, then  $\mathbf{{\bm \psi}}(t\ast\langle j\rangle)=\mathbf{{\bm \psi}}(t)\ast \underline{\overline 0}k\ast\langle 1 \rangle$. \end{enumerate}
Note that,   for all $n$, for all $s,t$, if $\beta(s) = \beta(t) = 0$ and $s\neq_n t$ then ${\bm \psi}(s)\neq_{\eta(n)}{\bm \psi}(t)$. 
Let $\gamma$ be given such that $\forall\zeta  \in [\omega]^\omega\exists n[\gamma\circ\zeta(n) \neq _n \gamma\circ\zeta(n+1)]$. Define $\nu$ such that,  for each $p$, \textit{either} $\forall n \le p[\beta\bigl(\gamma (p)\bigr) =0]$ and $\nu(p) ={\bm \psi}\bigl(\gamma(p)\bigr)$,  \textit{or} $\exists  n \le p[\beta\bigl(\gamma(n)\bigr) \neq 0 ]$ and $\nu(p) = \langle p+2\rangle$. 
Let $\zeta$ be an element of $[\omega]^\omega$. Find $p$ such that $\gamma\circ\zeta(p) \neq _p
 \gamma\circ\zeta(p+1)$. If   $\forall m \le \zeta(p+1)[\beta\bigl(\gamma(m)\bigr)= 0]$, then  
 $\nu\circ\zeta(p)={\bm \psi}\bigl(\gamma\circ\zeta(p)\bigr) \neq_{\eta(p)} {\bm \psi}\bigl(\gamma\circ\zeta(p+1)\bigr)=\nu\circ\zeta(p+1)$. If
$\exists m \le \zeta(p+1)[\beta\bigl(\gamma(m)\bigr)\neq 0]$, then   $\nu\circ\zeta(p+1) =\langle \zeta(p+1)+2\rangle $ and $\nu\circ\zeta(p+2) = \langle \zeta(p+2)+2\rangle$ and $\nu\circ\zeta(p+2) \neq_1 \nu\circ\zeta(p+1)$. 
 We thus see that $\forall \zeta \in [\omega]^\omega\exists p[ \nu\circ\zeta(p+1) \neq_{\eta(p)} \nu\circ\zeta(p)]$. 
Applying $\overleftarrow{\mathbf{BW}_{\omega^\omega}}$, we find $p$ such that $\nu(p) \notin 2^{<\omega}$. Conclude that $\nu(p) = \langle p+2\rangle$ and $\exists m \le p[\beta\bigl(\gamma(m)\bigr) \neq 0]$.

We thus see that, for all $\beta,\gamma$, if $Appfan^+(\beta)$ and $\forall \zeta \in [\omega]^\omega\exists n[\gamma\circ\zeta(n) \neq_n \gamma\circ\zeta(n+1)]$, then  $\exists n[\beta\bigl(\gamma(n)\bigr) \neq 0]]$, i.e. $(\overleftarrow{\mathbf{BW}_{\omega^\omega}})^+$. 
\end{proof}

\begin{theorem}\label{T:bw+app*} $\mathsf{BIM}\vdash (\overleftarrow{\mathbf{BW}_{\omega^\omega}})^+ \rightarrow \mathbf{AppFT}$.\end{theorem}
  \begin{proof}  Assume $(\overleftarrow{\mathbf{BW}_{\omega^\omega}})^+$. 
 Let $\beta$, $\alpha$  be given such that $Appfan^+(\beta)$ and $\forall s[ s\in D_\alpha\rightarrow\beta(s)=0]$ and $Bar_{\mathcal{F}_\beta}(D_\alpha)$.  Find $\delta$ such that, for each $n$, the set $\{s\in \omega^n\mid \beta(s) = 0\}$ has at most $\delta(n)$ members.  Let $\zeta$ in $[\omega]^\omega$ be given such that $\zeta(0)\in D_\alpha$.   Define $\zeta^\ast$, $\zeta^\dag$ such that $\zeta^\ast(0)=\zeta^\dag(0)=\zeta(0)$ and, for each $n$, if $\forall i\le n+1[\zeta(i) \in D_\alpha]$, then $\zeta^\ast(n+1)=\zeta^\dag(n+1)=\zeta(n+1)$, and, if not, then $\zeta^\ast(n+1)=\zeta^\ast(n)\ast\langle p\rangle$ where $p=\mu q[\beta(\zeta^\ast (n)\ast\langle q \rangle)= 0]$, and $\zeta^\dag(n+1)=\langle p\rangle$ where $p=\mu q[\langle q\rangle >\zeta^\dag(n)\;\wedge\; \beta(\langle q \rangle)\neq 0]$. 
 Note that both $\zeta^\ast, \zeta^\dag$ are in $[\omega]^\omega$ and, for each $n$,  $\beta\bigl(\zeta^\ast(n)\bigr)=0$. Also note that, for each $n$, there exists $m\le \sum\limits_{i\le n }\delta(i)$ such that $\mathit{length}\bigl(\zeta^\ast(m)\bigr) > n$. Define $\varepsilon$ in $[\omega]^\omega$ such that $\varepsilon(0)=0$ and, for each $n$, $\varepsilon(n+1) = \mu p>\varepsilon(n)[ \mathit{length}\bigl(\zeta^\ast(p)\bigr)>\mathit{length}\bigl(\zeta^{\ast}(\varepsilon(n))\bigr)]$ and define $\zeta^{\ast\ast}=\zeta^\ast\circ \varepsilon$. Note that $\forall n[\mathit{length}\bigl(\zeta^{\ast\ast}(n)\bigr)\ge n]$.
       Let $\eta$ in $[\omega]^\omega$ be given. Define $\theta$ such that, for each $n$, \textit{if}  $\forall i\le n[\zeta^{\ast\ast}\circ\eta(i)=_i\zeta^{\ast\ast}\circ\eta(i+1)]$, then $\theta(n) = \bigl(\zeta^{\ast\ast}\circ\eta(n+1)\bigr)(n)$, and, if not, then $\theta(n)$ is the least $i$ such that $\beta(\overline \theta n \ast \langle i \rangle)=0$.  Note that  $\theta \in \mathcal{F}_\beta$, and find $n$ such that $\overline \theta n \in D_\alpha$. Note that $\zeta^{\ast\ast}\circ \eta(n+1)=\zeta^\ast\circ\varepsilon\circ\eta(n+1)$ and  distinguish two cases. \begin{enumerate}[\upshape (a)]
    \item $\exists i\le \varepsilon\circ \eta(n+1)[\zeta(i)\notin D_\alpha]$. 
    
    Find $p:= \mu i[\zeta(i)\notin D_\alpha]$ and note that $\zeta^\dag\circ\varepsilon\circ\eta(p+1)\neq_{p+1} \zeta^\dag \circ\varepsilon\circ\eta(p+2)$.   
     \item  $\forall i\le \varepsilon\circ \eta(n+1)[\zeta(i)\in D_\alpha]$.
    
    Then $\forall i\le\varepsilon\circ \eta(n+1)[\zeta(i)=\zeta^\ast(i)=\zeta^\dag(i)]$.
     Assume  that $\forall i\le n[\zeta^{\ast\ast}\circ\eta(i) =_i \zeta^{\ast\ast}\circ\eta(i+1)]$. Then $\overline \theta n = \overline{\zeta^{\ast\ast}\circ\eta(n)}n  \sqsubset \zeta^{\ast\ast}\circ\eta(n+1)$. Note  that $D_\alpha$ is thin and  conclude  that $\zeta^{\ast\ast}\circ \eta(n+1)\notin D_\alpha$.  Contradiction. Conclude that $\exists i\le n[\zeta^{\ast\ast}\circ\eta(i) \neq_i \zeta^{\ast\ast}\circ\eta(i+1)]$ and that $\exists i\le n[\zeta^\dag\circ\varepsilon\circ\eta(i) \neq_i \zeta^\dag\circ\varepsilon\circ\eta(i+1)]$.
      \end{enumerate}

 We thus see that $\forall \eta \in [\omega]^\omega \exists n[\zeta^\dag\circ\varepsilon\circ\eta(n) \neq_n \zeta^\dag\circ\varepsilon\circ\eta(n+1)]$.
 Using $(\overleftarrow{\mathbf{BW}_{\omega^\omega}})^+$, we find $n$ such that $\beta\bigl(\zeta^\dag\circ\varepsilon
 (n) \bigr)\neq 0$.
 Conclude that there exists $i \le \varepsilon(n)$ such that $\zeta(i)\notin D_\alpha$. 
We thus see that $\forall \zeta \in [\omega]^\omega\exists n[ \zeta(n) \notin D_\alpha]$, i.e. $D_\alpha$ is almost-finite.   We thus see that, in every explicit approximate fan, every thin bar is almost-finite, i.e.  $\mathbf{AppFT}$.\end{proof}

\begin{corollary}\label{C:sumup3}
$\;$ \\$\mathsf{BIM}\vdash \mathbf{AppFT} \leftrightarrow \mathbf{SemiappFT} \leftrightarrow \overleftarrow{\mathbf{BW}_{\omega^\omega}} \leftrightarrow  \overleftarrow{\mathbf{BW}}\leftrightarrow (\overleftarrow{ \mathbf{BW}_{\omega^\omega}})^+ $.\end{corollary}

\begin{proof} Use Theorems \ref{T:appsemiapp}, \ref{T:semiappcontrbw}, \ref{T:bwbwr}, \ref{T:bwbw+} and \ref{T:bw+app*}. \end{proof}

\begin{corollary}\label{Cor:sumup2} $\mathsf{BIM}\vdash \mathbf{AppFT} \rightarrow \mathbf{OI}([0,1])$. 
\end{corollary}
\begin{proof} Note: $\mathsf{BIM} \vdash \mathbf{AppFT} \leftrightarrow \overleftarrow{\mathbf{BW}}$ and:
$\mathsf{BIM}\vdash \mathbf{OI}([0,1]) \leftrightarrow \overleftarrow{\mathbf{Ded}}$ and: \\$\mathsf{BIM}\vdash \overleftarrow{\mathbf{BW}} \rightarrow \overleftarrow{\mathbf{Ded}}$. 
\end{proof}

\section{Ascoli's Lemma}

\subsection{Again the Bolzano-Weierstrass Theorem}
One might say that $\overleftarrow{\mathbf{BW}}$ states the {\it sequential compactness} of $[-1,1]$. We want to get rid of the reference to points not lying in $[-1,1]$.

 \begin{theorem}\label{T:01seqcpct}
 The following statements are equivalent in $\mathsf{BIM}$:
 \begin{enumerate}[\upshape (i)]
 \item $\overleftarrow{\mathbf{BW}}$: For all $\gamma$ in $\mathbb{Q}^\omega$,   \\if
$\forall \zeta \in [\omega]^\omega\exists n[|\gamma\circ\zeta(n+1) -_\mathbb{Q} \gamma\circ\zeta(n)| 
>_\mathbb{Q} \frac{1}{2^n}]$, then $\exists n[|1_\mathbb{Q}<_\mathbb{Q} |\gamma(n)| ]$.
\item 
  For all  $\gamma$ in $(\mathbb{Q} \cap [-1,1])^\omega$, for all $\alpha$, \\if  $\forall \zeta \in [\omega]^\omega\exists n[|\gamma\circ\zeta(n+1) -_\mathbb{Q} \gamma\circ\zeta(n)|>_\mathbb{Q} \frac{1}{2^n} \;\vee\; \alpha(n) \neq 0]$, then $\exists n[\alpha(n) \neq 0].$
 \end{enumerate}
 \end{theorem} 
 
 Taking $\alpha = \underline 0$, we see that (ii) implies the constructively weaker statement:
 $ \forall \gamma \in (\mathbb{Q} \cap [-1,1])^\omega\neg\forall \zeta \in [\omega]^\omega\exists n[|\gamma\circ\zeta(n+1) -_\mathbb{Q} \gamma\circ\zeta(n)|>_\mathbb{Q} \frac{1}{2^n}]$, that is:
 \textit{the assumption that, for some $\gamma$ in $(\mathbb{Q}\cap[-1,1])^\omega$, every subsequence positively fails to converge, leads to a contradiction}.

\begin{proof} (i) $\Rightarrow$ (ii). Let $\gamma$ in $(\mathbb{Q} \cap[-1,1])^\omega$ and  $\alpha$ be given such that $\forall \zeta \in [\omega]^\omega\exists n[|\gamma\circ\zeta(n+1) -_\mathbb{Q} \gamma\circ\zeta(n)|>_\mathbb{Q} \frac{1}{2^n} \;\vee\; \alpha(n) \neq 0]$. Define $\delta$ in $\mathbb{Q} ^\omega$ such that,  for each $n$, if $\forall i \le n[\alpha(i) = 0]$, then $\delta(n) = \gamma(n)$, and, if $\exists i \le n[\alpha(i) \neq 0]$, then $\delta(n) =_\mathbb{Q} n+_\mathbb{Q} 2$.
We prove that every subsequence of $\delta$ positively fails to converge. Assume $\zeta\in[\omega]^\omega$. Find $n$ such that $|\gamma\circ\zeta(n+1) -_\mathbb{Q} \gamma\circ\zeta(n)|>_\mathbb{Q} \frac{1}{2^n} \;\vee\; \alpha(n) \neq 0$. \textit{Either} $\forall i \le \zeta(n+1)[\alpha(i) =0]$ and $\delta\circ\zeta(n) = \gamma\circ\zeta(n)$ and $\delta\circ\zeta(n+1)=\gamma\circ\zeta(n+1)$ and $|\delta\circ\zeta(n+1) -_\mathbb{Q} \delta\circ\zeta(n)|>_\mathbb{Q} \frac{1}{2^n}$, \textit{or} $\exists i \le \zeta(n+1)[\alpha(i) \neq 0]$ and $|\delta\circ\zeta(n+2) -_\mathbb{Q} \delta\circ\zeta(n+1)| =_\mathbb{Q} \zeta(n+2) -_\mathbb{Q} \zeta(n+1) >_\mathcal{R} \frac{1}{2^{n+1}}$. We thus see that $\forall\zeta\in [\omega]^\omega\exists n[|\delta\circ\zeta(n+1) -_\mathbb{Q}\delta\circ\zeta(n)|>_\mathbb{Q} \frac{1}{2^n}]$. 
Now  use (i) and find $n$ such that $|\delta(n)|>_\mathbb{Q} 1_\mathbb{Q} $.  Conclude that $\delta(n) \neq  \gamma(n)$ and $\exists i \le n[\alpha(i) \neq 0]$.

\smallskip (ii) $\Rightarrow$ (i). 
Let $\gamma$ in $\mathbb{Q}^\omega$ be given such that  $\forall \zeta \in [\omega]^\omega\exists n[|\gamma\circ\zeta(n+1) -_\mathbb{Q} \gamma\circ\zeta(n)|>_\mathbb{Q}\frac{1}{2^n}]$. 
  Define $\alpha$  such that, for each $n$, $\alpha(n) \neq 0$ if and only if $\exists i\le n[|\gamma(i)| >_\mathbb{Q} 1_\mathbb{Q}]$.
  Define $\delta$ in $(\mathbb{Q}\cap[0,1])^\omega$ such that, for each $n$, \textit{if} $| \gamma(n)| \le_\mathbb{Q} 1_\mathbb{Q}$, then $\delta(n) = \gamma(n)$, and \textit{if not}, then $\delta(n) =0_\mathbb{Q}$.
  Note that,  for each $n$, if $\gamma(n) \neq \delta(n)$, then  $|\gamma(n)| >_\mathbb{Q} 1_\mathbb{Q}$ and $\exists i \le n[\alpha(i) \neq 0]$. 
   Let $\zeta$ be an element of $[\omega]^\omega$. Find $n$ such that $|\gamma\circ\zeta(n+1)-_\mathbb{Q} \gamma\circ\zeta(n)|>_\mathbb{Q} \frac{1}{2^n}$. \textit{Either} $|\delta\circ\zeta(n+1) -_\mathbb{Q} \delta\circ\zeta(n)|>_\mathcal{R} \frac{1}{2^n}$ \textit{or} $\exists i \le \zeta(n+1)[\alpha(i) \neq 0]$.
   We thus see that $\forall \zeta \in [\omega]^\omega\exists n[|\delta\circ\zeta(n+1) -_\mathbb{Q} \delta\circ\zeta(n)|>_\mathbb{Q} \frac{1}{2^n} \;\vee\; \alpha(n) \neq 0]$. Using (ii), find $n$ such that  $\alpha(n) \neq 0$, and conclude that
  $\exists i \le n[|\gamma(i)| >_\mathbb{Q} 1_\mathbb{Q}]$.
 \end{proof}

We now give a similar reformulation of $(\overleftarrow{\mathbf{BW}_{\omega^\omega}})^+$.

\begin{theorem}\label{T:scfbeta} The following statements are equivalent in $\mathsf{BIM}$.
\begin{enumerate}[\upshape (i)]\item $(\overleftarrow{\mathbf{BW}_{\omega^\omega}})^+$:  For all $\beta,\gamma$, if $Appfan^+(\beta)$ and $\forall \zeta \in [\omega]^\omega\exists n[\gamma\circ\zeta(n) \neq_n \gamma\circ\zeta(n+1)]$, then  $\exists n[\beta\circ\gamma(n) \neq 0].$\item For all $\beta, \gamma, \alpha$, if  $Appfan^+(\beta)$ and  $\forall n[\beta\circ \gamma(n)=0]$ and $ \forall \zeta \in [\omega]^\omega \exists n[\gamma\circ\zeta(n+1) \neq_n \gamma\circ\zeta(n) \;\vee \alpha(n) \neq 0]$, then  $\exists n[\alpha(n) \neq 0]$. \item For all $\beta, \gamma, \alpha$, if  $Appfan^+(\beta)$ and  $\forall n[\gamma^n \in \mathcal{F}_\beta]$ and $ \forall \zeta \in [\omega]^\omega \exists n[\gamma^{\zeta(n+1)} \neq_n \gamma^{\zeta(n)} \;\vee \alpha(n) \neq 0]$, then  $\exists n[\alpha(n) \neq 0]$.
\end{enumerate}\end{theorem}\begin{proof} (i) $\Rightarrow$ (ii).  Let $\beta, \alpha, \gamma$ be  given such that $Appfan^+(\beta)$ and $\forall n[\beta \circ \gamma(n)=0]$   and   $\forall \zeta \in [\omega]^\omega\exists n[\gamma\circ\zeta(n+1) \neq_n \gamma\circ \zeta(n) \;\vee\;\alpha(n) \neq 0]$. We define $\delta$ such that,  for each $n$, \textit{if} $\exists i \le n[\alpha(i) \neq 0]$, then $\delta(n) = \langle n+2\rangle$, and, \textit{if not}, then $\delta(n) = \gamma(n)$.
Let  $\zeta$ in $[\omega]^\omega$ be given. Find $n$ such that $\gamma\circ\zeta(n+1) \neq_n\gamma\circ\zeta(n)\;\vee\;\alpha(n) \neq 0$. \textit{Either} $\delta\circ\zeta(n+1) \neq_n \delta\circ\zeta(n)$ \textit{or} $\exists i \le \zeta(n+1)[\alpha(i) \neq 0]$ and $\delta\circ\zeta(n+1) \neq_1 \delta\circ\zeta(n+2)$. 
  We thus see that $\forall \zeta \in [\omega]^\omega\exists n[ \delta\circ\zeta(n+1) \neq_n \delta\circ\zeta(n) ]$.
 Now  use (i) and find $n$ such that $\beta \circ \delta(n) \neq 0$. Conclude that $\delta(n) = \langle n+2\rangle$ and $\exists i \le n[\alpha(i) \neq 0]$.

\smallskip (ii) $\Rightarrow$ (i). Let $\beta$ be given such that $Appfan^+(\beta)$ and
let $\gamma$ be given such that $\forall \zeta \in [\omega]^\omega\exists n[\gamma\circ\zeta(n+1)\neq_n \gamma\circ\zeta(n)]$. 
  Define $\alpha$ such that, for each $n$, $\alpha(n) \neq 0$ if and only $\exists i \le n[\beta\circ \gamma(i) \neq 0]$. 
  Define $\delta$   such that, for each $n$, \textit{if} $\forall i \le n[\beta\circ \gamma(i)=0]$, then $\delta(n) = \gamma(n)$, and, \textit{if not}, then $\delta(n) = \langle \;\rangle$.  
  Let $\zeta$ be an element of $[\omega]^\omega$. Find $n$ such that $\gamma\circ\zeta(n+1) \neq_n\gamma\circ\zeta(n)$. \textit{Either} $\delta\circ\zeta(n+1) \neq_n \delta\circ\zeta(n)$ \textit{or} $\exists i \le \zeta(n+1)[\alpha(i) \neq 0]$.
  We thus see that $\forall \zeta \in [\omega]^\omega\exists n[ \delta\circ\zeta(n+1) \neq_n \delta\circ\zeta(n) \;\vee\; \alpha(n) \neq 0]$.
  Now use (ii) and find $n$ such that  $\alpha(n) \neq 0$, and conclude that
  $\exists i \le n[\beta\circ \gamma(i) \neq 0]$.
  
 \smallskip (ii) $\Rightarrow$ (iii). Let $\beta$ be given such that $Appfan^+(\beta)$ and
let $\gamma,\alpha$ be given such that $\forall n[\gamma^n\in \mathcal{F}_\beta]$ and $\forall \zeta \in [\omega]^\omega\exists n[\gamma^{\zeta(n+1)}\neq_n \gamma^{\zeta(n)}]\;\vee\;\alpha(n)\neq 0]$. Define $\delta$ such that, for each $n$, $\delta(n)=\overline{\gamma^n}n$. Note that $\forall n[\beta\circ\delta(n)=0]$ and  
 $\forall \zeta \in [\omega]^\omega\exists n[\delta\circ\zeta(n+1)\neq_n \delta\circ\zeta(n)\;\vee\;\alpha(n)\neq 0]$.  Using (ii), conclude that $\exists n[\alpha(n)\neq 0]$. 
 
 \smallskip (ii) $\Rightarrow$ (ii). Let $\beta$ be given such that $Appfan^+(\beta)$ and
let $\gamma, \alpha$ be given such that $\forall n[\beta\circ\gamma(n)=0]$ and $\forall \zeta \in [\omega]^\omega\exists n[\gamma\circ\zeta(n+1)\neq_n \gamma\circ\zeta(n)\;\vee\;\alpha(n)\neq 0]$. Define $\delta$ such that, for each $n$, $\gamma(n)\sqsubset \delta^n$ and $\delta^n \in \mathcal{F}_\beta$. One may do so by requiring that, for each $n$, for each $i\ge length\bigl(\gamma(n)\bigr)$, $\delta^n(i) = \mu p [\beta(\overline{\delta^n}(i)\ast\langle p \rangle)=0]$. Note that 
 $\forall \zeta \in [\omega]^\omega\exists n[\delta^{\zeta(n+1)}\neq_n \delta^{\zeta(n)}\;\vee\;\alpha(n)\neq 0]$.  Using (iii), conclude that $\exists n[\alpha(n)\neq 0]$. 

 \end{proof}

 \subsection{Introducing $C([-1,1])$.} 

\begin{definition}  
$s\in \omega^n$ is a \emph{block} if  and only if \begin{enumerate}[\upshape (i)]
\item for each $i<n$, $\bigl(s(i)\bigr)'\in \mathbb{S}$ and $\bigl(s(i)\bigr)'' 
\in \mathbb{S}$, and
\item $((s(0))')'=-1$  and, for all $i$, if $i+1<n$, then $((s(i))')''=((s(i+1))')'$, and $((s(n))')''=1$, and, \item  for all $i$, if $i+1<n$ ,
 then $\bigl(s(i)\bigr)''\approx_\mathbb{S}\bigl(s(i+1)\bigr)''$. \end{enumerate}
$Block:=\{s\mid s$ is a block$\}$.

 For each $n$, for each $s$  in $Block\cap\omega^n$,  $\emph{height}(s) := \max_\mathbb{Q}\{((s(i))'')''-((s(i))'')'\mid i<n\}$, and  
 $\emph{mesh}(s):=\min\{((s(i))')''-((s(i))')'\mid i<n\})$.
 
For all $s,t$ in $\mathit{Block}$, 
$t\sqsubseteq_{Block} s$, \emph{the block $t$ is a refinement of the block $s$}, if and only if $\forall i < \mathit{length}(t)\exists j < \mathit{length}(s)[\bigl(t(i)\bigr)'\sqsubseteq_\mathbb{S} \bigl(s(j)\bigr)'\;\wedge\;\bigl(t(i)\bigr)''\sqsubseteq_\mathbb{S} \bigl(s(j)\bigr)'']$, and $s \;\#_{Block}\;t$, \emph{$s$ deviates from $t$}, if and only if $\exists i <\mathit{length}(s)\exists j<\mathit{length}(t)[(s(i))'\approx_\mathbb{S} (t(j))'\;\wedge\; \bigl(s(i)\bigr)''\;\#_\mathbb{S} \;\bigl(t(j)\bigr)'']$, and $s\approx_{Block} t$, $s$ \emph{does not deviate from} $t$, if and only if $\neg (s\#_{Block}t)$. 
$\varphi$ in $Block^\omega$ \emph{shrinks} if and only if $\forall n[\varphi(n+1) \sqsubseteq_\mathit{Block} \varphi(n)]$ and $\varphi$ \emph{dwindles} if and only if $\forall m\exists n[\mathit{height}\bigl(\varphi(n)\bigr) < \frac{1}{2^m}]$. $\mathrm{C}([-1,1])$ is the set of all shrinking and dwindling elements of  $Block^\omega$.
For all $\varphi, \psi$ in $C[-1,1]$, $\varphi\; \#_{Block^\omega}\;\psi$, \emph{$\varphi$ deviates from $\psi$}, if and only if $\exists n[\varphi(n)\;\#_{Block}\;\psi(n)]$ and $\varphi=_{Block^\omega}\psi$, \emph{$\varphi$ does not deviate from $\psi$}, if and only if $\forall n[\varphi(n)\approx_{Block} \psi(n)]$. For all $\varphi$ in $\mathrm{C}([-1,1])$,  for all $\alpha$ in $[-1,1]$, for all   $\beta$ in $\mathcal{R}$, we define  $\varphi: \alpha \mapsto \beta$, \emph{$\varphi$ maps $\alpha$ onto $\beta$}, if and only if, for each $m$, there exist $n$, $i$ such that $i < \mathit{length}\bigl(\varphi(n)\bigr)$ and $\alpha(n) \approx_\mathbb{S} \bigl((\varphi(n))(i)\bigr)'$ and $\bigl((\varphi(n))(i)\bigr)''\approx_\mathbb{S} \beta(m)$. \end{definition} 

\begin{lemma}\label{L:C[-1,1]} One may prove in $\mathsf{BIM}$, that, for all $\varphi$ in $\mathrm{C}([-1,1])$,\begin{enumerate}[\upshape (i)] \item  for all $\alpha, \gamma$ in $[-1,1]$, for all $ \beta,\delta$ in $\mathcal{R}$, if $\varphi:\alpha \mapsto \beta$ and $\varphi:\gamma\mapsto \delta$, then, \begin{enumerate}[\upshape (a)] \item for each $n$, if $|\alpha -\gamma|< \frac{1}{2}mesh\bigl(\varphi(n)\bigr)$, then $|\beta-\delta| < 2\cdot heigth\bigl(\varphi(n)\bigr)$, and,  \item if $\alpha=_\mathcal{R} \gamma$, then $\beta=_\mathcal{R}  \delta$, and, \end{enumerate}\item for all $\alpha$ in $[0,1]$, there exists $\beta$ in $\mathcal{R}$ such that $\varphi: \alpha \mapsto \beta$.\end{enumerate} \end{lemma}

\begin{proof} The proof is left to the reader. \end{proof}
\subsection{Introducing  moduli of uniform continuity}

\begin{definition}[Canonical elements of $\mathrm{C}[-1,1\rbrack$]

$\mathbb{S}_0:=\{s\in\mathbb{S}\mid s''-_\mathbb{Q} s'=_\mathbb{Q} 1_\mathbb{Q}\;\wedge\; 2\cdot s'\in \mathbb{Z}\}$, and,
for each $n$, $\mathbb{S}_{n+1}:=\{s\in\mathbb{S}\mid (2\cdot_\mathbb{Q} s', 2\cdot_\mathbb{Q} s'') \in \mathbb{S}_{n+1}\}$.
\footnote{Brouwer already made use of these \emph{canonical rational segments}, see \cite[\S 5]{brouwer25}.} 
For all $m,n$, $CBlock_{m,n}:=\{s\in Block\mid \forall i<\mathit{length}(s)[\bigl(s(i)\bigr)' \in \mathbb{S}_m\;\wedge\; \bigl(s(i)\bigr)'' \in \mathbb{S}_n\}$. 
For each $\delta$ in $[\omega]^\omega$,  $\mathrm{C}_\delta([-1,1]):=\{\varphi\in \mathrm{C}([-1,1])\mid \forall n[ \varphi(n)\in CBlock_{\delta(n), n}]\}$. \end{definition} 

\begin{lemma}\label{L:uniformc} One may prove in $\mathsf{BIM}$ that, for all $\delta$ in $[\omega]^\omega$, for all $\varphi $ in $\mathrm{C}_\delta([-1,1])$, for all $n>0$,   for all $\alpha,\beta$ in $[-1,1]$, for all $\gamma, \varepsilon$ in $\mathcal{R}$, if $\varphi:\alpha \mapsto \beta$ and $\varphi:\gamma \mapsto \varepsilon$ and $|\alpha -_\mathcal{R} \beta|<_\mathcal{R} \frac{1}{2^{\delta(n)+1}}$, then $|\gamma -_\mathcal{R} \varepsilon|<_\mathcal{R} \frac{1}{2^{n-1}}$.
\end{lemma}
\begin{proof} The proof is left to the reader. \end{proof}

So $\delta$ may be seen as a \textit{modulus of uniform continuity} that is valid for every member of $C_\delta([-1,1])$. Note that, for each $\varphi$ in $\mathrm{C}_\delta([-1,1])$, for each $n$, $length\bigl(\varphi(n)\bigr)=2^{\delta(n)+1}$. 
 
 \begin{definition}[Bounded elements of $\mathrm{C}([-1,1\rbrack)$] 
 
 For each $\delta$ in $[\omega]^\omega$,  \\ $\mathrm{C}_\delta([-1,1],[-1,1] ):= \{\varphi\in \mathrm{C}_\delta([-1,1])\mid \forall n \forall i<2^{\delta(n)+1}[\bigl((\varphi(n))(i)\bigr)''\sqsubseteq_\mathbb{S} (-1,1)\}$.  
  \end{definition}
 
 \begin{definition} We call the following statement \emph{Ascoli's Lemma},
 
 $\mathbf{Asc}$:  For all $\delta$  in $[\omega]^\omega$, for all $\varphi$ in $\bigl(\mathrm{C}_\delta([-1,1], [-1, 1])\bigr)^\omega$, 
 
 $\exists \zeta \in [\omega]^\omega \forall n[\varphi^{\zeta(n+1)}(n) \approx_{Block} \varphi^{\zeta(n)}(n)]$. 
 
 \smallskip The following statement is a \emph{contrapositive of Ascoli's Lemma}, 
 
 $\overleftarrow{\mathbf{Asc}}$: for all $\delta$ in $[\omega]^\omega$, for all $\varphi$ in $\bigl(\mathrm{C}_\delta([-1,1], [-1, 1])\bigr)^\omega$, for all $\alpha$, 
 
 if $\forall \zeta \in [\omega]^\omega \exists n[\varphi^{\zeta(n+1)}(n)\;\#_{Block}\; \varphi^{\zeta(n)} (n)\;\vee\;  \alpha(n)\neq 0]$, then  $\exists n[\alpha(n) \neq 0]$.\end{definition}  
 
The following argument show that $\mathbf{Asc}$ proves $\mathbf{BW}$.
 
 Assume $\mathbf{Asc}$. Let $\gamma$ in $\mathbb{Q}^\omega$ be given such that $\forall n[|\gamma(n)|\le_\mathbb{Q} 1_\mathbb{Q}]$. Define $\delta$ such that $\forall n[\delta(n)=n]$.  Define $\varphi$ in $\bigl(C_\delta([-1,1], [-1,1])\bigr)^\omega$ such that, for each $n$, for all $\alpha$ in $[-1,1]$, $\varphi^n: \alpha\mapsto \bigl(\gamma(n)\bigr)_\mathcal{R}$.  Applying $\mathbf{Asc}$, find $\zeta$ in $[\omega]^\omega$ such that $\forall n[\varphi^{\zeta(n+1)} =_n \varphi^{\zeta(n)}]$.
 Conclude that $\forall n[|\gamma\circ\zeta(n)-_\mathbb{Q}\gamma\circ\zeta(n+1)|\le_\mathbb{Q} \frac{1}{2^n}]$.  Conclude $\mathbf{BW}$. 
 
 As $\mathbf{BW}$ is constructively false, so is $\mathbf{Asc}$. 
 
 \smallskip  We now observe  that, for every $\delta$ in $[\omega]^\omega$, $C_\delta([-1,1],[-1,1])$ is an explicit fan. Given  $\delta$, define   $\delta^\ast$ such that, for each $t$, $\delta^\ast(t) =0$ if and only if $\forall i<\mathit{length}(t)[t(i) \in CBlock_{\delta(i),i}]$ and $\forall i <\mathit{length}(t) - 1[ t(i+1) \sqsubseteq_{Block} t(i)]$. Note that $\delta^\ast$ is  a spread-law and that $\mathcal{F}_{\delta^\ast}$ coincides with $C_\delta([0,1],[0,1])$. Note that, for each $n$, for each $t$ in $\omega^n$ such that $\delta^\ast(t)=0$, the set $\{m\mid \delta^\ast(t\ast\langle m\rangle )=0\}$ is a subset of $ CBlock_{\delta(n), n}$ and has at most $2^{\delta(n) + n +2}$ elements. Conclude that  $\delta^\ast$ is  an explicit fan-law.

 \begin{theorem}\label{T:asc} The following statements are equivalent in $\mathsf{BIM}$.
 \begin{enumerate}[\upshape (i)] \item $(\overleftarrow{\mathbf{BW}_{\omega^\omega}})^+$. 
\item $\overleftarrow {\mathbf{Asc}}$.  \item $\mathbf{AppFT}$.
 \end{enumerate}
 \end{theorem}
 \begin{proof} (i) $\Rightarrow$ (ii). This follows from the fact that, for every $\delta$,  $\delta^\ast$ is an  explicit fan-law, see Theorem \ref{T:scfbeta}(iii). 
 
  Use the observation that, for each $\varphi$ in $C_\delta([-1,1], [-1,1])^\omega$, 
 
 if 
  $\forall \zeta \in [\omega]^\omega \exists n[\varphi^{\zeta(n+1)}(n)\;\#_{Block}\; \varphi^{\zeta(n)} (n)\;\vee\;  \alpha(n)\neq 0]$, 
  
  then $\forall \zeta \in [\omega]^\omega \exists n[\varphi^{\zeta(n+1)}\neq_n \varphi^{\zeta(n)} \;\vee\;  \alpha(n)\neq 0]$.

\smallskip
(ii) $\Rightarrow$ (i). Assume (ii). According to Corollary \ref{C:sumup3}, it suffices to prove $\overleftarrow{\mathbf{BW}}$. We also use Theorem \ref{T:01seqcpct}. Let $\gamma$, $\alpha$ be given such that   $\gamma \in \mathbb{Q} \cap [-1,1])^\omega$ and   $\forall \zeta \in [\omega]^\omega\exists n[|\gamma\circ\zeta(n+1) -_\mathbb{Q} \gamma\circ\zeta(n)|>_\mathbb{Q} \frac{1}{2^n} \;\vee\; \alpha(n) \neq 0]$. Slightly adapting the proof of Lemma \ref{L:ftconvprep}(ii), one may conclude that $\forall \zeta \in [\omega]^\omega\exists n[|\gamma\circ\zeta(n+1) -_\mathbb{Q} \gamma\circ\zeta(n)|>_\mathbb{Q} \frac{1}{2^{n-1}} \;\vee\; \alpha(n) \neq 0]$. Let $\delta$ in $[\omega]^\omega$ be given. Define $\varphi$ such that, for each $n$, $\varphi^n \in \mathrm{C}_\delta([-1,1])$, and, for all $\alpha$ in $[-1,1]$ $\varphi:\alpha\mapsto \bigl(\gamma(n)\bigr)_\mathcal{R}$.  Note that, for all $p,q,n $, if $\varphi^p =_{n+1} \varphi^q$, then $|\gamma(p) -_\mathbb{Q} \gamma(q)|<_\mathbb{Q} \frac{1}{2^n}$. Conclude that $\forall \zeta \in [\omega]^\omega \exists n[\varphi^{\zeta(n+1)}\;\#_{Block}\; \varphi^{\zeta(n)} \;\vee\;  \alpha(n)\neq 0]$. Using (ii), conclude that $\exists n[\alpha(n)\neq 0]$. We thus see that, for all $\alpha$, for all $\gamma$ in $(\mathbb{Q}\cap [-1.1])^\omega$ , if $\forall \zeta \in [\omega]^\omega\exists n[|\gamma\circ\zeta(n+1) -_\mathbb{Q} \gamma\circ\zeta(n)|>_\mathbb{Q} \frac{1}{2^{n-1}} \;\vee\; \alpha(n) \neq 0]$, then $\exists n[\alpha(n)\neq 0]$, i.e. $\overleftarrow{\mathbf{BW}}$, see Theorem \ref{T:01seqcpct}.

\smallskip
(i) $\Leftrightarrow$ (iii). According to Corollary \ref{C:sumup3}   (i) is equivalent, in $\mathsf{BIM}$, to $\mathbf{AppFT}$.
 \end{proof}

	\section{Ramsey's Theorem}
	\subsection{The intuitionistic (infinite) Ramsey Theorem}

	\begin{definition}For all $k>0$, for every  $X\subseteq\omega$,  we define: $X$ is \emph{$k$-almost-full},  $Almost$-$full_k(X)$,  if and only if $\forall \zeta \in [\omega]^\omega \exists s \in [\omega]^k[\zeta \circ
	 s \in X]$.
	 
	  For each $k>0$,    the following statement is called \emph{the $k$-dimensional  Intuitionistic Ramsey Theorem}, $\mathbf{IRT}(k)$:
	
	\smallskip $\forall \alpha\forall \beta[\bigl(Almostfull_k(D_\alpha) \;\wedge\;
Almostfull_k( D_\beta)\bigr) \rightarrow Almostfull_k(D_\alpha \cap D_\beta)].$

\smallskip
The statement $\mathbf{IRT}(1)$ is called the \emph{Intuitionistic Pigeonhole Principle}. 

\smallskip The statement $\mathbf{IRT}:=\forall k>0[\mathbf{IRT}(k)]$ is called the \emph{Intuitionistic Ramsey Theorem}.

\smallskip For each $k\ge 1$, the following statement is called \emph{the $k$-dimensional  Classical Ramsey Theorem}, $\mathbf{CRT}(k)$:
	
	\smallskip $\forall \alpha[\exists \zeta \in [\omega]^\omega\forall s \in[\omega]^k[\alpha\circ s \in D_\alpha]\;\vee\; \exists \zeta \in [\omega]^\omega\forall s \in[\omega]^k[\alpha\circ s \notin D_\alpha]].$
	
	\smallskip The statement $\mathbf{CRT}:=\forall k>0[\mathbf{CRT}(k)]$ is called the \emph{Classical Ramsey Theorem}.\end{definition}
	\begin{theorem} \begin{enumerate}[\upshape (i)]\item  $\mathsf{BIM}\vdash  \mathbf{CRT}(1)\rightarrow \mathbf{LPO}$. \item $\mathsf{BIM} + X\vee\neg X\vdash\forall k>0[\mathbf{IRT}(k)\leftrightarrow \mathbf{CRT}(k)]$.  \end{enumerate} \end{theorem}
	
	\begin{proof} (i) Assume $\mathbf{CRT}(1)$.
	Let $\alpha$  be given. Define $\beta$  such that $\forall n[\beta(\langle n\rangle ) \neq 0 \leftrightarrow \exists i \le n[\alpha(i) \neq 0]]$.  
	Note that $ \exists \zeta \in [\omega]^\omega\forall n[\langle\zeta(n)\rangle\in D_\beta] \leftrightarrow \exists n[\alpha(n)\neq 0]$ and 
	 that  $ \exists \zeta \in [\omega]^\omega\forall n[\langle\zeta(n)\rangle\notin D_\beta] \leftrightarrow \forall n[\alpha(n)= 0]$. 
Using $\mathbf{CRT}(1)$, conclude that 
	$\exists n[\alpha(n)\neq 0]\;\vee\;\forall n[\alpha(n)=0]$. 
	 Conclude that  $\forall \alpha[\exists n[\alpha(n)\neq 0]\;\vee\;\forall n[\alpha(n)=0]]$, i.e. $\mathbf{LPO}$.
	  
	  \smallskip (ii) Let $k>0$ be given and assume $\mathbf{IRT}(k)$. Let $\alpha$ be given and note that $D_\alpha$ and $\omega\setminus D_\alpha$ can not both be $k$-almost-full. If $D_\alpha$ is not $k$-almost-full, then $\exists\zeta \in [\omega]^\omega\forall s \in [\omega]^k[\alpha\circ s \notin D_\alpha]$ and, if $\omega\setminus D_\alpha$ is not $k$-almost-full, then $\exists\zeta \in [\omega]^\omega\forall s \in [\omega]^k[\alpha\circ s \in D_\alpha]$. We may conclude that $\mathbf{CRT}(k)$.

	Conversely, let $k>0$ be given and assume $\mathbf{CRT}(k)$. Let $\alpha, \beta$ be given such that   $D_\alpha, D_\beta$ are both $k$-almost-full.  
	Let $\zeta$ be an element of $[\omega]^\omega$. Define $\gamma=\alpha\circ \zeta$ and note that $D_\gamma$ is also $k$-almost-full. Conclude $\neg\exists \eta \in [\omega]^\omega\forall s \in [\omega]^k[\eta\circ s \notin D_\gamma]$. Using $\mathbf{CRT}(k)$, find $\eta$ in $[\omega]^\omega$ such that $\forall s \in [\omega]^k[\eta\circ   s \in D_\gamma]$, i.e. $ \forall s \in [\omega]^k[\zeta\circ\eta\circ   s \in D_\alpha]$. Find $s$ in $[\omega]^k$ such that $\zeta\circ\eta\circ s \in D_\beta$  and define $t:= \eta \circ s$. Note that $\zeta \circ t \in D_\alpha \cap D_\beta$. We thus see that $\forall \zeta\in [\omega]^\omega \exists t\in [\omega]^k[\zeta\circ t \in D_\alpha\cap D_\beta]$, i.e.  $D_\alpha\cap D_\beta$ is $k$-almost-full. We may conclude that $\mathbf{IRT}(k)$. 
	\end{proof}
	
	 Ramsey's Theorem is proven in \cite{ramsey} and  the intuitionistic version is treated in \cite{veldmanbezem93}, \cite{veldman2004b} and \cite{veldman2008}. In the latter three papers, the \textit{full Intuitionistic Ramsey Theorem} is considered. This theorem states that 
	\textit{for all $k\ge 1$, for all   $R, T\subseteq\omega$, if both $R,T$ are $k$-almost-full, then $R\cap T$ is $k$-almost-full}. 
	
	 In the statement $\forall k>0[\mathbf{IRT}(k)]$ one restricts oneself to  {\it decidable} subsets of $\omega$.  
	
	\begin{theorem}\label{T:irt} 
	 The following statements are equivalent in $\mathsf{BIM}$:
	 \begin{enumerate}[\upshape (i)]
	 \item  $\mathbf{AppFT}$.
	 \item $\forall k>0[\mathbf{IRT}(k)]$.
	 \item  $\mathbf{IRT}(3)$.
\item  $\mathbf{EnDec?!}\;\wedge\;\mathbf{IRT}(2)$.

\end{enumerate}
\end{theorem}

\begin{proof} (i) $\Rightarrow$ (ii).  We use induction and first prove $\mathbf{IRT}(1)$. 

  Let $\alpha, \beta$ be given such that both $D_\alpha$ and $D_\beta$ are 1-almost-full.
Let $\zeta$  in $[\omega]^\omega$ be given.  Note that $\forall n  \exists m > n[\langle\zeta(m)\rangle \in D_\alpha]$. 
 Define $\eta$ such that $\eta(0)=\mu m[\langle\zeta(m)\rangle \in D_\alpha]$  and $\forall n   [\langle\zeta\circ\eta(n+1)=\mu m>\eta(n)[\langle\zeta(m)\rangle \in D_\alpha]]$. Note that $\zeta\circ\eta\in[\omega]^\omega$ and $\forall n[\langle\zeta\circ\eta(n)\rangle \in D_\alpha]$. Find $m$ such that  $\langle \zeta\circ\eta(m)\rangle \in D_\beta$, define $q:= \eta(m)$ and note that 
$\langle \zeta(q)\rangle\in D_\alpha \cap D_\beta$.
We thus see that $\forall \zeta \in [\omega]^\omega \exists q[\zeta(q) \in D_\alpha \cap D_\beta]$, i.e.  $D_\alpha \cap D_\beta$ is 1-almost-full.
 
 \smallskip
 Now assume $k \ge 1$ and $\mathbf{IRT}(k)$. We are going to prove $\mathbf{IRT}(k+1)$.

 Let $\alpha, \beta$ be given such that  $D_\alpha,D_\beta$ both are $(k+1)$-almost-full. We have to prove that also $D_\alpha\cap D_\beta$ is $(k+1)$-almost-full. We first set ourselves the modest goal of proving that the set $D_\alpha\cap D_\beta$ is {\it inhabited}, i.e. $\exists u \in [\omega]^{k+1}[u \in D_\alpha \cap D_\beta]$.
 
 For each $s\in [\omega]^{<\omega}$ and 
	 each $X\subseteq \omega$, 
	  define 
	  \begin{center} 
	 $s$ is {\em $(k+1)$-prehomogeneous for $X$} \end{center} if and only if $\forall t\in[\omega]^k\forall i\forall j[t(k-1)<i <j <\mathit{length}(s)\rightarrow \bigl(s \circ (t\ast\langle i \rangle) \in X\leftrightarrow s \circ (t\ast\langle j\rangle) \in X\bigr)].$

	 Define  $\varphi$  such that
	 $\varphi(0) =  \langle \; \rangle$, and,
	 for each $n$, $\varphi(n+1) = \varphi(j) \ast \langle n\rangle$ where $j$ is the largest  $i\le n$ such that $\varphi(i) \ast\langle n \rangle$ is $(k+1)$-prehomogeneous for $D_\alpha$ and for $D_\beta$. Note that $\varphi(n+1)$ is well-defined, as $\varphi(0) \ast\langle n \rangle$ is $(k+1)$-prehomogeneous for $D_\alpha$ and for $D_\beta$.
	 
	Define $\delta$ such that $\delta(\langle\;\rangle)\neq 0$ and $\forall t\forall n[\delta(t\ast\langle n\rangle) \neq 0\leftrightarrow \varphi(n)=t\ast\langle n \rangle]$.  
	Note that $\forall s[\delta(s)\neq 0\leftrightarrow \exists n[\varphi(n)=s]]$. 
Also note that $D_\delta \subseteq [\omega]^{<\omega}$ and that $D_\delta$  is a {\it tree}, i.e. $\forall s\forall n[s\ast \langle n \rangle \in D_\delta\rightarrow s \in D_\delta]$.
	$D_\delta$ is called the \textit{$(k+1)$-Erd\"os-Rad\'o tree belonging to $\alpha$ and $\beta$}.
	 $D_\delta$ is infinite and, for each $n$, for each $s$ in $D_\delta\cap [\omega]^n$,  there are at most $4^{\binom{n}{k}}$ numbers $i$ such that $\delta(s\ast\langle i \rangle) \neq 0$. 
	 
	 This is because, for each $n$, for each $s$ in $D_\delta\cap [\omega]^n$, there are $\binom{n}{k}$ elements $t$ of $[\omega]^{k}$ such that $t(k-1)<n$, and, for each such $t$,  for each $i$, $(s\circ t)\ast\langle i \rangle$ belongs to one of the four sets $D_\alpha \cap D_\beta$, $D_\alpha \setminus D_\beta$, $D_\beta \setminus D_\alpha$ and $\omega\setminus(D_\alpha \cup D_\beta)$.  Note that, for all $i,j$, if $s\ast\langle i \rangle$ belongs to $D_\delta$ and $i<j$, and, for all  elements $t$ of $[\omega]^{k}$ such that $t(k-1)<n$, $(s\circ t)\ast\langle i\rangle$ and $(s\circ t)\ast\langle j\rangle$ belong to the same of these four sets, then  $s\ast\langle i, j\rangle$ is $(k+1)$-prehomogeneous for both $D_\alpha$ and $D_\beta$, and $s\ast\langle j\rangle$ does not belong to $D_\delta$. 
	 
	 Define $\gamma$ such that $\gamma(0) =1$, and, for each $n$, $\gamma(n+1) = \gamma(n) \cdot 4^{\binom{n}{k}}$. Note that, for each $n$, there are at most $\gamma(n)$ elements $s$ of $D_\delta$ such that $\mathit{length}(s) = n$. 
	   Define $\varepsilon$ such that $\forall s[\varepsilon(s) = 0\leftrightarrow \exists t\exists n[t \in D_\delta\;\wedge \;s = t\ast\overline{\underline 0}n]]$. 
	  Note that $\forall s[\varepsilon(s) = 0\leftrightarrow \exists n[\varepsilon(s\ast\langle n \rangle)=0 ]]$.  
	  Define $\eta$ such that $\eta(0) = 1$, and $\forall n[\eta(n+1) = \eta(n) \cdot( 4^{\binom{n}{k}}+1)]$. Note that, for each $n$, there are at most $\eta(n)$ numbers $s$ in $\omega^n$ such that   $\varepsilon(s) = 0$.
	   We thus see that the set $\mathcal{F}_\varepsilon$ is an explicit approximate fan. 
	   
	   Now define $B:=\{s \mid \varepsilon(s) = 0\;\wedge\;\exists t\in [\omega]^{k+1}[t(k)<\mathit{length}(s)\;\wedge\;s \circ t \in D_\alpha \cap D_\beta]\;\vee\;\delta(s) = 0\}.$ We  prove that  $B$ is a bar in $\mathcal{F}_\varepsilon$. 
	    Let $\gamma$ in $\mathcal{F}_\varepsilon$ be given.
	 Define $\gamma^\ast$ such that $\gamma^\ast(0) = \gamma(0)$ and, for each $n$, if $\gamma(n+1)\neq 0$ 
	  then $\gamma^\ast(n+1) = \gamma(n+1)$, and, if $\gamma(n+1)=0$, then $\gamma^\ast(n+1)= \gamma^\ast(n) +1$.
	 Note that $\gamma^\ast \in [\omega]^\omega$ and, if $\gamma \in [\omega]^{\omega}$, then $\gamma^\ast = \gamma$, and   $\forall n[\overline \gamma n \neq \overline{\gamma^*}n\rightarrow \delta(\overline \gamma n)=0]$. 
	   Define $\alpha^\ast$ and $\beta^\ast$  such that, for each $t$ in $[\omega]^k$,
	    $\alpha^\ast(t) = \alpha\bigl(\gamma^\ast\circ (t\ast\langle j\rangle)\bigr)$ and $\beta^\ast(t) = \beta\bigl(\gamma^\ast\circ (t\ast\langle j\rangle)\bigr)$ where $j := t(k-1)+1$.

	Let $\zeta$ in $[\omega]^\omega$ be given. Find $s$ in $[\omega]^{k+1}$ such that $\gamma^\ast \circ \zeta \circ s \in D_\alpha$.
	Define $n:=(\zeta \circ s)(k) +1$ and  $t:= \overline s k$ and $i:= s(k)$. Note that $s=t\ast\langle i\rangle$ and $\overline{\gamma^\ast}n\circ \zeta \circ (t\ast\langle i \rangle) \in D_\alpha$, i.e. $\overline{\gamma^\ast}n\circ \bigl((\zeta \circ t)\ast\langle \zeta (i)\rangle\bigr) \in D_\alpha$. Conclude that {\it either} $\overline{\gamma^\ast}n\circ \bigl((\zeta \circ t)\ast\langle j \rangle\bigl) \in D_\alpha$,
 where $j:=(\zeta \circ t)(k-1) +1$, and therefore, 	$\zeta\circ t \in D_{\alpha^\ast}$, {\it or} $\overline {\gamma^\ast}n$ is not $(k+1)$-prehomogeneous for $D_\alpha$ and, therefore, $\delta(\overline {\gamma^\ast} n )=0$. Conclude that  $\forall \zeta \in [\omega]^\omega\exists t\in [\omega]^k [\zeta \circ t \in D_{\alpha^\ast} \;\vee \;\exists n[\delta\bigl(\overline{\gamma^\ast}n\bigr)=0]].$  
One may  prove by a similar argument that
$\forall \zeta \in [\omega]^\omega\exists t\in [\omega]^k [\zeta \circ t \in D_{\beta^\ast}\;\vee \;\exists n[\delta\bigl(\overline{\gamma^\ast}n\bigr)=0]]$.

	Define $\alpha^{\ast\ast},\beta^{\ast\ast}$  such that $\forall s[\alpha^{\ast\ast}(s)\neq 0 \leftrightarrow\bigl(\alpha^\ast(s)\neq 0\;\vee\;\exists n\le s[\delta(\overline{\gamma^\ast}n)=0]\bigr)]$ and $\forall s[\beta^{\ast\ast}(s)\neq 0 \leftrightarrow\bigl(\beta^\ast(s)\neq 0\;\vee\;\exists n\le s[\delta(\overline{\gamma^\ast}n)=0]\bigr)]$. 
 Conclude that
	$\forall \zeta \in [\omega]^\omega \exists t \in [\omega]^k[\zeta \circ t \in D_{\alpha^{\ast\ast}}] \;\wedge\; \forall \zeta \in [\omega]^\omega \exists t \in [\omega]^k[\zeta \circ t \in D_{\beta^{\ast\ast}}] $.
	Use $\mathbf{IRT}(k)$ and  conclude that
	$\forall \zeta \in [\omega]^\omega \exists t \in [\omega]^k[\zeta \circ t \in D_{\alpha^{\ast\ast}}\cap D_{\beta^{\ast\ast}}]$.

	Find $t$ in $[\omega]^k$ such that $t \in D_{\alpha^{\ast\ast}} \cap D_{\beta^{\ast\ast}}$.  
	\textit{Either} $t \in D_{\alpha^\ast} \cap D_ {\beta^\ast}$ \textit{or} $\exists n[\delta(\overline{\gamma^\ast}n) = 0]$, i.e.
	\textit{either} $\gamma^\ast\circ(t\ast\langle j \rangle) = (\gamma^\ast \circ t)\ast \langle \gamma^\ast( j)\rangle \in D_\alpha \cap D_\beta$, where $j = t(k-1) + 1$, \textit{or} $\exists n[\delta(\overline{\gamma^\ast} n) = 0]$. In both cases, we find $n$ such that $\overline {\gamma^\ast} n \in B$.
	 Note that   \textit{either} $\overline \gamma n = \overline{\gamma^\ast} n$ \textit{or} $\delta(\overline \gamma n) = 0$. In both cases, $\overline \gamma n \in B$.

	We thus see that $\forall \gamma\in\mathcal{F}_\varepsilon\exists n[\overline\gamma n \in B]$, i.e. $Bar_{\mathcal{F}_\varepsilon}(B)$.
	
	Find $\zeta$ in $[\omega]^\omega$ such that $\zeta(0)=\mu m[m\in D_\delta]$, and, for each $n$, $\zeta(n+1)=\mu m[m \in D_\delta \;\wedge\; m>\zeta(n)]$. Note that $\forall m[m\in D_\delta \rightarrow \varepsilon(m)=0]$. Using  $\mathbf{AppFT}$ and Corollary \ref{C:almostfinitestrongbar},   conclude that $Strongbar_{\mathcal{F}_\varepsilon}(B)$ and  find $n,m$ such that $\overline{\zeta(n)}m \in B$ and $\zeta(n) \in B$.  
	Conclude that $\exists t \in [\omega]^{k+1}[s \circ t \in D_\alpha \cap D_\beta]$, and $\exists u \in [\omega]^{k+1}[u \in D_\alpha \cap D_\beta]$. We thus see that, for all $\alpha, \beta$, if $D_\alpha, D_\beta$ are $(k+1)$-almost-full, then $\exists u \in [\omega]^{k+1}[u \in D_\alpha \cap D_\beta]$.
	
	\smallskip
	
Again, let $\alpha, \beta$ be given such that $D_\alpha, D_\beta$ both are $(k+1)$-almost-full
and let  $\gamma$ in $[\omega]^\omega$ be given.
	Define $\alpha',\beta'$ such that $\forall t\in[\omega]^{<\omega}[\alpha'(t)= \alpha(\gamma \circ t)\;\wedge\;\beta'(t) = \beta(\gamma \circ t)]$. 
	 Note that  $D_{\alpha'}$ is $(k+1)$-almost-full, as $\forall\delta\in[\omega]^\omega\exists s\in[\omega]^{k+1}[ \gamma \circ \delta \circ s \in D_\alpha]$ and  $\forall\delta\in[\omega]^\omega\exists s\in[\omega]^{k+1}[\delta \circ s \in D_{\alpha'}]$.
	Also $D_{\beta'}$ is  $(k+1)$-almost-full. Conclude that $\exists u\in[\omega]^{k+1}[u \in D_{\alpha'}\cap D_{\beta'}]$, and that $\exists u\in[\omega]^{k+1}[\gamma \circ u \in D_\alpha \cap D_\beta]$.
	We thus see that $\forall \gamma \in [\omega]^\omega\exists u\in[\omega]^{k+1}[\gamma \circ u \in D_\alpha \cap D_\beta]$, i.e. $D_\alpha \cap D_\beta$ is $(k+1)$-almost-full.

	  Conclude that $\forall k>0[\mathbf{IRT}(k) \rightarrow \mathbf{IRT}(k+1)]$ and  $\forall k>0[\mathbf{IRT}(k)]$.

	  \smallskip

	  (ii) $\Rightarrow$ (iii). Obvious. 
	  
	  \smallskip (iii) $\Rightarrow$ (iv). Note that $\mathsf{BIM}\vdash \mathbf{IRT}(3) \rightarrow \mathbf{IRT}(2)$. It thus suffices to show that $\mathsf{BIM} \vdash \mathbf{IRT}(3) \rightarrow \mathbf{EnDec?!}$. 
	  
	  Assume $\mathbf{IRT}(3)$. Let $\gamma$, $n$ be given such that $\forall \alpha[(n \notin D_\alpha \;\wedge\; D_\alpha \subseteq E_\gamma)\rightarrow \exists p[p \notin D_\alpha \;\wedge\; p \in E_{\gamma}]].$
       We will prove that $n \in E_\gamma$, i.e. $\exists j[\gamma(j) = n+1]$.

         Define  $\alpha$ such that for all $s$ in $[\omega]^3$, $\alpha(s)\neq 0$ if and only if $\exists i<s(0)[i\in E_{\overline\gamma s(2)}\setminus E_{\overline\gamma s(1)}]\;\vee\;n\in E_{\overline\gamma s(2)}].$
         We now prove that $D_\alpha$ is $3$-almost-full. 
     Assume $\delta\in[\omega]^\omega$. Define $\varepsilon$  such that, for all $p$, $\varepsilon(p)\neq 0$ if and only if  $p\neq n\;\wedge\;\forall k[ k=\mu i[p < \delta(i]\rightarrow \;p\in E_{\overline\gamma\delta(k+1)}].$ 
       Note that $D_\varepsilon \subseteq E_\gamma \;\wedge\; n \notin D_\varepsilon$. Find  $p$ in $E_\gamma\setminus D_\varepsilon$. Distinguish two cases.
      \begin{enumerate} \item $p = n$. Find $j$ such that $\gamma(j) = n+1$. Find $k$ such that $j < \delta(k+2)$ and note that $\langle \delta(0), \delta(1), \delta(k+2)\rangle \in D_\alpha$.

       \item  $p \neq n$.
         Find $k:=\mu i[p < \delta(i)]$. Find $j$ such that $\gamma(j) = p+1$.
        Note that $p \notin D_\varepsilon$, and, therefore, $p \notin E_{\overline\gamma\delta(k+1)}$. Find $l$ such that $j< \delta(l)$ and note that $\langle \delta(k), \delta(k+1), \delta(l)\rangle \in D_\alpha$.
       
      \end{enumerate}
   We thus see that $\forall \delta\in [\omega]^\omega\exists s \in [\omega]^3[\delta\circ s \in D_\alpha]$, i.e. $D_\alpha$ is $3$-almost-full.    
       
       Define $\beta$   such that $\forall s \in [\omega]^3[\beta(s) \neq 0\leftrightarrow\forall i < s(0)[i\in E_{\overline\gamma s(2)}\rightarrow i\in E_{\overline\gamma s(1)}]]$.
       We now prove that $D_\beta$ is $3$-almost-full.
       Assume $\delta\in[\omega]^\omega$.  
       Also assume that $\forall k \le \delta(0)[\delta(0), \delta(k), \delta(k+1)\rangle \notin D_\beta]$. 
       Then $\forall k \le \delta(0)\exists i < \delta(0)[i \in E_{\overline \gamma \bigl(\delta(k+1)\bigr)}\setminus E_{\overline \gamma\bigl(\delta(k)\bigr)}]$.  
       Clearly, this is impossible. Conclude that $\exists k\le \delta(0)[\langle \delta(0), \delta(k), \delta(k+1)\rangle \in D_\beta]$. We thus see that $\forall \delta\in [\omega]^\omega\exists s \in [\omega]^3[\delta\circ s \in D_\beta]$, i.e. $D_\beta$ is $3$-almost-full.

      Using $\mathbf{IRT}(3)$,       conclude that  $D_\alpha \cap D_\beta$ is $3$-almost-full and, in particular,  $\exists s[ s \in [\omega]^3[s \in D_\alpha \cap D_\beta]$. Find such $s$ and conclude that $n\in E_{\overline \gamma s(2)}\subseteq E_\gamma$.
       
    We thus see that, for all $\gamma$, $n$, if  $\forall \alpha[(n \notin D_\alpha \;\wedge\; D_\alpha \subseteq E_\gamma)\rightarrow \exists p[p \notin D_\alpha \;\wedge\; p \in E_{\gamma}]]$, then 
        $n \in E_\gamma$. One easily concludes $\mathbf{EnDec?!}$.
           
       \smallskip (iv) $\Rightarrow$ (i).  As $\mathsf{BIM}$ proves $ \mathbf{EnDec?!} \leftrightarrow \overleftarrow{\mathbf{Ded}}$, see Corollary \ref{Cor:sumup}, and also $\mathbf{AppFT} \leftrightarrow \overleftarrow{\mathbf{BW}}$, see Corollary \ref{C:sumup3}, it  suffices to show that  $\mathsf{BIM}$ proves  $\bigl(\overleftarrow{\mathbf{Ded}} + \mathbf{IRT}(2)\bigr)\rightarrow \overleftarrow{\mathbf{BW}}$.\footnote{The argument that follows is also used for \cite[Theorem 35]{veldman2023}.} 
       
       Assume $\overleftarrow{\mathbf{Ded}}$ and  $\mathbf{IRT}(2)$. 
       Let
$ \gamma$ in $\mathbb{Q}^\omega$ be given such that $\forall \zeta \in [\omega]^\omega\exists n[|\gamma\circ\zeta(n+1) -_\mathbb{Q} \gamma\circ\zeta(n)| 
>_\mathbb{Q} \frac{1}{2^n}]$. 
We shall prove that $\exists n[|\gamma(n)| >_\mathbb{Q} 1_\mathbb{Q} ]$.

Let $\zeta$ in $[\omega]^\omega$ be given. Define $\delta$ in $\mathbb{Q}^\omega$ such that $\delta(0) := \gamma\circ\zeta (0)$, and, for each $n$, {\it if} $\gamma\circ\zeta(n+1) \ge_\mathbb{Q}\delta(n)]$, then $\delta(n+1) = \gamma\circ\zeta(n+1)$, and, {\it if not}, then $\delta(n+1) = \delta(n) +_\mathbb{Q} 1_\mathbb{Q}$. 
Note that $\forall n[\delta(n+1) \ge_\mathbb{Q} \delta(n)]$ and $\forall \eta \in [\omega]^\omega\exists n[\delta\circ\eta(n+1) -_\mathbb{Q} \delta\circ\eta(n) 
>_\mathbb{Q} \frac{1}{2^n}]$. Using $\overleftarrow{\mathbf{Ded}}$, find $n$ such that $\delta(n) >_\mathbb{Q} 1_\mathbb{Q}$. Note that \textit{either} $\delta(n) = \gamma\circ \zeta(n)$ and $\gamma\circ \zeta(n)>_\mathbb{Q}1_\mathbb{Q}$, \textit{or} $\delta(n) \neq \gamma\circ\zeta(n)$ and $\exists i\le n[\gamma\circ\zeta(i) >_\mathbb{Q} \gamma\circ\zeta(i+1)]$. We thus see that 
$\forall \zeta \in [\omega]^\omega\exists n[ \gamma\circ \zeta(n)>_\mathbb{Q}1_\mathbb{Q} \;\vee\; \gamma\circ\zeta(n) >_\mathbb{Q} \gamma\circ\zeta(n+1)]$. 
One may also prove that  $\forall \zeta \in [\omega]^\omega\exists n[ \gamma\circ\zeta(n)<_\mathbb{Q} (-1)_\mathbb{Q} \;\vee\; \gamma\circ\zeta(n) <_\mathbb{Q} \gamma\circ\zeta(n+1)]$. 
Use $\mathbf{IRT}(2)$ and conclude that $\exists n[|\gamma(n)| >_\mathbb{Q} 1_\mathbb{Q} ]$.

We thus see that, for each $ \gamma$ in $\mathbb{Q}^\omega$, if  $\forall \zeta \in [\omega]^\omega\exists n[|\gamma\circ\zeta(n+1) -_\mathbb{Q} \gamma\circ\zeta(n)| 
>_\mathbb{Q} \frac{1}{2^n}]$, then $\exists n[|\gamma(n)| >_\mathbb{Q} 1_\mathbb{Q}]$, i.e. $\overleftarrow{\mathbf{BW}}$. 
\end{proof}

     \subsection{The Paris-Harrington Theorem}

We first show that     the Intuitionistic Ramsey Theorem extends from \textit{decidable} to \textit{enumerable} subsets of $\omega$.
   \begin{corollary}\label{C:ramseyenumer}  One may prove, in $\mathsf{BIM} + \forall k>0[\mathbf{IRT}(k)]$:
   
   $\forall k>0\forall \alpha\forall \beta[\bigl(Almostfull_k(E_\alpha) \;\wedge\;
Almostfull_k(E_\beta)\bigr) \rightarrow Almostfull_k(E_\alpha\cap E_\beta)]$.
\end{corollary}
\begin{proof} Let $k, \alpha, \beta$ be given such that both $E_\alpha$ and $E_\beta$ are $k$-almost-full. 
Define  $\delta, \varepsilon$, such that $\forall s[\delta(s) \neq 0\leftrightarrow \bigl(s \in [\omega]^{k+1}\;\wedge\;\overline s k \in E_{\overline\alpha s}\bigr)]$ and $\forall s[\varepsilon(s) \neq 0\leftrightarrow \bigl(s \in [\omega]^{k+1}\;\wedge\;\overline s k \in E_{\overline\beta s}\bigr)]$. Note that  $D_\delta$ and $D_\varepsilon$ are $k+1$-almost-full. Using $\mathbf{IRT}(k+1)$, conclude that $D_\delta \cap D_\varepsilon$ is $k+1$-almost-full.  Note that $\forall s[s \in D_\delta \cap D_\varepsilon\rightarrow\overline s k \in E_\alpha \cap E_\beta]$. Conclude that $E_\alpha \cap E_\beta$ is $k$-almost-full. \end{proof}

\begin{definition}
For all positive integers $m$, $k$, $[m]^k:=\{s\in[\omega]^k\mid s(k-1) < m\}$.
For all  positive integers $m,k$, for all $c, r$, we define $c: [m]^k \rightarrow r$ if and only if $\forall s\in[m]^k[s < \mathit{length}(c)\;\wedge\; c(s) < r]$.
For all $r, k, n, M$ such that $n\ge k$, we define $M\rightarrow_\ast (n)^k _r$
if and only if,  for all $ c:[M]^k \rightarrow r$, there exist $l,s$ such that $s \in [M]^l$ and
$l\ge n$ and $l\ge s(0)$ and   $\forall u \in [l]^k \forall v \in [l]^k[c(s \circ u) = c(s \circ v)]]$.

    The \emph{Paris-Harrington Theorem} $\mathbf{PH}$ is the following statement:
     \[\forall r \forall k\forall n\ge k \exists M [M \rightarrow_\ast (n)^k_r]\] \end{definition}
     In \cite{veldmanbezem93}, it is explained how $\mathbf{PH}$ may be derived from $\mathbf{IRT}$. It seems useful to consider this argument again in the formal context of $\mathsf{BIM}$.
     
     \begin{definition}  $[\omega]^{<\omega} := \bigcup\limits_{k \in \omega}[\omega]^k$.  
     
     $X\subseteq[\omega]^{<\omega}=\omega$ is called \emph{$\omega$-almost-full}, $Almostfull_\omega(X)$, if and only if $\forall \zeta \in [\omega]^\omega  \exists s \in [\omega]^{<\omega}[\zeta \circ s \in X]$.

      Let $r$ be a positive integer and let $\chi$ be given such that $\forall n[\chi(n) < r]$. One calls $\chi$   an  \emph{$r$-colouring of $\omega$}.
      
 For each       $X\subseteq [\omega]^{<\omega}$, for each $k>0$,  we let $X^{\chi, k}$ be the set of all $s$ in $X$ that are $\chi,k$-\emph{monochromatic}, i.e.  for all $u,v$ in $[\omega]^k$, if both   $u(k-1) < \mathit{length}(s)$ and $v(k-1) < \mathit{length}(s)$, then $\chi(s\circ u) = \chi(s \circ v).$ \end{definition}
 
       Note that, if $X$ is  an enumerable subset of $\omega$, then also $X^{\chi, k}$ is an enumerable subset of $\omega$.
      
     \begin{corollary}\label{C:PH}  One may prove, in $\mathsf{BIM} + \forall k>0[\mathbf{IRT}(k)]$:
     
       $\forall k>0\forall r>0\forall \chi:\omega\rightarrow r\forall \alpha[\mathit{Almost}\mathit{full}_\omega(E_\alpha)\rightarrow \mathit{Almost}\mathit{full}_\omega\bigl((E_\alpha)^{\chi, k}\bigr)].$
       
     \end{corollary}
     
     \begin{proof} Let $k>0$ be given. 
     We use induction on $r$. Note that the case $r=1$ is trivial. Now assume $r\ge 1$ is given such that the case $r$ of the statement has been established.
     
      Let $\chi$ be given such that $\forall n[\chi(n) < r+1]$. 
      Define $\chi_0$ such that, for all $n$, if $\chi(n)<r$, then $\chi_0(n)=\chi(n)$, and, if $\chi(n)=r$, then $\chi_0(n)=r-1$. Define $\chi_1$ such that, for all $n$, if $\chi(n)>0$, then $\chi_1(n)=\chi(n)-1$, and, if $\chi(n)=0$, then $\chi_1(n)=0$.  Note that $ \forall i<2\forall n[\chi_i(n) < r]$.
     Let $\alpha$ be given such that $E_\alpha\subseteq[\omega]^{<\omega}$ is $\omega$-almost-full. 
    Note that $(E_\alpha)^{\chi_0, k}$ and $(E_\alpha)^{\chi_1, k}$ are enumerable subsets of $\omega$, and according to
     the induction hypothesis,  $\omega$-almost-full subsets of $[\omega]^{<\omega}$.
     
     Let $\zeta$  in $[\omega]^\omega$ be given. Note that, for each $s$, if 
      $\zeta \circ s \in (E_\alpha)^{\chi_0, k}$, then \textit{either} $\zeta \circ s \in (E_\alpha)^{\chi, k}$, \textit{or}, for some $t \in [\omega]^k$, $t(k-1) < \mathit{length}(s)$ and $\chi(\zeta\circ s \circ t) = r$, and, if $\zeta \circ s \in (E_\alpha)^{\chi_1, k}$, then \textit{either} $\zeta \circ s \in (E_\alpha)^{\chi, k}$, \textit{or}, for some $t \in [\omega]^k$, $t(k-1) < \mathit{length}(s)$ and $\chi(\zeta\circ s \circ t) = 0$.
      
      Again, let $\zeta$  in $[\omega]^\omega$ be given.
       Let $QED$ be the statement: $\exists s \in [\omega]^{<\omega}[\zeta \circ s \in X^{\chi, k}]$. 
      Note that,  for each $\eta$ in $[\omega]^\omega$, there exists $s$ such that $\zeta \circ \eta \circ s \in X^{\chi_0, k}$ and, therefore, \textit{either} $QED$ \textit{or}  $\exists u\in[\omega]^k[\chi(\zeta \circ \eta \circ u) = r]$. 
      Define $Y_0:=\{t\in[\omega]^k\mid \chi(\zeta \circ t) = r\;\vee\;QED\}.$ Note that $Y_0$ is $k$-almost-full and an enumerable subset of $\omega$.
      Also define $Y_1:=\{t\in[\omega]^k\mid \chi(\zeta \circ t) = 0\;\vee\;QED\}.$ Note that also $Y_1$ is $k$-almost-full and an enumerable subset of $\omega$.
    Using Corollary \ref{C:ramseyenumer}, conclude that $Y_0 \cap Y_1$ is $k$-almost-full. Find $t$ such that $\zeta \circ t \in Y_0 \cap Y_1$. Then $\chi(\zeta\circ t) = r$ or $QED$, and $\chi(\zeta\circ t) = 0$ or $QED$. Note that $r >0$ and $QED$.
     
     We thus see that  $\forall \zeta \in[\omega]^\omega\exists s[\zeta\circ s \in (E_\alpha)^{\chi, k}]$, i.e. $(E_\alpha)^{\chi,k}$ is $\omega$-almost-full. \end{proof}
\begin{definition} Let $r,k$ be positive integers.  

Let $c$ be given such that $\forall n< \mathit{length}(c)[c(n) < r]$.

For all  $X\subseteq[\omega]^{<\omega}$,   $X^{c, k}$ is the set of all $s$ in $X$ that are $c, k$-\emph{monochromatic}, i.e. for all $u, v$ in $[\omega]^k$, if $u(k-1) < \mathit{length}(s)$ and $v(k-1) < \mathit{length}(s)$, then $s\circ u < \mathit{length}(c)$ and $s \circ  v < \mathit{length}(c)$ and $c(s\circ u) = c(s \circ v)$.  \end{definition}  
     
     Recall that $\mathbf{PH}$  is the statement: $\forall r>0\forall k>0\forall n \exists M[ M \rightarrow_\ast (n)^k_r]$.
   \begin{theorem}\label{T:PH}  
   $\mathsf{BIM} + \forall k>0[\mathbf{IRT}(k)]\vdash \mathbf{PH}.$
   \end{theorem}
   \begin{proof}
  Let $r, k, n$ be positive integers. Note that $\mathcal{F}:=\{\chi\mid \forall m[\chi(m)<r]\}$  is a fan. Define $$X:=\{s\in[\omega]^{<\omega}\mid\mathit{length}(s) \ge n\;\wedge\;\mathit{length}(s) \ge s(0)\}.$$ Note that $X$ is $\omega$-almost-full and a decidable subset of $\omega$. According to Corollary \ref{C:PH}, for each $\chi$ in $\mathcal{F}$,  the set $X^{\chi, k}$ is $\omega$-almost-full, and, in particular,  $\exists s[s \in X^{\chi, k}]$.  Define $$B:=\{c\mid \forall m <\mathit{length}(c)[c(m) < r]\;\wedge\;\exists s[s\in X^{c,k}]\}.$$ Note that $B$ is a decidable subset of $\omega$ and a bar in   $\mathcal{F}$. 
  
  Note that,  in $\mathsf{BIM} + \forall k>0[\mathbf{IRT}(k)]$ one  may prove the Fan Theorem $\mathbf{FT}$, see Theorem \ref{T:irt}, Corollaries \ref{C:sumup3} and \ref{Cor:sumup2} and  Theorems \ref{T:oihb} and \ref{T:hbft}.
  
   Using $\mathbf{FT}$, find $M$ such that $\forall\chi\in\mathcal{F} \exists m\le M[\overline \chi m \in B]$, and, therefore, \\$\forall \chi \in \mathcal{F}[\overline \chi M\in B]$. Assume $c: [M]^k \rightarrow r$. Note that $M \le \mathit{length}(c)$. Find $\chi$ in $\mathcal{F}$ such that $\overline \chi M = \overline c M$. Find $s$ in $X^{\overline \chi M, k}$ and note that $s \in X^{c,k}$. We thus see that $\forall c:[M]^k \rightarrow r\exists s[length(s) \ge n \;\wedge \; length(s)\ge s(0) \;\wedge\; s \;is\; c,k$-$monochromatic]$, i.e.  $M \rightarrow_\ast (n)^k_r$.
 \end{proof}      
       
       It is a famous fact that the Paris-Harrington Theorem can not be proven in classical or intuitionistic arithmetic, see \cite{paris-harrington}, that is: $\mathsf{BIM} \nvdash \mathbf{PH}$. Using Troelstra's result that $\mathbf{FT}$  is conservative over intuitionistic arithmetic, see \cite{troelstra}, we  conclude: 
       \begin{corollary}\label{Cor:incomplete1} $\mathsf{BIM} \nvdash \mathbf{PH}$ and $\mathsf{BIM} \nvdash \mathbf{FT} \rightarrow \mathbf{AppFT}$. \end{corollary}
  Note that the second part of Corollary \ref{Cor:incomplete1} is an easy consequence of the second part of Corollary \ref{Cor:incomplete}.

       Note that, in $\mathsf{BIM}$,  the axiom scheme of induction is not restricted to arithmetical formulas. The classical system $\mathsf{ACA_0}$ (implicitly) has such a restriction and $\mathsf{ACA}_0$  \textit{is} conservative over classical arithmetic, see \cite[page 367, Remark IX]{Simpson}. In $\mathsf{ACA_0}$, one may prove  (proper versions of) $\mathbf{CRT}(1)$ and  $\forall k[\mathbf{CRT}(k+1) \rightarrow \mathbf{CRT}(k+2)]$ and, therefore,  $\mathbf{RT}(3)$, but not the Paris-Harrington Theorem and not $\forall k>0[\mathbf{CRT}(k)]$, see \cite[Section III.7]{Simpson}.

\section{Markov's Principle}
       
      We consider two \textit{semi-classical} axioms: Markov's Principle and Kuroda's Principle. There is no good argument why either one of these principles should be taken as an axiom for  constructive arithmetic or analysis.

\begin{definition}      \emph{Markov's Principle} $(\mathsf{MP_1})$ is the statement 

$\forall \alpha[\neg \neg \exists n[\alpha(n) \neq 0] \rightarrow \exists n[\alpha(n) \neq 0]]$.

\smallskip
      \emph{Kuroda's Principle of Double Negation Shift} $(\mathbf{DNS}_0)$ is the statement

      for every subset $R$ of $\omega$, $\forall n [\neg \neg R(n)] \rightarrow \neg \neg \forall n[R(n)]$,
      
     or, equivalently,
     
      \textit{for every subset $R$ of $\omega$, $\neg \neg \forall n[R(n) \vee \neg R(n)]$}.  
       
     \smallskip  
       $X\subseteq\omega$ is $X$ \emph{nearly-decidable}, or \emph{classically decidable}\footnote{This expression is used in \cite{moschovakis}.}, if and only if
       
        $\neg \neg \exists \beta[D_\beta = X]$, that is, $\neg \neg \exists \beta \forall n[ n \in X \leftrightarrow \beta(n) \neq 0]$.  
        
  \smallskip $\mathbf{\Sigma}^0_1$-$\mathbf{ND}$ is the statement
         
         \textit{Every enumerable subset of $\omega$ is nearly-decidable}, $\forall \gamma \neg \neg \exists \beta [E_\gamma = D_\beta]$.\end{definition}
       
       \begin{theorem}\label{T:solovay} In $\mathsf{BIM} + \mathsf{MP_1}$, the following statements are equivalent:
       \begin{enumerate}[\upshape (i)]
       \item $\mathbf{EnDec?!}$.

       \item $\mathbf{\Sigma}^0_1$-$\mathbf{ND}$.
       \item $\mathbf{\Sigma}^0_1$-$\mathbf{BI}$.
       \end{enumerate}
       \end{theorem}
       
    \begin{proof} (i)  $\Rightarrow$ (ii). Assume $\mathbf{EnDec?!}$.  Let $\gamma$ be given and assume that   $\neg \exists \beta[D_\beta =E_\gamma]$. 
    Let $\beta$ be given such that  $D_\beta \subseteq E_\gamma$. Conclude that $\neg (E_\gamma \subseteq D_\beta]$, and  $\neg \forall n[n \in E_\gamma \rightarrow \beta(n) \neq 0]$, and $\neg \forall n \forall j[\gamma(j) = n+1 \rightarrow \beta(n)\neq 0]$, and $\neg \neg \exists p[ \gamma(p') = p'' +1\; \wedge\; \beta(p'') =0]$.
    Use $\mathsf{MP}_1$ and conclude that $\exists p[\gamma(p') = p'' +1  \;\wedge \; \beta(p'') = 0]$ and $\exists q[q \in E_\gamma\setminus D_\beta]$.
     We thus see that $\forall \beta \in 2^\omega[D_\beta \subseteq E_\gamma \rightarrow \exists q[q \in E_\gamma \setminus D_\beta]]$. 
     Using $\mathbf{EnDec?!}$, we conclude that $E_\gamma = \omega$, i.e. $E_\gamma = D_{\underline 1}$. Contradiction.
   We thus see that   $\forall \gamma\neg\neg \exists \beta[ E_\gamma = D_\beta]$, i.e.  $\mathbf{\Sigma}^0_1$-$\mathbf{ND}$.
    
    \smallskip
    (ii) $\Rightarrow$ (iii). Assume $\mathbf{\Sigma}^0_1$-$\mathbf{ND}$. Let $\gamma$ be given such that $Bar_{\omega^\omega}(E_\gamma)$ and  $\forall s [s\in E_\gamma \leftrightarrow \forall n[s\ast\langle n \rangle \in E_\gamma]]$.  Assume we find  $\beta$ such that $E_\gamma = D_\beta$. 
    Note that $\forall s[\beta(s) = 0 \rightarrow \neg \forall n[\beta(s\ast\langle n \rangle)\neq 0]]$, and: $\forall s[\beta(s) = 0 \rightarrow \neg \neg \exists n[\beta(s \ast \langle n \rangle) = 0]]$. Using $\mathsf{MP}_1$,   conclude that $\forall s[\beta(s) = 0 \rightarrow\exists n[\beta(s\ast\langle n \rangle) = 0 ]]$. 
    
     Assume $\beta(\langle \;\rangle) = 0$. Find $\delta$  such that, for each $n$, $\delta(n)$ is the least $p$ such that $\beta(\overline \delta n \ast \langle p \rangle)= 0$. Note that $\forall n[\beta(\overline \delta n) = 0]$, and: $\delta$ does not meet $D_\beta$ and $E_\gamma = D_\beta$ is not a bar in $\omega^\omega$. Contradiction. We have to conclude that  $\beta(\langle \; \rangle)\neq 0$ and: $\langle \; \rangle \in E_\gamma$.
     
      We thus see that, if $\exists \beta \in 2^\omega[ E_\gamma = D_\beta]$, then $0=\langle \; \rangle  \in E_\gamma$. 
     Using $\mathbf{\Sigma}^0_1$-$\mathbf{ND}$, note that $\neg \neg \exists \beta \in 2^\omega[E_\gamma = D_\beta]$, and conclude that  $\neg \neg(\langle \; \rangle \in E_\gamma)$, i.e. $\neg\neg \exists p[\gamma(p) =1]$. Using $\mathsf{MP}_1$ once more,  conclude that  $\exists p[\gamma(p) = 1]$ and $0=\langle\;\rangle \in E_\gamma$. 
    
     Conclude that, for all $\gamma$, if $Bar_{\omega^\omega}(E_\gamma)$ and $\forall s[s\in E_\gamma \leftrightarrow \forall n[s\ast\langle n \rangle \in E_\gamma]]$, then $\langle\;\rangle \in E_\gamma$, i.e. $\mathbf{\Sigma}^0_1$-$\mathbf{BI}$.
     
     \smallskip
     (iii) $\Rightarrow$ (i). See Theorem \ref{T:ebioi01} and Corollary  \ref{Cor:sumup}.
	\end{proof}

    The surprising observation that, in a context like $\mathsf{BIM} + \mathsf{MP_1}$, $\mathbf{\Sigma^0_1}$-$\mathbf{BI}$ implies $\mathbf{\Sigma}^0_1$-$\mathbf{ND}$   is due to R. Solovay, see Lemma 5.3 in \cite{moschovakis}.  J.R.~Moschovakis made me see that $\mathsf{BIM} +\mathsf{MP}_1 \vdash \mathbf{\Sigma}^0_1$-$\mathbf{ND}\rightarrow \mathbf{EnDec?!}$.

  \medskip
  \begin{definition}We extend the language of $\mathsf{BIM}$ by introducing an infinite sequence of binary predicate symbols $S^0_1$, $P^0_1, S^0_2, P^0_2, \ldots$ with the following defining axioms: 
   \begin{enumerate}[\upshape (i)]
   \item $\forall m[ S^0_1(m, \gamma)\leftrightarrow m \in E_\gamma]$, 
   \item $\forall m[ P^0_1 (m,\gamma)\leftrightarrow \neg S^0_1(m,\gamma)]$, and
   \item for each $n>0$, $\forall m[ S^0_{n+1}(m,\gamma)\leftrightarrow \exists p[ P^0_{n}((m,p), \gamma)]]$, and, 
   \item for each $n>0$, $\forall m[ P^0_{n+1}(m,\gamma)\leftrightarrow \forall p[  S^0_{n}((m,p),\gamma)]]$.
   \end{enumerate}
   
   A subset $X$ of $\omega$ is \emph{(positively) arithmetical} if and only if there exist $n>0$, $\gamma$ such that $\forall m[m \in X \leftrightarrow S^0_n(m, \gamma)]$ or  $\forall m[m \in X \leftrightarrow P^0_n(m, \gamma)]$.

 \smallskip For each $n$, we let $\mathbf{\Sigma}^0_n$-$\mathbf{ND}$ be the   the statement
 
  \textit{Every $\mathbf{\Sigma}^0_n$-subset of $\omega$ is nearly-decidable},  $\forall \gamma \neg \neg \exists \beta\forall m[S^0_n(m, \gamma) \leftrightarrow  \beta(m) \neq 0]$,
  
   and  we let $\mathbf{\Pi}^0_n$-$\mathbf{ND}$  be the statement
  
   \textit{Every $\mathbf{\Pi}^0_n$-subset of $\omega$ is nearly-decidable},  $\forall \gamma \neg \neg \exists \beta\forall m[P^0_n(m, \gamma) \leftrightarrow  \beta(m) \neq 0]$. 
   \end{definition}
   \begin{theorem}\label{T:last} For each $n$, $\mathsf{BIM}+ \mathbf{\Sigma}_1^0$-$\mathbf{ND}\vdash \mathbf{\Sigma}_n^0$-$\mathbf{ND}\;\wedge\;\mathbf{\Pi}_n^0$-$\mathbf{ND}$.
   
      \end{theorem}
       
  \begin{proof}  We use induction\footnote{in the metalanguage}. Assume $\mathbf{\Sigma}_1^0$-$\mathbf{ND}$.

   Let $\gamma, \beta$ be given such that $\forall m[S^0_1(m, \gamma)\leftrightarrow \beta(m) \neq 0]$. Define $\delta$ such that $\forall m[\delta(m) =0 \leftrightarrow \beta(m) \neq 0]$. Note that $\forall m[P^0_1(m,\gamma) \leftrightarrow \delta(m) \neq 0]]$. Conclude that,  if $\exists \beta\forall m[S^0_1(m,\gamma) \leftrightarrow \beta(m) \neq 0]$, then $\exists \beta\forall m[P^0_1(m,\gamma) \leftrightarrow \beta(m) \neq 0]$. Note that $\neg\neg\exists \beta\forall m[S^0_1(m,\gamma) \leftrightarrow \beta(m) \neq 0]$ and conclude that $\neg\neg\exists \beta\forall m[P^0_1(m,\gamma) \leftrightarrow \beta(m) \neq 0]$. Conclude $\mathbf{\Pi}^0_1$-$\mathbf{ND} $.
   
  We thus see that $\mathsf{BIM}+ \mathbf{\Sigma}_1^0$-$\mathbf{ND}\vdash \mathbf{\Pi}_{1}^0$-$\mathbf{ND}$.
  
 \smallskip Assume $n>0$ and $\mathbf{\Pi}_n^0$-$\mathbf{ND}$. Let $\gamma,  \beta$ be given such that $\forall m[P^0_n(m, \gamma) \leftrightarrow \beta(m) \neq 0]$.  Note that $\forall m[ S^0_{n+1}(m, \gamma)\leftrightarrow \exists p[\beta\bigl((m,p)\bigr)\neq 0]]$. Define $\delta$  such that, for all $m,p$, if $\beta\bigl((m,p)\bigr)\neq 0$, then $\delta\bigl((m,p)\bigr) = m+1$, and, if not, then $\delta\bigl((m,p)\bigr) = 0$. Note that $\forall m[ m \in E_\delta\leftrightarrow S^0_{n+1}(m,\gamma)]$, i.e.  $\forall m[S^0_1(m,\delta)\leftrightarrow S^0_{n+1}(m,\gamma)]$. Use $\mathbf{\Sigma}^0_1$-$\mathbf{ND}$ and conclude that $\neg\neg\exists \beta\forall m[S^0_{n+1}(m, \gamma)\leftrightarrow \beta(m) \neq 0]$. Conclude that, if $\exists \beta\forall m[P^0_n(m,\gamma) \leftrightarrow \beta(m) \neq 0]$, then $\neg\neg\exists \beta\forall m[S^0_{n+1}(m,\gamma) \leftrightarrow \beta(m) \neq 0]$. Note that $\neg\neg\exists \beta\forall m[P^0_n(m,\gamma) \leftrightarrow \beta(m) \neq 0]$ and conclude that $\neg\neg\exists \beta\forall m[S^0_{n+1}(m,\gamma) \leftrightarrow \beta(m) \neq 0]$. Conclude  $\mathbf{\Sigma}^0_{n+1}$-$\mathbf{ND} $.
 
  We thus see that $\mathsf{BIM}+ \mathbf{\Sigma}_1^0$-$\mathbf{ND}\vdash \mathbf{\Pi}_n^0$-$\mathbf{ND}\rightarrow\mathbf{\Sigma}_{n+1}^0$-$\mathbf{ND}$.
  
  Assume $n>0$ and $\mathbf{\Sigma}_n^0$-$\mathbf{ND}$. Let $\gamma,  \beta$ be given such that $\forall m[S^0_n(m, \gamma) \leftrightarrow \beta(m) \neq 0]$.  Note that $\forall m[ P^0_{n+1}(m, \gamma)\leftrightarrow \forall p[\beta\bigl((m,p)\bigr)\neq 0]]$. Define $\delta$  such that, for all $m,p$, if $\beta\bigl((m,p)\bigr)=0$, then $\delta\bigl((m,p)\bigr) = m+1$, and, if not, then $\delta\bigl((m,p)\bigr) = 0$. Note that $\forall m[ m \notin E_\delta\leftrightarrow P^0_{n+1}(m,\gamma)]$, i.e.  $\forall m[P^0_1(m,\delta)\leftrightarrow P^0_{n+1}(m,\gamma)]$. Use $\mathbf{\Pi}^0_1$-$\mathbf{ND}$ and conclude that $\neg\neg\exists \beta\forall m[P^0_{n+1}(m, \gamma)\leftrightarrow \beta(m) \neq 0]$. Conclude that, if $\exists \beta\forall m[S^0_n(m,\gamma) \leftrightarrow \beta(m) \neq 0]$, then $\neg\neg\exists \beta\forall m[P^0_{n+1}(m,\gamma) \leftrightarrow \beta(m) \neq 0]$. Note that $\neg\neg\exists \beta\forall m[S^0_n(m,\gamma) \leftrightarrow \beta(m) \neq 0]$ and conclude that $\neg\neg\exists \beta\forall m[P^0_{n+1}(m,\gamma) \leftrightarrow \beta(m) \neq 0]$. Conclude $\mathbf{\Pi}^0_{n+1}$-$\mathbf{ND} $.
  
   We thus see that $\mathsf{BIM}+ \mathbf{\Sigma}_1^0$-$\mathbf{ND}\vdash \mathbf{\Sigma}_n^0$-$\mathbf{ND}\rightarrow\mathbf{\Pi}_{n+1}^0$-$\mathbf{ND}$.
   \end{proof}
      
 Theorem \ref{T:last} is also due to R. Solovay, see Lemma 5.5 and Theorem 5.6 in \cite{moschovakis}, see also \cite{moschovakis2}.  J.R.~Moschovakis showed that $\mathsf{BIM} + \mathbf{\Sigma}^0_1$-$\mathbf{ND}$ proves the theorem that the constructive arithmetical hierarchy is proper and also  an intuitionistic version of $\Delta^1_1$-comprehension, see Corollaries 5.8  and 5.9 in \cite{moschovakis}.
	
 \begin{theorem}\label{T:hurrah} $\mathsf{BIM}+\mathsf{MP}_1\vdash \mathbf{\Sigma}^0_1$-$\mathbf{ND} \rightarrow \overleftarrow{\mathbf{BW}}$. \end{theorem}
   
   \begin{proof}
Assume that $\gamma \in \mathbb{Q}^\omega$ and $\forall \zeta \in [\omega]^\omega\exists n[|\gamma\circ\zeta(n+1) -_\mathbb{Q} \gamma\circ\zeta(n)| 
>_\mathbb{Q} \frac{1}{2^n}]$. \\Define $\delta$ in $\mathbb{Q}^\omega$ such that, for each $n$, \textit{if} $\forall i \le n[\gamma(i)\in [-1,1]]$, then $\delta(n) = \gamma(n)$, and, \textit{if not}, then $\delta(n) = 0_\mathbb{Q}$. Note that $\forall n[\delta(n)\in [-1,1]]$.
\\Define $C:=\{s \in \mathbb{S}\mid s\sqsubseteq_\mathbb{S} (-1_\mathbb{Q}, 1_\mathbb{Q})$ {\it and } $\exists m \forall n>m[\delta(n) <_\mathbb{Q} s'\;\vee\; s''<_\mathbb{Q} \delta(n) ]\}$.  Recall that, for each $s$ in $\mathbb{S}$, $L(s) = (s'', \frac{s'+_\mathbb{Q} s''}{2})$ and $R(s) = (\frac{s'+_\mathbb{Q} s''}{2}, s'')$. Note that,  for each $s$ in $\mathbb{S}$ such that $s\sqsubseteq (-1_\mathbb{Q}, 1_\mathbb{Q})$, if both $L(s) \in C$ and $R(s) \in C$, then $s \in C$.  Also note that $C$ is a $\mathbf{\Sigma}^0_2$-subset of $\omega$ and find $\varepsilon$ such that $\forall s[s \in C \leftrightarrow S^0_2(s,\varepsilon)]$.  Assume we find $\eta$ such that $C=D_\eta$, i.e. $\forall s[s \in C \leftrightarrow \eta(s) \neq 0]$. Note that,   for each $s$ in $\mathbb{S}$ such that $s\sqsubseteq (-1_\mathbb{Q}, 1_\mathbb{Q})$, if  $\eta\bigl(L(s)\bigr) \neq 0$ and $\eta\bigl(R(s)\bigr) \neq 0$, then $\eta(s) \neq 0$, and, therefore, if $\eta(s) =0$ then either $\eta\bigl(L(s)\bigr)=0$ or $\eta\bigl(R(s)\bigr)=0$.  Note that $\eta\bigl((0_\mathbb{Q}, 1_\mathbb{Q})\bigr) =0$. Define $\lambda$ such that $\lambda(0) = (-1_\mathbb{Q}, 1_\mathbb{Q})$, and, for each $n$, if $\eta\bigl(L(\lambda(n))\bigr) =0$, then $\lambda(n+1) = L\bigl(\lambda(n)\bigr)$, and, if not, then $\lambda(n+1) = R\bigl(\lambda(n)\bigr)$.
Note that,  for each $n$, $\eta\bigl(\lambda(n)\bigr) = 0$. 

 Let $s$ in $\mathbb{S}$ be given such that $ s\sqsubseteq_\mathbb{S} (-1_\mathbb{Q}, 1_\mathbb{Q})$ and $\eta(s)=0$. Then $\neg \exists m \forall n>m[\delta(n)
<_\mathbb{Q} s'\;\vee\; s'' <_\mathbb{Q} \delta(n)]$ and, therefore, $\forall m \neg\neg\exists n>m[s'\le_\mathbb{Q} \delta(n)
\le_\mathbb{Q} s'']$. Use $\mathsf{MP}_1$ and conclude that $\forall m \exists n>m[s'\le_\mathbb{Q} \delta(n)
\le_\mathbb{Q} s'']$. 

Conclude that $\forall p\forall m \exists n>m[ \bigl(\lambda(p)\bigr)'\le_\mathbb{Q} \delta(n) \le_\mathbb{Q} \bigl(\lambda(p)\bigr)'']$. Define $\zeta$ such that $\zeta(0) =0$ and, for each $n$, $\zeta(n+1)$ is the least $p$ such that $p >\zeta(n)$ and  $\bigl(\lambda(n)\bigr)'\le_\mathbb{Q}\delta(p)\le_\mathbb{Q} \bigl(\lambda(n)\bigr)''$. Note that,  for each $n$, $|\delta\circ\zeta(n+2)-_\mathbb{Q}\delta\circ\zeta(n+1)|\le_\mathbb{Q}\frac{1}{2^n}$.  Find $n$ such that $|\gamma\circ\zeta(n+2)-_\mathbb{Q}\gamma\circ\zeta(n+1)|>_\mathbb{Q}\frac{1}{2^n}$ and conclude that $\exists i\le \zeta(n+2)[\gamma(i)\neq \delta(i)]$ and $\exists i\le \zeta(n+2)[|\gamma(i)|>_\mathbb{Q} 1_\mathbb{Q}]$. 

We thus see that,  if $\exists \eta\forall s[S^0_2(s, \varepsilon) \leftrightarrow \eta(s) \neq 0]$, then $\exists n[|\gamma(n)| >_\mathbb{Q} 1_\mathbb{Q} ]$. Using $\mathbf{\Sigma}^0_1$-$\mathbf{ND}$ and Theorem \ref{T:last}, we conclude that $\neg\neg\exists \eta\forall s[S^0_2(s, \varepsilon) \leftrightarrow \eta(s) \neq 0]$. We thus find that $\neg \neg \exists n[|\gamma(n)| >_\mathbb{Q} 1_\mathbb{Q} ]$, and, using $\mathsf{MP}_1$ once more, that $\exists n[|\gamma(n)| >_\mathbb{Q} 1_\mathbb{Q}]$. 

We may conclude $\overleftarrow{\mathbf{BW}}$. \end{proof}
	
	The first item of the next theorem occurs already in \cite{veldman2005}, Section 3.20.

  \begin{theorem}\label{T:verylast}$\;$ \begin{enumerate}[\upshape (i)]  \item $\mathsf{BIM}\vdash \forall \alpha[\forall \zeta \in [\omega]^\omega \exists n[\zeta(n) \notin D_\alpha]\rightarrow \neg\neg\exists n\forall m>n[m \notin D_\alpha]]$.    
   \item $\mathsf{BIM} \vdash\mathsf{MP}_1\leftrightarrow \forall \alpha[\neg\neg\exists n\forall m>n[m \notin D_\alpha]\rightarrow\forall \zeta \in [\omega]^\omega \exists n[\zeta(n) \notin D_\alpha] ]$.   
    \item $\mathsf{BIM}+\mathsf{MP}_1+\mathbf{\Sigma}^0_1$-$\mathbf{ND}\vdash\forall \beta[Almfan(\beta)\rightarrow \neg\neg Fan(\beta)]$.   
    \item  $\mathsf{BIM} +\mathsf{MP}_1 + \mathbf{AppFT}\vdash \mathbf{AlmFT}$. \end{enumerate} \end{theorem}
   
   \begin{proof} (i)  
   
   Let $\alpha$ be given such that $D_\alpha$ is \textit{almost-finite}, i.e. $\forall \zeta \in [\omega]^\omega \exists n[\zeta(n) \notin D_\alpha]$. Assume that $\neg\exists n\forall m>n[m \notin D_\alpha]$. Then $\forall n\neg \neg \exists m>n[m \in D_\alpha]$, and, by $\mathsf{MP}_1$, $\forall n \exists m>n[m \in D_\alpha]$. Define $\zeta$ such that $\zeta(0)=\mu p[p \in D_\alpha]$ and, for each $n$, $\zeta(n+1) = \mu p[p>\zeta(n) \;\wedge\; p \in D_\alpha]$. Note that $\forall n[\zeta(n) \in D_\alpha ]$. Contradiction. Conclude that $\neg\neg\exists n\forall m>n[m \notin D_\alpha]$, i.e. $D_\alpha$ is \textit{not-not-finite}.
   
  \smallskip (ii) Assume $\mathsf{MP}_1$.  Let $\alpha$ be given such that $D_\alpha$ is {\it not-not-finite}, i.e. $\neg\neg\exists n\forall m>n[m \notin D_\alpha]$. Assume  $\zeta \in [\omega]^\omega$. Note that,  \textit{if} $\exists n\forall m>n[m \notin D_\alpha]$, \textit{then} $\exists n[\zeta(n) \notin D_\alpha]$. Conclude that $\neg\neg \exists n[\zeta(n) \notin D_\alpha]$ and, using $\mathsf{MP}_1$, that 
   $\exists n[\zeta(n) \notin D_\alpha]$. We thus see that $\forall \zeta \in [\omega]^\omega \exists n[\zeta(n) \notin D_\alpha]$, i.e. $D_\alpha$ is almost-finite.
   
 \smallskip  Now assume that $\forall \alpha[\neg\neg\exists n\forall m>n[m \notin D_\alpha]\rightarrow \forall \zeta \in [\omega]^\omega \exists n[\zeta(n) \notin D_\alpha]]$, i.e. every decidable subset of $\omega$ that is not-not-finite is also almost-finite.  Let $\alpha$ be given such that $\neg\neg \exists n[\alpha(n) \neq 0]$. Define $\beta$ such that $\forall n[\beta(n)\neq 0 \leftrightarrow \bigl(\alpha(n)\neq 0 \;\wedge\;\forall i < n[\alpha(i) = 0]\bigr)]$. Note that $\neg\neg\exists n\forall m>n[m \notin D_\beta]$, i.e. $D_\beta$ is not-not-finite. Conclude that  $D_\beta$ is almost-finite and find $n$ such that $n \notin D_\beta$, and, therefore, $\alpha(n) \neq 0$. We thus see that $\exists n[\alpha(n) \neq 0]$. Conclude that $\forall\alpha[\neg\neg\exists n[\alpha(n)\neq 0]\rightarrow \exists n[\alpha(n)\neq 0]]$, i.e.  $\mathsf{MP}_1$.
 
   \smallskip
  (iii) Assume $\mathsf{MP}_1$. Let $\beta$ be given such that $Almostfan(\beta)$, i.e. $Spr(\beta)\;\wedge\;\forall s[\beta(s)=0 \rightarrow \forall \zeta \in [\omega]^\omega\exists n[\beta(s\ast\langle \zeta(n)\rangle)\neq 0]]$. Using (i), conclude that $\forall s[\beta(s)=0 \rightarrow \neg\neg \exists n \forall m>n  [\beta(s\ast\langle m\rangle)= 0]]$. Assume we find $\delta$ such that $\forall s\forall n[\delta(s\ast\langle n \rangle ) \neq 0 \leftrightarrow \forall m>n[\beta(s\ast \langle m \rangle) \neq 0]]$. Note that $\forall s\neg\neg \exists n[\delta(s\ast\langle n \rangle) \neq 0]$. Using $\mathsf{MP}_1$, conclude that $\forall s \exists n[\delta(s\ast\langle n \rangle) \neq 0]$. Conclude that $Fan^+(\beta)$, i.e. $\beta$ is an {\it explicit} fan-law. Conclude that, if $\exists \delta \forall s\forall n [\delta(s\ast \langle n \rangle)\neq 0 \leftrightarrow  \forall m>n  [\beta(s\ast\langle m\rangle)= 0]]$, then $Fan^+(\beta)$. Using  $\mathbf{\Sigma}^0_1$-$\mathbf{ND}$, and  its consequence  $\mathbf{\Pi}^0_1$-$\mathbf{ND}$, see Theorem \ref{T:last},  conclude that  $\neg\neg \exists \delta \forall s\forall n [\delta(s\ast \langle n \rangle)\neq 0 \leftrightarrow  \forall m>n  [\beta(s\ast\langle m\rangle)= 0]]$. Conclude $\neg\neg Fan^+(\beta)$. 
   
   \smallskip (iv) Assume $\mathsf{MP}_1$ and  $\mathbf{AppFT}$. Note that, in $\mathsf{BIM}$,  $\mathbf{AppFT}$ implies $\mathbf{EnDec?!}$, see Corollary \ref{Cor:sumup2} and, together with $\mathsf{MP}_1$, also $\mathbf{\Sigma}^0_1$-$\mathbf{ND}$, see Corollary \ref{Cor:sumup} and Theorem \ref{T:solovay}.  We will prove $\mathbf{AlmFT}$. Let $\beta$ be given such that $Almfan(\beta)$. Let $\alpha$ be given such that $Thinbar_{\mathcal{F_\beta}}(D_\alpha)$. Note that, if $Fan^+(\beta)$, then $\exists n\forall m>n[m \notin D_\alpha]$, by Theorem \ref{T:ftext}(ii). As, by (iii), $\neg\neg Fan^+ (\beta)$, conclude that $\neg\neg \exists n\forall m>n[m \notin D_\alpha]$. Now use (ii) and conclude that $\forall \zeta \in [\omega]^\omega \exists n[\zeta(n)\notin D_\alpha]$, i.e. $D_\alpha$ is almost-finite.  We thus see that $\forall \beta\forall \alpha[\bigl(Almfan(\beta) \;\wedge\;Thinbar_{\mathcal{F}_\beta}(D_\alpha)\bigr) \rightarrow \forall \zeta \in [\omega]^\omega \exists n[\zeta(n)\notin D_\alpha]]$, i.e. $\mathbf{AlmFT}$.\end{proof} 
 
	\begin{corollary}\label{Cor:sumup3} $\mathsf{BIM}+\mathsf{MP}_1\vdash \mathbf{OI}([0,1]) \leftrightarrow \mathbf{AppFT}\leftrightarrow \mathbf{\Sigma}^0_1$-$\mathbf{ND}\leftrightarrow \mathbf{AlmFT}$. \end{corollary}\begin{proof} Use Corollaries \ref{Cor:sumup} and \ref{Cor:sumup2} and Theorems \ref{T:last}, \ref{T:hurrah} and \ref{T:verylast}.  \end{proof} Some of the equivalences in Corollary \ref{Cor:sumup3} are studied in \cite{ilik0} and \cite{ilik}.
	
	\section{Some notations and conventions and basic facts}\label{S:notation}
\subsection{Finite sequences and infinite sequences of natural numbers}\label{SS:codfs} We assume that the language of $\mathsf{BIM}$ contains constants for the primitive recursive functions and relations 
 and that their defining
 equations have been added to the axioms of $\mathsf{BIM}$.
 In particular,  there is a constant $p$ denoting the function
 enumerating the prime numbers, $p(0)=2, p(1)=3, \ldots$. 
 
 For each $k$, there is a $k$-ary function symbol $\langle \;\rangle_k$    for the function  coding  sequences of natural numbers of length $k$ by
 natural numbers:
 \[\langle m_0, \ldots, m_{k-1} \rangle = \bigl(p(0)\bigr)^{m_0 }\cdot \dots \cdot 
\bigl(p(k-2))\bigr)^{m_{k-2} } \cdot \bigl(p(k-1)\bigr)^{m_{k-1} + 1} - 1\]

In practice, we omit the subscript $k$ and write $\langle\;\rangle$ for $\langle \; \rangle_k$. 

  $\mathit{length}(0) := 0$ and, for each $a>0$,   $length(a)$ is the greatest number $k\le a $ such that $p(k-1)$  divides $a + 1$.

 For each $a$, for each $i$, if $i +1< length(a) $, then $a(i)$ is the greatest number
 $q$ such that $(p(i))^{q}$ divides $a + 1$, and, if $i+1 = length(a)$,
 $a(i)$ is the  greatest number
 $q$ such that $(p(i))^{q +1}$ divides $a + 1$, and if $i\ge\mathit{length}(a)$, then $a(i) = 0$. Observe that 
 $a= \langle a(0), a(1), \dots, a(i -1) \rangle$, where $i = length(a)$.
 
  For each $m$, $\omega^m:=\{a\mid length(a)=m\}$.
  
  For each $m$,  $[\omega]^m:=\{a\in \omega^m\mid \forall i[i+1<m\rightarrow a(i)<a(i+1)]\}$. 
  
For each $a,b$, $a\ast b$ is the number $s$ satisfying $\mathit{length}(s) = \mathit{length}(a) + \mathit{length}(b)$ and, for each $n$, if $n<\mathit{length}(a)$, then $s(n) = a(n)$ and, if $\mathit{length}(a) \le n < \mathit{length}(s)$, then $s(n) = b\bigl(n-\mathit{length}(a)\bigr)$.

 For each $a$, for each $\alpha$, $a\ast \alpha$ is the element $\beta$ of $\omega^\omega$ satisfying, for each $n$, if $n<\mathit{length}(a)$, then $\beta(n) = a(n)$ and, if $\mathit{length}(a) \le n $, then $\beta(n) = \alpha\bigl(n-\mathit{length}(a)\bigr)$.
 
For each $a$, for each $n \le length(a)$ we define: $\overline a(n) = \langle
 a(0), \ldots, a(n-1) \rangle$. If confusion seems unlikely, we sometimes write:
 ``$\overline a n$'' and not: ``$\overline a (n)$''.

 For all $a, b$, 
 $a \sqsubseteq b\leftrightarrow\exists n \le \mathit{length}(b)[a = \overline b n]$, and 
  $a \sqsubset b\leftrightarrow(a \sqsubseteq b\;\wedge\; a \neq b)$.
  
  For each $n$, for all $a, b$, $a \neq_n b\leftrightarrow \exists j<n[j<\min\bigl(length(a), length(b)\bigr)\;\wedge\; a(j) \neq b(j)]$, and $a=_n b\leftrightarrow \neg(a\neq_n b)$.
  
  For each $\alpha$, for each $n$, $\overline{\alpha}(n) := \langle
 \alpha(0), \ldots \alpha(n-1) \rangle$. If confusion seems unlikely, we
 sometimes write: ``$\overline \alpha n$'' and not: ``$\overline \alpha (n)$''.
For each $s$, for each $\alpha$, $s\sqsubset \alpha\leftrightarrow \exists n[\overline \alpha n = s]$.
For each $s$, $\omega^\omega\cap s:=\{\alpha|s\sqsubset \alpha\}$. 
 $2^{<\omega}$ is the set of all natural numbers $s$ coding a \textit{finite binary sequence}, that is, such that, for all $n<\mathit{length}(s)$, $s(n) = 0 \; \vee \; s(n) =1$.
 \textit{Cantor space} $2^\omega$ is the set of all $\gamma$ such that $\forall n [\overline \gamma n \in 2^{<\omega}]$. For each $s$ in $2^{<\omega}$, $2^\omega\cap s :=\{\alpha \in 2^\omega|s \sqsubset \alpha\}$.
 
 For each $\alpha$, for each $s$, $\alpha \circ s $ is the element $t$ of $\omega$ such that $\mathit{length}(t) = \mathit{length}(s)$, and for all $i<\mathit{length}(t)$, $t(i) = \alpha\bigl(s(i)\bigr)$.
 
 For each $\alpha$, for each $\beta$, $\alpha \circ \beta $ is the element $\gamma$ of $\omega^\omega$ satisfying $\forall n[\gamma (n)= \alpha\bigl(\beta(n)\bigr)]$.
 
 For each $X\subseteq \omega$, $X^\omega:=\{\alpha\mid\forall n[\alpha(n)\in X]\}$. 
 
 For all $\alpha$, for all $n$, $\alpha^n$ is the element $\beta$ of $\omega^\omega$ satisfying $\forall m[\beta(m) = \alpha(\langle n, m\rangle)]$. 
 
 For each $\mathcal{X}\subseteq \omega^\omega$, $\mathcal{X}^\omega:=\{\alpha\mid \forall n[\alpha^n \in \mathcal{X}]\}$.
 
\subsection{Lexicographical ordering and Kleene-Brouwer-ordering.}\label{SS:ordering}

 $\;$\\For all $s$, $t$,
 $s<_{lex} t \leftrightarrow
 \exists i[ i< \mathit{length}(s)\;\wedge\; i<
 \mathit{length}(t)\;\wedge\;\overline s i = \overline t i \;\wedge\; s(i) < t(i)]$.
 
 For all $\alpha, \beta$, $\alpha <_{lex} \beta\leftrightarrow\exists n[\overline \alpha n <_{lex} \overline \beta n]$. 
 
 For all $\alpha$, for all $n$, for all $t$ in $\omega^n$, $\alpha<_{lex} t\leftrightarrow \overline \alpha n<_{lex}t$.

 For each  $\mathcal{X}\subseteq\omega^\omega$, for each $\gamma$, $\mathcal{X}_{<_{lex}\gamma}:=\{\alpha \in \mathcal{X}\mid\alpha <_{lex} \gamma\}$.
 
  For all $\mathcal{A},\mathcal{X}\subseteq\omega^\omega$, $\mathcal{A}$ is  \emph{progressive in $\mathcal{X}$} if and only if $\forall \gamma \in \mathcal{X}[\mathcal{X}_{<_{lex}\gamma} \subseteq \mathcal{A}\rightarrow \gamma \in \mathcal{A}]$.
 
   For all $s,t$,  $s<_{KB}t\leftrightarrow (t \sqsubset s\;\vee\; s <_{lex} t)$.
   The ordering $<_{KB}$ is called the \textit{Kleene-Brouwer ordering} of $\omega$, and sometimes the \textit{Lusin-Sierpi{\'n}ski ordering} of $\omega$.
   The Kleene-Brouwer ordering $<_{KB}$ is a decidable and linear ordering of $\omega$:
    $\mathsf{BIM}\vdash\forall s \forall t[s <_{KB}t \;\vee\; s=t\;\vee\; t<_{KB} s]$.
\subsection{Increasing sequences.}\label{SS:inc}

We define,  for each $n$,  $[\omega]^n:=\{s|\mathit{length}(s) = n \;\wedge\; \forall i[i+1<n \rightarrow s(i) < s(i+1)]\}$. We also define: $[\omega]^{<\omega}:=\bigcup\limits_{n \in \omega} [\omega]^n$ and $[\omega]^\omega:=\{\alpha|\forall n[\alpha(n) <\alpha(n+1)]\}$. \subsection{Bars and thin bars.}

For each subset $\mathcal{X}$ of $\omega^\omega$, for each subset $B$ of $\omega$,  \textit{$B$ is a bar in $\mathcal{X}$}, notation: $Bar_\mathcal{X}(B)$, if and only $\forall\alpha\mathcal{X}\exists n[\overline \alpha n \in B]$, and \textit{$B$ is a \emph{thin} bar in $\mathcal{X}$}, notation: $Thinbar_\mathcal{X}(B)$ if and only if $Bar_\mathcal{X}(B)\;\wedge\;\forall s \in B \forall t \in B[s\sqsubseteq t \rightarrow s=t]$.
\subsection{Decidable and enumerable subsets of $\omega$}\label{SS:decen}
For each  $\alpha$,  $D_\alpha:=\{i|\alpha(i)\neq 0\}$.  The set $D_\alpha$ is \textit{the subset of $\omega$ decided by $\alpha$}.   $A\subseteq\omega$ is  a \textit { decidable} subset of $\omega$ if and only if $\exists \alpha[A=D_\alpha]$. 
 For each $a$,   $D_a:=\{i|i < \mathit{length}(a)|a(i) \neq 0\}$. $X\subseteq \omega$ is  a {\it finite} subset of $\omega$ if and only if $\exists a[X=D_a]$. Note that, for each $\alpha$, $D_\alpha = \bigcup\limits_{n \in \omega} D_{\overline \alpha n}$. Note that, for each $\alpha$, $D_\alpha$ is a finite subset of $\omega$ if and only if $\exists m\forall n\ge m[\alpha(n)=0]$. 

For each $\alpha$,  $E_\alpha := \{n| \exists p [\alpha(p) = n+ 1] \}$. The set $E_\alpha$ is  \textit{the subset of $\omega$ enumerated by $\alpha$}.
 $A\subseteq\omega$ is an \textit{enumerable} subset of $\omega$, or $A$ {\it  belongs to the class $\mathbf{\Sigma}^0_1$}, if and only if  $\exists \alpha[A=E_\alpha]$.  
 For each $a$,   $E_a := \{n| \exists p < \mathit{length}(a)[a(p) = n+ 1] \}$.  Note that, for each $a$, $E_a$ is a finite subset of $\omega$ and,  for each $\alpha$, $E_\alpha = \bigcup\limits_{n \in \omega} E_{\overline \alpha n}$.
  \subsection{Open subsets of $\omega^\omega$} For every $\alpha$,  $\mathcal{G}_\alpha:=\{\gamma\mid\exists n[ \overline \gamma n \in D_\alpha]\}$.   $\mathcal{G}\subseteq\omega^\omega$ is an \textit{open} subset of $\omega^\omega$ if and only if $\exists \alpha[\mathcal{G}=\mathcal{G}_\alpha]$. 
   For each $t$ in $\omega$, $\mathcal{G}_t:=\{\gamma\mid \exists n[\overline \gamma n \in D_t]$.
 
 Note that, for every $\alpha$, $\mathcal{G}_\alpha = \bigcup\limits_{n \in \omega} \mathcal{G}_{\overline \alpha n}$. Also note that,  for every  $\mathcal{X}\subseteq\omega^\omega$,  $\forall \alpha[\mathcal{X} \subseteq\mathcal{G}_\alpha\leftrightarrow Bar_{\mathcal{X}}(D_\alpha)]$ and  $\forall t[\mathcal{X} \subseteq\mathcal{G}_t\leftrightarrow Bar_{\mathcal{X}}(D_t)]$. 
 
 \subsection{Spreads, fans and explicit fans}\label{SS:fans} For each $\beta$,  $\mathcal{F}_\beta:=\{\alpha\mid \forall n [\beta(\overline \alpha n) = 0]\}$.  
 $\mathcal{F}\subseteq \omega^\omega$ is a \textit{closed} subset of $\omega^\omega$ if and only if $\exists\beta[\mathcal{F}=\mathcal{F}_\beta]$.   $\beta$  \textit{is a spread-law}, notation $Spr(\beta)$,  if and only if $\forall s[\beta(s) = 0 \leftrightarrow \exists n[\beta(s\ast\langle n \rangle) =0]]$.   $\mathcal{F}\subseteq\omega^\omega$ is a \textit{spread} or a \textit{located-and-closed} subset of $\omega^\omega$ if and only if  
   $\exists \beta[Spr(\beta)\;\wedge\; \mathcal{F} =\mathcal{F}_\beta]$. $\mathcal{X}\subseteq \omega^\omega$ is {\it inhabited} if and only if $\exists \alpha[\alpha \in \mathcal{X}]$.  Note that,  for every $\beta$, if $Spr(\beta)$, then $\mathcal{F}_\beta$ is inhabited if and only if $\beta(\langle \;\rangle) = 0$, and $\mathcal{F}_\beta = \emptyset$ if and only if $\beta(\langle \;\rangle)\neq 0$. 
   $\beta$ \textit{is a finitary spread-law}, or \textit{a fan-law}, notation: $Fan(\beta)$, if and only if $Spr(\beta)$ and $\forall s\exists n \forall m[\beta(s\ast\langle m \rangle)=0\rightarrow m\le n]$.  $\beta$ \textit{is an explicit fan-law}, notation $Fan^+(\beta)$, if and only if $Spr(\beta)$ and $\exists \gamma \forall s \forall m[\beta(s\ast\langle m \rangle)=0\rightarrow m\le \gamma(s)]$. One may prove in $\mathsf{BIM}$ that, for every $\beta$, if $Fan(\beta)$, then $Fan^+(\beta)$ if and only if  $\exists \delta \forall n\forall s\in\omega^n[ \beta(s) = 0) \rightarrow s < \delta(n)]$.   $\mathcal{F}\subseteq \omega^\omega$ is an \textit{(explicit)  fan}  if and only if there exists an (explicit) fan-law $\beta$ such that  $\mathcal{F}=\mathcal{F}_\beta$.

 \subsection{Real numbers}\label{SS:reals} The development of real analysis in $\mathsf{BIM}$ has been described in \cite{veldman2011b}. Recall that in Subsection \ref{SS:bimaxioms},  we introduced the following notation: for each $n$, $n':= K(n)$ and $n'' := L(n)$, and for all $m, n$, $(m,n) := J(m, n)$. The last part of Axiom 2 then reads as follows: $\forall m \forall n[(m,n)'= m \;\wedge\; (m,n)'' =n\;\wedge \; n = (n', n'')]$. 
 
 For all $m,n$ in $\omega$, $m =_\mathbb{Z} n \leftrightarrow m'+ n'' = m'' + n'$ and $m <_\mathbb{Z} n \leftrightarrow m'+ n'' < m'' + n'$  and $m \le_\mathbb{Z} n \leftrightarrow m'+ n'' \le m'' + n'$ and $m+_\mathbb{Z} n := (m'+n'', m''+n')$ and $m-_\mathbb{Z} n := (m'-n'', m''-n')$ and $m\cdot_\mathbb{Z} n :=(m'\cdot n'+m''\cdot n'',m'\cdot n''+m''\cdot n')$ and $0_\mathbb{Z} :=(0,0)$ and $1_\mathbb{Z} :=(1,0)$. $\mathbb{Q}$ is the set of all $m$ such that $m''>_\mathbb{Z} 0_\mathbb{Z}$. For all $p,q$ in $\mathbb{Q}$, $p=_\mathbb{Q} q\leftrightarrow p'\cdot_\mathbb{Z} q'' =_\mathbb{Z} p''\cdot_\mathbb{Z} q'$ and $p<_\mathbb{Q} q\leftrightarrow p'\cdot_\mathbb{Z} q'' <_\mathbb{Z} p''\cdot_\mathbb{Z} q'$ and $p\le_\mathbb{Q} q\leftrightarrow p'\cdot_\mathbb{Z} q'' \le_\mathbb{Z} p''\cdot_\mathbb{Z} q'$ and $p+_\mathbb{Q} q := (p'\cdot_\mathbb{Z} q'' +_\mathbb{Z} p''\cdot_\mathbb{Z} q',p''\cdot_\mathbb{Z} q'')$ and $p-_\mathbb{Q} q := (p'\cdot_\mathbb{Z} q'' -_\mathbb{Z} p''\cdot_\mathbb{Z} q',p''\cdot_\mathbb{Z} q'')$and $p\cdot_\mathbb{Q} q:=(p'\cdot_\mathbb{Z} q', p''\cdot_\mathbb{Z} q'')$  and  $0_\mathbb{Q} :=(0_\mathbb{Z}, 1_\mathbb{Z})$ and $1_\mathbb{Q} :=(1_\mathbb{Z}, 1_\mathbb{Z})$.  $\mathbb{S}$ is the set of (the code numbers of) the \textit{rational segments}, i.e. $\forall s[s\in \mathbb{S} \leftrightarrow  (s' \in \mathbb{Q} \;\wedge\;s''\in \mathbb{Q} \;\wedge\; s' <_\mathbb{Q} s'')]$. For every $s$ in $\mathbb{S}$,  $\mathit{length}_\mathbb{S}(s) := s'' -_\mathbb{Q} s'$. For all $s,t$  in $\mathbb{S}$,  $s \sqsubseteq_\mathbb{S} t\leftrightarrow t' \le_\mathbb{Q} s' \le_\mathbb{Q} s'' \le_\mathbb{Q} t''$ and: $s \sqsubset_\mathbb{S} t\leftrightarrow t' <_\mathbb{Q} s' <_\mathbb{Q} s'' <_\mathbb{Q} t''$ and $s <_\mathbb{S} t\leftrightarrow s'' <_\mathbb{Q} t'$, and $s \le_\mathbb{S} t\leftrightarrow s' <_\mathbb{Q} t''$, and $s \;  \#_\mathbb{S} \; t\leftrightarrow (s <_\mathbb{S}t\;\vee\;t <_\mathbb{S} s)$, and $s 
\approx_\mathbb{S} t\leftrightarrow (s\le_\mathbb{S}t\;\wedge\;t \le_\mathbb{S} s)$.
Note that $\forall s \in \mathbb{S}\forall t\in\mathbb{S}[s \approx_\mathbb{S} t\leftrightarrow \exists u \in \mathbb{S}[u \sqsubset_\mathbb{S} s\;\wedge\;u \sqsubset_\mathbb{S}t]]$.
For all $s$ in $\mathbb{S}$, $double_\mathbb{S}(s)$ is the element $t$ of $\mathbb{S}$ satisfying $\frac{t'+t''}{2}=\frac{s'+s''}{2}$ and $length_\mathbb{S}(t)=2\cdot_\mathbb{Q}length_\mathbb{S}(s)$.

 The set $\mathcal{R}$ of the \textit{real numbers} is the set of all $\alpha$ such that $\forall n[\alpha(n)\in\mathbb{S}\;\wedge\;\alpha(n+1) \sqsubset_\mathbb{S} \alpha(n)]$ and $\forall m\exists n[
\mathit{length}_\mathbb{S}(\alpha(n)) <_\mathbb{Q} \frac{1}{2^m}]$.  For all $\alpha, \beta$ in $\mathcal{R}$,  $\alpha <_\mathcal{R}\beta\leftrightarrow \exists n[\alpha(n) <_\mathbb{S} \beta(n)]$ and
$\alpha \le_\mathcal{R} \beta\leftrightarrow \neg(\beta <_\mathcal{R} \alpha)$, and $\alpha =_\mathcal{R} \beta\leftrightarrow (\alpha \le_\mathcal{R} \beta\;\wedge\; \beta \le_\mathcal{R} \alpha)$.
For all $\alpha, \beta$ in $\mathcal{R}$, $[\alpha, \beta) := \{\gamma \in \mathcal{R}\mid 0_\mathcal{R} \le_\mathcal{R} \gamma <_\mathcal{R} \beta \}$, and $[\alpha, \beta] := \{\gamma \in \mathcal{R}\mid 0_\mathcal{R} \le_\mathcal{R} \gamma \le_\mathcal{R} \beta \}$.
For every $p$ in $\mathbb{Q}$,  $p_\mathcal{R}$ is the element of $\mathcal{R}$ satisfying: for each $n$, $p_\mathcal{R}(n) = (p-_\mathbb{Q}\frac{1}{2^n}, p+_\mathbb{Q}\frac{1}{2^n})$.  For each $s$ in $\mathbb{S}$, we define $\overline s := [(s')_\mathcal{R}, (s'')_\mathcal{R}]$.

\subsection{Coverings}\label{SS:cov} Let $\mathcal{X}$ be a subset of $\mathcal{R}$ and let $C$ be a subset of $\mathbb{S}$. We define: $C$ \textit{covers} $\mathcal{X}$, notation: $Cov_C(\mathcal{X})$,  if and only if $\forall \alpha\in\mathcal{X}\exists n\exists  s\in C[\alpha(n) \sqsubset_\mathbb{S} s]$.

  Let $C$ be a finite subset of $\mathbb{S}$, and let $s$ be an element of $\mathbb{S}$. Note that $C$ covers $\overline s=[(s')_\mathcal{R},(s'')_\mathcal{R}]$ if and only if  $\exists n\exists u[\mathit{length}(u) = n\;\wedge\;\forall i<n[u(i) \in C]\;\wedge\;\bigl(u(0)\bigr)'<_\mathbb{Q} s'<_\mathbb{Q}\bigl(u(0)\bigr)''\;\wedge\;\forall i<n-1[u(i) \approx_\mathbb{S} u(i+1)]\;\wedge\;\bigl(u(n-1)\bigr)'<_\mathbb{Q} s''<_\mathbb{Q}\bigl(u(n-1)\bigr)'']$.
  We thus may decide,  if $C$ covers $\overline s$ or not. Also note that, if $C$ covers $\overline s$, then $\exists n\forall \gamma \in \overline s\forall \delta\in \overline s[|\gamma -_\mathcal{R}\delta|<_\mathcal{R} \frac{1}{2^n}\rightarrow  \exists t \in C\exists m \exists n[ \gamma(m)\sqsubset_\mathbb{S}t \; \wedge\; \delta(n) \sqsubset_\mathbb{S}t]]$.

 \subsection{Open subsets of $\mathcal{R}$}\label{SS:open} For every $\alpha$ we let $\mathcal{H}_\alpha$ be the set of all $\gamma$ in $\mathcal{R}$ such that $\exists s \in \mathbb{S}\exists n[ s\in D_\alpha\;\wedge\;\gamma(n) \sqsubset_\mathbb{S} s]$.  A subset $\mathcal{H}$ of $\mathcal{R}$ is \textit{open} if and only if, for some $\alpha$, $\mathcal{H}=\mathcal{H}_\alpha$. 
   For each $t$ in $\omega$ we let $\mathcal{H}_t$ be the set of all $\gamma$ in $\mathcal{R}$ such that, for some $\exists s\exists n[ s\in D_t\;\wedge\;\gamma(n) \sqsubset_\mathbb{S} s]$.
  Note that,  for every $\alpha$, $\mathcal{H}_\alpha = \bigcup\limits_{n \in \omega} \mathcal{H}_{\overline \alpha n}$.
 Also note that, for every  $\mathcal{X}\subseteq\mathcal{R}$,  $\forall \alpha[\mathcal{X} \subseteq\mathcal{H}_\alpha\leftrightarrow Cov_{\mathcal{X}}(D_\alpha)]$ and  $\forall t[\mathcal{X} \subseteq\mathcal{H}_t\leftrightarrow Cov_{\mathcal{X}}(D_t)]$.
 
For all $\mathcal{A}\subseteq\mathcal{R}$, for all $\alpha, \beta$ in $\mathcal{R}$ such that $\alpha\le_\mathcal{R}\beta$, $\mathcal{A}$ is  \emph{progressive in $[\alpha,\beta]$} if and only if $\forall \gamma \in [\alpha,\beta ][[0_\mathcal{R}, \gamma)\subseteq \mathcal{A}]\rightarrow \gamma \in \mathcal{A}]$.

\end{document}